\documentclass[a4paper]{article}
\usepackage{geometry}
\usepackage{amsmath,amsfonts,amssymb,amsthm,stmaryrd,mathrsfs,mathdots}
\usepackage[small]{diagrams}
\usepackage{hyperref}
\usepackage{bm}
\usepackage{tikz}
\usepackage{tikz-cd}
\tikzset{%
    symbol/.style={%
        draw=none,
        every to/.append style={%
            edge node={node [sloped, allow upside down, auto=false]{$#1$}}}
    }
}

\usepackage{sectsty}
\sectionfont{\sffamily}
\subsectionfont{\sffamily}
\subsubsectionfont{\sffamily}

\newtheoremstyle{myplain}
  {\topsep}   
  {\topsep}   
  {\itshape}  
  {0pt}       
  {\bfseries\sffamily} 
  {.}         
  {5pt plus 1pt minus 1pt} 
  {}          
  
\newtheoremstyle{mydefinition}
  {\topsep}   
  {\topsep}   
  {\normalfont}  
  {0pt}       
  {\bfseries\sffamily} 
  {.}         
  {5pt plus 1pt minus 1pt} 
  {}          
  
\newtheoremstyle{myremark}
  {0.5\topsep}   
  {0.5\topsep}   
  {\normalfont}  
  {0pt}       
  {\sffamily} 
  {.}         
  {5pt plus 1pt minus 1pt} 
  {}          

\theoremstyle{mydefinition}
\newtheorem{definition}{$\blacktriangleright$ Definition}[subsection]
\theoremstyle{myplain}
\newtheorem{theorem}[definition]{$\blacktriangleright$ Theorem}
\newtheorem{lemma}[definition]{$\blacktriangleright$ Lemma}
\newtheorem{corollary}[definition]{$\blacktriangleright$ Corollary}
\newtheorem{proposition}[definition]{$\blacktriangleright$ Proposition}
\theoremstyle{myremark}
\newtheorem*{remark}{$\blacktriangleright$ Remark}
\newtheorem*{notation}{$\blacktriangleright$ Notation}
\newtheorem{example}[definition]{$\blacktriangleright$ Example}

\newtheorem*{convention}{$\blacktriangleright$ Convention}

\newarrow{Corresponds}<--->
\newarrow{Dotscorresponds}<...>
\newarrow{Equals}=====
\newarrow{Into}C--->
\newarrow{Embed}>--->
\newarrow{DepTo}----{triangle}
\newarrow{DotsTo}....>
\newarrow{DotsDepTo}....{triangle}
\newarrow{dLine}=====

\begin{document}
\title{\Huge \bf Categories with Dependence and Semantics of Type Theories}
\author{\Large \bf Norihiro Yamada \\
{\tt norihiro1988@gmail.com} \\
University of Oxford
}

\maketitle

\begin{abstract}
In order to capture the \emph{categorical} counterpart of the path from simple type theories (STTs) to dependent type theories (DTTs) that \emph{faithfully} reflect syntactic phenomena, we introduce a generalization of cartesian closed categories (CCCs), called \emph{cartesian closed categories with dependence (CCCwDs)}. 
Indeed, CCCwDs are a categorical concept as they are defined in terms of universal properties that naturally generalize those of CCCs; also, CCCwDs have a direct counterpart of terms, and moreover their strict version induces categories with families that support strict $1$-, $\Pi$- and $\Sigma$-types, i.e., a `binder-free' formulation of the core part of DTTs.

As a main theorem, we prove that CCCwDs give a sound and complete interpretation of the ($\mathsf{1}$, $\mathsf{\Pi}$, $\mathsf{\Sigma}$)-fragment of Martin-L\"{o}f type theory (MLTT), a canonical representative of DTTs.
Moreover, we establish a 2-equivalence between the 2-category of \emph{contextual} (in a suitable sense similar to contextual categories by Cartmell) CCCwDs and the 2-category of MLTTs.
In addition, we show that the 2-category of \emph{constant} (in a suitable sense that abstractly captures simple types) contextual CCCwDs is 2-equivalent to the 2-category of contextual CCCs, which is well-known to be 2-equivalent to the 2-category of \emph{simply-typed $\lambda$-calculi} equipped with $\mathsf{1}$- and $\mathsf{\times}$-types, a canonical representative of STTs. 
In this manner, we show that the categorical path from constant contextual CCCwDs to contextual CCCwDs corresponds to the syntactic path from STTs to DTTs, and also clarify its relation with the interpretation of STTs in CCCs.  

As a `by-product', fundamental concepts in category theory, e.g., categories, isomorphisms, functors, natural transformations, limits, adjoints, etc. are naturally generalized.
Hence, the present work also sets out a general theory of `categories for dependence'.
\end{abstract}

\tableofcontents

\section{Introduction}
In mathematics, \emph{category theory} \cite{mac2013categories} studies algebra of \emph{morphisms}, which originated from the work in algebraic topology by Samuel Eilenberg and Saunders Mac Lane in 1940s.
Later, it has turned out to be extremely general, giving a synthetic formulation of various concepts in mathematics, physics, logic and computer science \cite{baez2010physics,maclane2012sheaves,pierce1991basic,lambek1988introduction}.

In logic and computer science, \emph{type theories} \cite{hindley1997basic,pierce2002types} refer to a particular kind of formal logical systems or programming languages, in which \emph{proofs} or \emph{terms} are always \emph{typed} in \emph{contexts}. 
Historically, the concept of \emph{types} was first introduced to his \emph{theory of types} by Bertrand Russell in 1900s to circumvent the well-known paradox in foundations of mathematics discovered by himself, and later types were also incorporated by Alonzo Church into his \emph{$\lambda$-calculus} in 1940s to save it from another yet similar paradox, which led to type theories in their modern form.

An important concept in type theories is \emph{dependent types} \cite{hofmann1997syntax}, which are types that may contain \emph{free variables}.
A type theory is \emph{dependent} if it has dependent types; otherwise it is (and its types are) \emph{simple}.
Under the \emph{Curry-Howard isomorphism} \cite{sorensen2006lectures}, the generalization of simple type theories (STTs) to dependent type theories (DTTs) corresponds in logic to the path from propositional logic to predicate logic, where simple types and dependent types correspond to propositions and predicates, respectively. 
For concreteness, we focus on \emph{simply-typed $\lambda$-calculi (STLCs)} \cite{church1940formulation} and \emph{Martin-L\"{o}f type theories (MLTTs)} \cite{martin1982constructive,martin1984intuitionistic,martin1998intuitionistic} in this article for they are canonical representatives of STTs and DTTs, respectively.

\emph{Categorical logic} \cite{pitts2001categorical,jacobs1999categorical,lambek1988introduction} is the study of connections between categories and type theories.
In categorical logic, syntax and semantics of type theories are both captured by categories with additional structures, and interpretations by functors preserving those structures. 
For syntax, such a categorical reformulation abstracts tedious syntactic formalisms (such as capture-free substitution) and enables one to focus on the abstract essence of type theories; also, neat categorical structures in a sense `justify' the corresponding syntactic concepts.
For semantics, on the other hand, categorical semantics provides a convenient framework to establish concrete semantics of type theories as in general it is simpler and subsumes more interpretations than the traditional set-theoretic semantics. 
Moreover, by \emph{theory-category correspondences} \cite{pitts2001categorical,taylor1999practical}, one may regard categories as a syntax-independent abstraction of type theories, or conversely, type theories as a symbolic presentation of (strict) categories.

Categorical semantics of STTs has been well-established, namely \emph{cartesian closed categories (CCCs)} \cite{pitts2001categorical,jacobs1999categorical,lambek1988introduction}, but it is not quite the case for DTTs. 
More specifically, abstract semantics of DTTs is typically either:
\begin{enumerate}
\item Categorical yet indirect (in the sense that there is no counterpart of terms); or 
\item Direct yet not categorical (often called \emph{algebraic} or \emph{equational}). 
\end{enumerate}
Note that CCCs are a categorical concept, and they give a direct semantics of STTs, in which terms are interpreted by morphisms. 

\begin{remark}
Categorical concepts are usually defined in terms of universal properties, and thus they are \emph{weak}, i.e., identified only up to isomorphisms, while type-theoretic concepts are \emph{strict}, i.e., identified on-the-nose. 
Therefore, strictly speaking, one must employ strict structures to give a semantics of type theories; for instance, a fibred category has a substitution as a part of the structure \cite{jacobs1999categorical}. 
However, as long as a semantics employs the strict version of weak concepts, it would be fair to regard the semantics as categorical. 
\end{remark}

In fact, categorical semantics of DTTs has mathematical elegance, capturing type-theoretic phenomena in terms of categorical concepts, but usually they do not directly correspond to the syntax, having no counterpart of terms in particular.
We regard this indirectness as a problem:
\begin{itemize}

\item It would be difficult to say that indirect semantics faithfully captures DTTs since terms are a primitive concept that plays a central role in the syntax; 

\item Direct semantics of DTTs would be interesting as an abstraction of \emph{type dependence} as it is a natural and fundamental phenomenon which arises not only in DTTs but also in a variety of mathematical instances, e.g., see the set-theoretic example below; and

\item It is harder to work with indirect semantics of DTTs since intuition from the syntax is not directly applicable.

\end{itemize}
For instance, in \emph{locally cartesian closed categories (LCCCs)} \cite{seely1984locally} and \emph{categories with attributes (CwAs)} \cite{cartmell1986generalized,moggi1991category,pitts2001categorical}, each term $\mathsf{\Gamma \vdash f : A}$ in a DTT is identified with the \emph{context morphism} \cite{hofmann1997syntax,pitts2001categorical} $(\mathit{id}_{\mathsf{\Gamma}}, \mathsf{f}) : \mathsf{\Gamma} \to \mathsf{\Gamma, x : A}$ because there is no direct counterpart of terms in these frameworks.
In the syntax, however, $\mathsf{f}$ and $(\mathit{id}_{\mathsf{\Gamma}}, \mathsf{f})$ are not identified, and the former is clearly more primitive than the latter.
In terms of the set-theoretic semantics, a context $\mathsf{\vdash \Gamma \ ctx}$, a type $\mathsf{\Gamma \vdash A \ type}$ and a term $\mathsf{\Gamma \vdash f : A}$ in a DTT should be, intuitively speaking, interpreted respectively as a set $\Gamma$, an indexed set $A = (A_\gamma)_{\gamma \in \Gamma}$ and a function $f : \Gamma \to \bigcup_{\gamma \in \Gamma}A_\gamma$ such that $f(\gamma) \in A_\gamma$ for all $\gamma \in \Gamma$, but those categorical frameworks, specialized to the set-theoretic semantics, interpret the term $\mathsf{\Gamma \vdash f : A}$ as the function $\langle \mathit{id}_\Gamma, f \rangle : (\gamma \in \Gamma) \mapsto (\gamma, f(\gamma)) \in \Gamma \times \bigcup_{\gamma \in \Gamma}A_\gamma$.
This point holds also for other categorical semantics of DTTs such as \emph{comprehension categories} \cite{jacobs1999categorical}, \emph{display map categories} \cite{lamarche1989simple,taylor1987recursive,hyland1989theory,taylor1999practical} and \emph{indexed categories} \cite{curien1989alpha,obtulowicz1989categorical}.

In contrast, algebraic semantics of DTTs such as \emph{categories with families (CwFs)} \cite{dybjer1996internal,hofmann1997syntax} is closer to the syntax, having a direct counterpart of terms, but it appears mathematically \emph{ad-hoc} because it just employs an equational characterization that is not the strict version of weak concepts, i.e., it is not quite categorical.
Indeed, this point is one of the main motivations for the recent work on \emph{natural models of homotopy type theory (NMs)} \cite{awodey2016natural} by Awodey; it gives a categorical reformulation of CwFs.
However, it seems difficult to take the resulting categorical structure as a generalization of CCCs; at least, its relation with the standard interpretation of STLCs in CCCs remains unclear. 
This point is unsatisfactory if one takes account of the rather simple relation between STTs and DTTs.

To summarize, we have not yet found a categorical and direct counterpart of the path from STTs to DTTs.
We are concerned with such a counterpart for mathematical as well as practical points of view.
From the mathematical standpoint, such a categorical structure is of interest in its own right because it would accurately capture the abstract essence of dependent types, which are a natural and fundamental concept not only in logic and computation but also (as we shall see) in many other mathematical instances, e.g., in sets as illustrated above.

From the practical viewpoint, a direct correspondence with the syntax would make it much more convenient as a tool to establish concrete semantics of DTTs since each semantic concept has a direct type-theoretic counterpart, which delivers useful intuition.
If a direct semantics is in addition defined categorically, then in general it has fewer structures and axioms than an algebraic semantics so that it is simpler to verify that a concrete structure forms a semantics of DTTs through the former than the latter.

Motivated by these points, we introduce \emph{cartesian closed categories with dependence (CCCwDs)} in order to capture DTTs in a mathematically natural as well as true-to-syntax fashion. 
Roughly, \emph{categories with dependence (CwDs)} are categories equipped with direct counterparts of types and terms, called \emph{dependent (D-) objects} and \emph{dependent (D-) morphisms}, respectively, that satisfy certain axioms, where categories constitute a special case.
A CwD is \emph{semi-cartesian} if it has a terminal object and  \emph{semi-dependent pair spaces (semi-$\Sigma$-spaces)}, which correspond respectively to the empty context and context extensions in DTTs. Let us note that \emph{strict} semi-cartesian CwDs (SCCwDs) coincide with CwFs, and thus they are nothing new; our contribution is in defining a natural \emph{cartesian closed} structure on SCCwDs.
An SCCwD is \emph{cartesian} if it has a \emph{unit D-object} and \emph{dependent pair spaces ($\Sigma$-spaces)}, which are generalizations of a terminal object and binary products in the sense that they are defined in terms of universal properties that naturally generalize those of a terminal object and binary products, respectively; in fact, a unit D-object and $\Sigma$-spaces form instances of generalized limits, called \emph{dependent (D-) limits}.
An SCCwD is \emph{closed} if it has \emph{dependent map spaces ($\Pi$-spaces)}, which are, similarly to $\Sigma$-spaces, a generalization of exponentials.
As a mathematical `justification' of these structures as a natural generalization of CCCs, we prove that unit D-objects, $\Sigma$- and $\Pi$-spaces satisfy certain expected properties, e.g., they are unique up to isomorphisms, $\Sigma$- and $\Pi$-spaces are respectively left and right adjoints to each other (up to projections), etc., and give several non-syntactic instances of CCCwDs.

We then show that CCCwDs give a sound and complete semantics of the ($\mathsf{1}$, $\mathsf{\Pi}$, $\mathsf{\Sigma}$)-fragment of MLTT.
Moreover, we establish a 2-equivalence between the 2-category $\mathsf{Ctx}\mathbb{CCC_D}$ of \emph{contextual} (in a suitable sense similar to \emph{contextual categories} by Cartmell) CCCwDs (CtxCCCwDs) and the 2-category $\mathbb{MLTT}$ of MLTTs, which can be seen as a dependent version of the theory-category correspondence between STLCs and CCCs \cite{pitts2001categorical,crole1993categories}.
In addition, we verify that the 2-category $\mathsf{ConCtx}\mathbb{CCC_D}$ of \emph{constant} (in a suitable sense that abstractly captures simple types) CtxCCCwDs is 2-equivalent to the 2-category $\mathsf{Ctx}\mathbb{CCC}$ of contextual CCCs (CtxCCCs). 
In this manner, we demonstrate that the categorical path from constant CtxCCCwDs to CtxCCCwDs corresponds to the syntactic path from STTs to DTTs, and also clarify its relation with the standard interpretation of STTs in CCCs.  
These results may be summarized schematically as: 
\begin{diagram}
&&&& \mathsf{Ctx}\mathbb{CCC_D} && \simeq && \mathbb{MLTT} \\
&&&& \uInto &&&& \uInto \\
\mathsf{Ctx}\mathbb{CCC} && \simeq && \mathsf{ConCtx}\mathbb{CCC_D} && \simeq && \mathbb{STLC}
\end{diagram}
where the conventional 2-equivalence between $\mathsf{Ctx}\mathbb{CCC}$ and the 2-category $\mathbb{STLC}$ of STLCs is recovered by composing the two 2-equivalences on the lower side in the diagram.

Notably, CCCwDs interpret MLTTs by giving rise to CwFs that support strict $1$-, $\Pi$- and $\Sigma$-types.
This construction is rather canonical, and the term models of MLTTs form such CwFs via CtxCCCwDs; in this sense, CCCwDs are a \emph{refinement} of CwFs.
Also, constant CtxCCCwDs refine CtxCCCs for the former more faithfully captures STLCs than the latter.

Finally, as a `by-product', fundamental concepts in category theory such as categories, isomorphisms, functors, natural transformations, limits and adjoints are naturally generalized.
Thus, the present work also sets out a general theory of `categories of dependence' though we mostly focus on what is relevant to the main results to keep the paper of a reasonable length, leaving a more thorough treatment as future work.

\paragraph{Related Work.}
NMs by Awodey \cite{awodey2016natural} are another categorical and direct semantics of MLTTs, and thus most relevant to the present work.
The mathematical structures of CCCwDs are, however, a priori very different from those of NMs, in particular the relation between NMs and CCCs is not immediately clear, though we leave a detailed comparison of the two as future work.
Also, to the best of our knowledge, the theory-category correspondences between CtxCCCwDs and MLTTs as well as between constant CtxCCCwDs and CtxCCCs are new.

\paragraph{Plan of the Paper.}
The rest of the paper proceeds as follows.
We first review the syntax of MLTTs in Section~\ref{MLTT} and the interpretation of MLTTs in CwFs in Section~\ref{CwFs}, where we also present the syntax of STLCs and the interpretation of them in CCCs.
Then, we introduce CwDs and morphisms between them in Section~\ref{CwDs}, and SCCwDs and morphisms between them in Section~\ref{SCCwDs}.
We then consider generalized terminal objects and binary products in SCCwDs and show that they interpret $\mathsf{1}$- and $\mathsf{\Sigma}$-types in Section~\ref{CCwDs}, and we present generalized exponentials and prove that they interpret $\mathsf{\Pi}$-types in Section~\ref{CCCwDs}.
Along with these developments, we give a basic theory of CwDs, generalizing some basic part of category theory.
The next two sections are devoted to certain 2-equivalences: Section~\ref{Equivalence} establishes a 2-equivalence between $\mathsf{ConCtx}\mathbb{CCC_D}$ and $\mathsf{Ctx}\mathbb{CCC}$, and Section~\ref{TheoryCategoryCorrespondenceBetweenMLTTsAndCCCwDs} gives a theory-category correspondence between MLTTs and  CtxCCCwDs.
Finally, Section~\ref{ConclusionAndFutureWork} draws a conclusion and proposes future work.

\section{Preliminaries}
As a preparation, we review the syntax of MLTTs (in Section~\ref{MLTT}) and the interpretation of them in CwFs (in Section~\ref{CwFs}).
If the reader is already familiar with these concepts, then she or he is encouraged to skip the present section and come back later in the `call-by-need' fashion. 

There are, however, mainly four reasons to present these concepts explicitly:
\begin{enumerate}

\item To define the term models (or the classifying CwFs) of MLTTs;

\item To show that strict CCCwDs are a refinement of CwFs that support $1$-, $\Sigma$- and $\Pi$-types;

\item To prove a theory-category correspondence between MLTTs and contextual CCCwDs; and

\item To explain a gap between the CCC- and CwF-semantics of STLCs.

\end{enumerate}

\subsection{The Syntax of Martin-L\"{o}f Type Theories}
\label{MLTT}
We first review the syntax of MLTTs, following the presentations of \cite{hofmann1997syntax,pitts2001categorical}, and also present the syntax of STLCs as sub-theories of MLTTs. 
Note that we mainly focus on the syntax; see, e.g., \cite{nordstrom1990programming,martin1984intuitionistic} for a general introduction to MLTTs.

A \emph{\bfseries Martin-L\"{o}f type theory (MLTT)} \cite{martin1975intuitionistic,martin1984intuitionistic,martin1998intuitionistic} is a formal logical system similar to \emph{natural deduction} \cite{gentzen1935untersuchungen,troelstra2000basic} except that vertices of a derivation tree are \emph{\bfseries judgements}, for which we usually write $\mathcal{J}$ possibly with subscripts. There are the following six kinds of judgements (followed by their intended meanings):
\begin{itemize}

\item $\mathsf{\vdash \Gamma \ ctx}$ ($\mathsf{\Gamma}$ is a \emph{\bfseries context}); 

\item $\mathsf{\Gamma \vdash A \ type}$ ($\mathsf{A}$ is a \emph{\bfseries type} in the context $\mathsf{\Gamma}$);

\item $\mathsf{\Gamma \vdash a : A}$ ($\mathsf{a}$ is a \emph{\bfseries term} (or \emph{\bfseries program}) of type $\mathsf{A}$ in the context $\mathsf{\Gamma}$);

\item $\mathsf{\vdash \Gamma = \Delta \ ctx}$ ($\mathsf{\Gamma}$ and $\mathsf{\Delta}$ are \emph{\bfseries (judgmentally) equal} contexts);

\item $\mathsf{\Gamma \vdash A = B \ type}$ ($\mathsf{A}$ and $\mathsf{B}$ are \emph{\bfseries (judgmentally) equal} types in the context $\mathsf{\Gamma}$); and

\item $\mathsf{\Gamma \vdash a = a' : A}$ ($\mathsf{a}$ and $\mathsf{a'}$ are \emph{\bfseries (judgmentally) equal} terms of type $\mathsf{A}$ in the context $\mathsf{\Gamma}$).

\end{itemize}
That is, an MLTT consists of \emph{axioms} $\frac{}{ \ \mathcal{J} \ }$ and \emph{(inference) rules} $\frac{ \ \mathcal{J}_1 \ \mathcal{J}_2 \dots \mathcal{J}_k \ }{ \mathcal{J}_0 }$, which are to make a \emph{conclusion} from \emph{hypotheses} by constructing a derivation tree exactly as natural deduction does.

As often expressed as `an MLTT \emph{internalizes} the Curry-Howard isomorphism \cite{sorensen2006lectures}', its contexts, types and terms represent respectively \emph{assumptions} (or \emph{premises}), \emph{formulas} and \emph{proofs} in logic as well.
Therefore, e.g., the judgement $\mathsf{\Gamma \vdash a : A}$ can be read as `the formula $\mathsf{A}$ has a proof $\mathsf{a}$ under the assumption $\mathsf{\Gamma}$', and so on.

\begin{convention}
We assume a countably infinite set $\mathscr{V}$ of \emph{variables}, written $\mathsf{x}$, $\mathsf{y}$, $\mathsf{z}$, etc. possibly with subscripts.
We write $\mathscr{FV}(\mathsf{E})$ for the set of all \emph{free variables} occurring in an expression $\mathsf{E}$.
\end{convention}

\begin{notation}
Greek capital letters $\mathsf{\Gamma}$, $\mathsf{\Delta}$, $\mathsf{\Theta}$, etc. range over contexts, capital letters \textsf{A}, \textsf{B}, \textsf{C}, etc. over types, and lower-case letters \textsf{a}, \textsf{b}, \textsf{c}, etc. over terms. Strictly speaking, they are \emph{equivalence classes of symbols with respect to $\alpha$-equivalence} $\mathsf{\equiv}$ \cite{hankin1994lambda} as usual.
\end{notation}

Each type construction in an MLTT is defined by its \emph{\bfseries formation}, \emph{\bfseries introduction}, \emph{\bfseries elimination} and \emph{\bfseries computation} rules.
The formation rule stipulates how to form the type, and the introduction rule defines \emph{canonical terms} of the type, which in turn defines terms of the type (including non-canonical ones).
Then, the elimination and computation rules describe respectively how to consume terms of the type and the result of such a consumption (in the form of an equation), both of which are easily justified by the introduction rule.

Below, following the presentation of \cite{hofmann1997syntax}, we first recall the MLTT with $\mathsf{1}$-, $\mathsf{\Pi}$- and $\mathsf{\Sigma}$-types in the \emph{strict} sense, i.e., with \emph{uniqueness rules}, for which we write $\mathsf{MLTT(1, \Pi, \Sigma)}$, in Sections~\ref{Contexts}-\ref{DependentPairTypes}.
It is the minimal MLTT for us, and it corresponds to the dependently-typed version of the STLC equipped with $\mathsf{1}$- and $\mathsf{\times}$-types.
We then introduce a general class of MLTTs based on the framework of \emph{dependently-typed algebraic theories} \cite{pitts2001categorical} (or \emph{generalized algebraic theories} \cite{cartmell1986generalized}) in Sections~\ref{ContextMorphismsAndGeneralizedSubstitutions} and \ref{MLTTsAsAlgebraicTheories}.
Strictly speaking, this class is more general than what one usually considers as the class of MLTTs because additional types are not necessarily defined by the four kinds of rules mentioned above.
We need this generality for the \emph{theory-category correspondence} between the 2-categories of MLTTs and of contextual CCCwDs in Section~\ref{TheoryCategoryCorrespondenceBetweenMLTTsAndCCCwDs}.

\subsubsection{Contexts}
\label{Contexts}
A \emph{\bfseries context} is a finite sequence $\mathsf{x_1 : A_1, x_2 : A_2, \dots, x_n : A_n}$ of pairs of a variable $\mathsf{x_i}$ and a type $\mathsf{A_i}$ such that the variables $\mathsf{x_1, x_2, \dots, x_n}$ are pair-wise distinct. 
We write $\mathsf{\diamondsuit}$ for the \emph{empty context}, i.e., the empty sequence $\bm{\epsilon}$; we usually omit $\mathsf{\diamondsuit}$ when it occurs on the left-hand side of the turnstile $\vdash$. 

We have the following axiom and rules for contexts:
\begin{align*}
&(\textsc{Ctx-Emp}) \frac{}{ \ \mathsf{\vdash \diamondsuit \ ctx} \ } \ \ \ (\textsc{Ctx-Ext}) \frac{ \ \mathsf{\Gamma \vdash A \ type} \ \ \ \mathsf{x} \in \mathscr{V} \setminus \mathscr{FV}(\mathsf{\Gamma}) \ }{ \ \mathsf{\vdash \Gamma, x : A \ ctx} \ } \\ 
&(\textsc{Ctx-ExtEq}) \frac{ \ \mathsf{\vdash \Gamma = \Delta \ ctx \ \ \ \Gamma \vdash A = B \ type} \ \ \ \mathsf{x} \in \mathscr{V} \setminus \mathscr{FV}(\mathsf{\Gamma}) \ \ \ \mathsf{y} \in \mathscr{V} \setminus \mathscr{FV}(\mathsf{\Delta}) \ }{ \ \mathsf{\vdash \Gamma, x : A = \Delta, y : B \ ctx\ } \ } 
\end{align*}

The rules Ctx-Emp and Ctx-Ext determine that contexts are exactly finite sequences of pairs of a variable and a type.
The rule Ctx-ExtEq is a \emph{congruence rule}, i.e., it states that judgmental equality is preserved under the rule Ctx-Ext, i.e., the `context extension'. 

\begin{convention}
We will henceforth skip writing down congruence rules for other constructions.
\end{convention}

\subsubsection{Structural Rules}
Next, we collect the inference rules for all types as \emph{\bfseries structural rules}:
\begin{align*}
&(\textsc{Var}) \frac{ \ \mathsf{\vdash x_1 : A_1, x_2 : A_2, \dots, x_n : A_n \ ctx} \ \ j \in \{ 1, 2, \dots, n \} \ }{ \ \mathsf{x_1 : A_1, x_2 : A_2, \dots, x_n : A_n \vdash x_j : A_j} \ } \ \ \ (\textsc{Ct-EqRefl})\frac{ \ \mathsf{\vdash \Gamma \ ctx} \ }{ \ \mathsf{\vdash \Gamma = \Gamma \ ctx} \ } \\ 
&(\textsc{Ct-EqSym})\frac{ \ \mathsf{\vdash \Gamma = \Delta \ ctx} \ }{ \ \mathsf{\vdash \Delta = \Gamma \ ctx} \ } \ \ \ (\textsc{Ct-EqTrans}) \frac{ \ \mathsf{\vdash \Gamma = \Delta \ ctx} \ \ \mathsf{\vdash \Delta = \Theta \ ctx} \ }{ \ \mathsf{\vdash \Gamma = \Theta \ ctx} \ }\\
&(\textsc{Ty-EqRefl})\frac{ \ \mathsf{\Gamma \vdash A \ type} \ }{ \ \mathsf{\Gamma \vdash A = A \ type} \ } \ \ \ (\textsc{Ty-EqSym})\frac{ \ \mathsf{\Gamma \vdash A = B \ type} \ }{ \ \mathsf{\Gamma \vdash B = A \ type} \ } \\ 
&(\textsc{Ty-EqTrans})\frac{ \ \mathsf{\Gamma \vdash A = B \ type} \ \ \mathsf{\Gamma \vdash B = C \ type} \ }{ \ \mathsf{\Gamma \vdash A = C \ type} \ } \ \ \ (\textsc{Tm-EqRefl})\frac{ \ \mathsf{\Gamma \vdash a : A} \ }{ \ \mathsf{\Gamma \vdash a = a : A} \ } \\ 
&(\textsc{Tm-EqSym})\frac{ \ \mathsf{\Gamma \vdash a = a' : A} \ }{ \ \mathsf{\Gamma \vdash a' = a : A} \ } \ \ \ (\textsc{Tm-EqTrans})\frac{ \ \mathsf{\Gamma \vdash a = a' : A} \ \ \mathsf{\Gamma \vdash a' = a'' : A} \ }{ \ \mathsf{\Gamma \vdash a = a'' : A} \ } \\
&(\textsc{Ty-Con})\frac{ \ \mathsf{\vdash \Gamma = \Delta \ ctx \ \ \Gamma \vdash A \ type} \ }{ \ \mathsf{\Delta \vdash A \ type} \ } \ \ \ (\textsc{Tm-Con})\frac{ \ \mathsf{\Gamma \vdash a : A \ \ \vdash \Gamma = \Delta \ ctx \ \ \Gamma \vdash A = B \ type} \ }{ \ \mathsf{\Delta \vdash a : B} \ } 
\end{align*}

The rule Var states the reasonable idea that we may give an element $\mathsf{x_j : A_j}$ if it occurs in the context just by `copy-catting' it.
The next nine rules stipulate that every judgmental equality is an equivalence relation.
The rules Ty-Conv and Tm-Conv formalize the idea that judgements should be preserved under the exchange of judgmentally equal contexts and/or types. 

The following \emph{\bfseries weakening} and \emph{\bfseries substitution} rules are \emph{admissible} (or \emph{derivable}) in MLTTs, but it is convenient to present them explicitly:
\begin{align*}
&(\textsc{Weak})\frac{ \ \mathsf{\Gamma, \Delta \vdash \mathsf{J}} \ \ \ \mathsf{\Gamma \vdash A \ type_i} \ \ \ \mathsf{x} \in \mathscr{V} \setminus \mathscr{FV}(\mathsf{\Gamma, \Delta}) \ }{ \ \mathsf{\Gamma, x : A, \Delta \vdash \mathsf{J}} \ } \\ 
&(\textsc{Subst})\frac{ \ \mathsf{\Gamma, x : A, \Delta \vdash \mathsf{J}} \ \ \ \mathsf{\Gamma \vdash a : A} \ }{ \ \mathsf{\Gamma, \Delta[a/x] \vdash \mathsf{J}[a/x]} \ }
\end{align*}
where $\mathsf{\mathsf{J}[a/x]}$ (resp. $\mathsf{\Delta[a/x]}$) denotes the \emph{capture-free substitution} \cite{hankin1994lambda} of $\mathsf{a}$ for $\mathsf{x}$ in $\mathsf{J}$\footnote{Here, $\mathsf{J}$ denotes the righthand side of an arbitrary judgement.} (resp. $\Delta$).

\subsubsection{Unit Type}
\label{UnitType}
We proceed to introduce specific type constructions. 
Let us begin with the simplest type, called the \emph{\bfseries unit type} (or the \emph{\bfseries $\bm{\mathsf{1}}$-type}) $\mathsf{1}$ in the \emph{strict} sense, which is the type that has just one canonical term $\mathsf{\star}$.
Thus, from the logical point of view, it represents the simplest true formula.

Its rules are the following:
\begin{align*}
&(\textsc{$\mathsf{1}$-Form})\frac{ \ \mathsf{\vdash \Gamma \ ctx} \ }{ \ \mathsf{\Gamma \vdash 1 \ type} \ } \ \ \ (\textsc{$\mathsf{1}$-Intro})\frac{ \ \mathsf{\vdash \Gamma \ ctx} \ }{ \ \mathsf{\Gamma \vdash \star : \bm{1} } \ } \ \ \  (\textsc{$\mathsf{1}$-Uniq})\frac{ \ \mathsf{\Gamma \vdash t : 1} \ }{ \ \mathsf{\Gamma \vdash t = \star : 1 } \ } \\ 
&(\textsc{$\mathsf{1}$-Elim})\frac{ \ \mathsf{\Gamma, x : \bm{1} \vdash P \ type \ \ \ \Gamma \vdash p : P[\star / x] \ \ \ \Gamma \vdash t : \bm{1} } \ }{ \ \mathsf{\Gamma \vdash R^{\bm{1}}(P, p, t) : P[t/x]} \ } \\
&(\textsc{$\mathsf{1}$-Comp})\frac{ \ \mathsf{\Gamma, x : \bm{1} \vdash P \ type \ \ \ \Gamma \vdash p : P[\star/x] } \ }{ \ \mathsf{\Gamma \vdash R^{\bm{1}}(P, p, \star) = p : P[\star/x]} \ }
\end{align*}

The formation rule $\mathsf{1}$-Form states that $\mathsf{1}$ is an \emph{atomic} type in the sense that we may form it without assuming any other types.
The introduction rule $\mathsf{1}$-Intro defines that it has just one canonical term, viz., $\mathsf{\star}$.
Thus, the uniqueness rule $\mathsf{1}$-Uniq should make sense, from which the remaining rules $\mathsf{1}$-Elim and $\mathsf{1}$-Comp immediately follow if we define $\mathsf{R^{1}(P, p, t) \stackrel{\mathrm{df. }}{\equiv} p}$.

\begin{convention}
Thus, we henceforth omit the last two rules.
\end{convention}

\subsubsection{Dependent Function Types}
\label{DependentFunctionTypes}
Let us introduce a central non-atomic type construction, called \emph{\bfseries dependent function types} (or \emph{\bfseries $\bm{\mathsf{\Pi}}$-types}) in the \emph{strict} sense.
The $\mathsf{\Pi}$-type $\mathsf{\Pi_{x:A}B}$ is intended to represent the space of \emph{dependent functions} from $\mathsf{A}$ to $\mathsf{B}$, and thus it is a generalization of the \emph{function type} $\mathsf{A \Rightarrow B}$ in STLCs.

The rules of $\mathsf{\Pi}$-types are the following:
\begin{align*}
&(\textsc{$\mathsf{\Pi}$-Form})\frac{ \ \mathsf{\Gamma \vdash A \ type} \ \ \ \mathsf{\Gamma, x : A \vdash B \ type} \ }{ \ \mathsf{\Gamma \vdash \Pi_{x : A}B \ type}  \ }  \ \ \ (\textsc{$\mathsf{\Pi}$-Intro})\frac{ \ \mathsf{\Gamma, x : A \vdash b : B} \ }{ \ \mathsf{\Gamma \vdash \lambda x^A . b : \Pi_{x : A} B}  \ } \\
&(\textsc{$\Pi$-Elim})\frac{ \ \mathsf{\Gamma \vdash f : \Pi_{x : A}B} \ \ \ \mathsf{\Gamma \vdash a : A} \ }{ \ \mathsf{\Gamma \vdash f(a) : B[a/x]}  \ } \ \ \ (\textsc{$\mathsf{\Pi}$-Comp})\frac{ \ \mathsf{\Gamma, x : A \vdash b : B} \ \ \ \mathsf{\Gamma \vdash a : A} \ }{ \ \mathsf{\Gamma \vdash (\lambda x^A . b)(a) = b[a/x] : B[a/x]}  \ } \\
&(\textsc{$\mathsf{\Pi}$-Uniq})\frac{ \ \mathsf{\Gamma \vdash f : \Pi_{x : A}B} \ \ \ \mathsf{x} \not \in \mathscr{FV}(\mathsf{\Gamma}) \ }{ \ \mathsf{\Gamma \vdash \lambda x^A . f(x) = f : \Pi_{x : A} B} \ }  
\end{align*}

\begin{notation}
We often omit the superscript $\mathsf{A}$ on $\mathsf{\lambda x^A . b}$ and write $\mathsf{A \Rightarrow B}$ for $\mathsf{\Pi_{x:A}B}$ if $\mathsf{x} \not \in \mathscr{FV}(\mathsf{B})$.
\end{notation}

The formation rule $\mathsf{\Pi}$-Form states that we may form the $\mathsf{\Pi}$-type $\mathsf{\Pi_{x:A}B}$ from types $\mathsf{A}$ and $\mathsf{B}$, where $\mathsf{B}$ may depend on $\mathsf{A}$.
The introduction rule $\mathsf{\Pi}$-Intro defines how to construct the canonical terms of $\mathsf{\Pi_{x:A}B}$, viz., \emph{$\lambda$-abstractions}; it is the usual way of defining a function $\mathsf{f}$ from $\mathsf{A}$ to $\mathsf{B}$, i.e., to specify its output $\mathsf{f(a) : B[a/x]}$ on each input $\mathsf{a : A}$.
Then, the elimination and computation rules $\mathsf{\Pi}$-Elim and $\mathsf{\Pi}$-Comp should make sense by the introduction rule.
Finally, the uniqueness rule $\mathsf{\Pi}$-Uniq stipulates that terms of $\mathsf{\Pi}$-types are only canonical ones.

\subsubsection{Dependent Pair Types}
\label{DependentPairTypes}
Another important non-atomic type construction is \emph{\bfseries dependent sum types} (or \emph{\bfseries $\bm{\mathsf{\Sigma}}$-types}) in the \emph{strict} sense. 
The $\mathsf{\Sigma}$-type $\mathsf{\Sigma_{x:A}B}$ represents the spaces of \emph{dependent pairs} of $\mathsf{a:A}$ and $\mathsf{b:B[a/x]}$, and thus it is a generalization of the \emph{product type} $\mathsf{A \times B}$ in STLCs.

The rules of $\mathsf{\Sigma}$-types are the following:
\begin{align*}
&(\textsc{$\mathsf{\Sigma}$-Form})\frac{ \ \mathsf{\Gamma \vdash A \ type} \ \ \ \mathsf{\Gamma, x : A \vdash B \ type} \ }{ \ \mathsf{\Gamma \vdash \Sigma_{x : A}B \ type}  \ } \\ 
&(\textsc{$\mathsf{\Sigma}$-Intro})\frac{ \ \mathsf{\Gamma, x : A \vdash B \ type \ \ \ \Gamma \vdash a : A \ \ \ \Gamma \vdash b : B[a/x]} \ }{ \ \mathsf{\Gamma \vdash \langle a, b \rangle : \Sigma_{x : A} B}  \ } \\
&(\textsc{$\mathsf{\Sigma}$-Elim})\frac{ \ \mathsf{\Gamma, z : \Sigma_{x : A}B \vdash C \ type \ \ \ \Gamma, x : A, y : B \vdash g : C[\langle x, y \rangle/z] \ \ \ \Gamma \vdash p : \Sigma_{x : A}B} \ }{ \ \mathsf{\Gamma \vdash R^{\Sigma}_{[z : \Sigma_{x : A}B]C}([x:A, y:B]g, p) : C[p/z]} \ } \\
&(\textsc{$\mathsf{\Sigma}$-Comp})\frac{ \ \mathsf{\Gamma, z : \Sigma_{x : A}B \vdash C \ type \ \ \ \Gamma, x : A, y : B \vdash g : C[\langle x, y \rangle/z] \ \ \ \Gamma \vdash a : A \ \ \ \Gamma \vdash b : B[a/x]} \ }{ \ \mathsf{\Gamma \vdash R^{\Sigma}_{[z : \Sigma_{x : A}B]C}([x:A, y:B]g, \langle a, b \rangle) = g[a/x, b/y] : C[\langle a, b \rangle/z]} \ } \\
&(\textsc{$\mathsf{\Sigma}$-Uniq})\frac{ \ \mathsf{\Gamma \vdash p : \Sigma_{x:A}B} \ }{ \ \mathsf{\Gamma \vdash \langle \pi^{A, B}_1(p), \pi^{A, B}_2(p) \rangle = p : \Sigma_{x:A}B} \ } 
\end{align*}
where 
\begin{align*}
&\mathsf{\Gamma \vdash \pi^{A, B}_1(p) \stackrel{\mathrm{df. }}{\equiv} R^\Sigma_{[z : \Sigma_{x : A}B]A}([x:A, y:B]x, p) : A}; \\
&\mathsf{\Gamma \vdash \pi^{A, B}_2(p) \stackrel{\mathrm{df. }}{\equiv} R^\Sigma_{[z : \Sigma_{x : A}B]B[\pi^{A, B}_1(z)/x]}([x:A, y:B]y, p]) : B[\pi^{A, B}_1(p)/x]}
\end{align*}
are \emph{projections} constructed by \textsc{$\mathsf{\Sigma}$-Elim}.

\begin{notation}
We usually abbreviate $\mathsf{[z : \Sigma_{x : A}B]C}$, $\mathsf{[x:A, y:B]g}$ and $\mathsf{\pi^{A, B}_i(p)}$ respectively as $\mathsf{C}$, $\mathsf{g}$ and $\mathsf{\pi_i(p)}$ if it does not bring any confusion.
We often write $\mathsf{A \times B}$ for $\mathsf{\Sigma_{x:A}B}$ if $\mathsf{x} \not \in \mathscr{FV}(\mathsf{B})$.
\end{notation}

The formation rule $\mathsf{\Sigma}$-Form is the same as the case of $\mathsf{\Pi}$-types.
The introduction rule $\mathsf{\Sigma}$-Intro specifies that canonical terms of a $\mathsf{\Sigma}$-type $\mathsf{\Sigma_{x : A} B}$ are pairs $\mathsf{\langle a, b \rangle : \Sigma_{x:A}B}$ of terms $\mathsf{a : A}$ and $\mathsf{b : B[a/x]}$; they are in particular \emph{dependent} pairs for the type $\mathsf{B[a/x]}$ of the second component $\mathsf{b}$ depends on the first component $\mathsf{a : A}$.
Again, the elimination and computation rules $\mathsf{\Sigma}$-Elim and $\mathsf{\Sigma}$-Comp should make sense by the introduction rule.
Finally, the uniqueness rule $\mathsf{\Sigma}$-Uniq stipulates that terms of $\mathsf{\Sigma}$-types are only canonical ones, i.e., dependent pairs.

\subsubsection{Context Morphisms and Generalized Substitution}
\label{ContextMorphismsAndGeneralizedSubstitutions}
We have presented the theory $\mathsf{MLTT(1, \Pi, \Sigma)}$.
Next, let us review a \emph{derived} concept in the syntax: A \emph{\bfseries context morphism} \cite{hofmann1997syntax,pitts2001categorical} from a context $\mathsf{\vdash \Delta \ ctx}$ to another $\mathsf{\vdash \Gamma \ ctx}$ such that $\mathscr{FV}(\mathsf{\Delta}) \cap \mathscr{FV}(\mathsf{\Gamma}) = \emptyset$, where we assume without loss of generality $\mathsf{\Gamma \equiv x_1 : A_1, x_2 : A_2, \dots, x_n : A_n \ ctx}$, is a finite sequence $\bm{\mathsf{g}} \equiv  (\mathsf{g_1}, \mathsf{g_2}, \dots, \mathsf{g_n}) : \mathsf{\Delta} \to \mathsf{\Gamma}$ of terms 
\begin{align*}
\mathsf{\Delta} &\vdash \mathsf{g_1 : A_1} \\
\mathsf{\Delta} &\vdash \mathsf{g_2 : A_2[g_1/x_1]} \\
&\vdots \\
\mathsf{\Delta} &\vdash \mathsf{g_n : A_n[g_1/x_1, g_2/x_2, \dots, g_{n-1}/x_{n-1}]}.
\end{align*}

In addition, we say that parallel context morphisms $\mathsf{\bm{\mathsf{g, g'}} : \Delta \to \Gamma}$ are \emph{\bfseries (judgmentally) equal} and write $\mathsf{\bm{\mathsf{g}} = \bm{\mathsf{g'}} : \Delta \to \Gamma}$ if so are their corresponding component terms, i.e.,
\begin{align*}
\mathsf{\Delta} &\vdash \mathsf{g_1 = g'_1 : A_1} \\
\mathsf{\Delta} &\vdash \mathsf{g_2 = g'_2 : A_2[g_1/x_1]} \\
&\vdots \\
\mathsf{\Delta} &\vdash \mathsf{g_n = g'_n : A_n[g_1/x_1, g_2/x_2, \dots, g_{n-1}/x_{n-1}]}.
\end{align*}

Given any syntactic expression $\mathsf{E}$, we define the \emph{\bfseries generalized substitution} $\mathsf{E}[\bm{\mathsf{g}}/\bm{\mathsf{x}}]$ of $\bm{\mathsf{g}}$ for $\bm{\mathsf{x}}$ in $\mathsf{E}$ (we often abbreviate it as $\mathsf{E}[\bm{\mathsf{g}}]$), where $\bm{\mathsf{x}} \equiv (\mathsf{x_1}, \mathsf{x_2}, \dots, \mathsf{x_n})$ and $\mathscr{FV}(\mathsf{E}) \cap \mathscr{FV}(\bm{\mathsf{x}}) = \emptyset$, to be the expression 
\begin{equation*}
\mathsf{E[g_1/x_1, g_2/x_2, \dots, g_n/x_n]} 
\end{equation*}
i.e., what is obtained from $\mathsf{E}$ by simultaneously substituting $\mathsf{g_i}$ for $\mathsf{x_i}$ in $\mathsf{E}$ for $i = 1, 2, \dots, n$.
Then, it is shown in \cite{hofmann1997syntax} that if $\mathsf{\Gamma, \Theta \vdash E}$ is a judgement that is derivable in $\mathsf{MLTT(1, \Pi, \Sigma)}$, then so is the generalized substitution $\mathsf{\Delta, \Theta}[\bm{\mathsf{g}}] \vdash \mathsf{E}[\bm{\mathsf{g}}]$. Clearly, it subsumes the rules \textsc{Weak} and \textsc{Subst}. 

Given another context morphism $\mathsf{\bm{\mathsf{f}} : \Theta \to \Delta}$, we define the \emph{\bfseries composition} $\mathsf{\bm{\mathsf{g}} \circ \bm{\mathsf{f}} : \Theta \to \Gamma}$ by: 
\begin{equation*}
\mathsf{\bm{\mathsf{g}} \circ \bm{\mathsf{f}}} \stackrel{\mathrm{df. }}{\equiv} (\mathsf{g_1}[\bm{\mathsf{f}}], \mathsf{g_2}[\bm{\mathsf{f}}], \dots, \mathsf{g_n}[\bm{\mathsf{f}}]). 
\end{equation*}
Also, we have the \emph{\bfseries identity context morphism} on $\mathsf{\vdash \Gamma \ ctx}$:
\begin{equation*}
\mathit{id}_{\mathsf{\Gamma}} \stackrel{\mathrm{df. }}{\equiv} (\mathsf{x_1}, \mathsf{x_2}, \dots, \mathsf{x_n}) : \mathsf{\Gamma \to \Gamma}.
\end{equation*}

It has been shown in \cite{hofmann1997syntax} that contexts and context morphisms derivable in $\mathsf{MLTT(1, \Pi, \Sigma)}$ modulo judgmental equality $=$ form a category. 
This suggests reformulating the \emph{derived} notion of substitution as \emph{primitive}, leading to the notion of CwFs in Section~\ref{CwFs}. 

Note in particular that singleton lists of terms are context morphisms, but terms themselves are \emph{not}.
As we shall see in later sections, this point is captured by the CwF-semantics yet ignored by the CCC-semantics of STLCs.

\subsubsection{Martin-L\"{o}f Type Theories as Algebraic Theories}
\label{MLTTsAsAlgebraicTheories}
Now, we are ready to present a general class of MLTTs, in which $\mathsf{MLTT(1, \Pi, \Sigma)}$ is the minimal theory.
Our formulation is based on the framework of \emph{dependently-typed algebraic theories} \cite{pitts2001categorical} (or \emph{generalized algebraic theories} \cite{cartmell1986generalized}).

\begin{definition}[GMLTTs]
A \emph{\bfseries generalized MLTT (GMLTT)} is a quadruple $\mathcal{T} = (L_{\mathcal{T}}, S_{\mathcal{T}}, C_{\mathcal{T}}, A_{\mathcal{T}})$, where:
\begin{itemize}

\item $L_{\mathcal{T}}$ is a finite set of symbols, called the \emph{\bfseries alphabet} of $\mathcal{T}$, such that $L_{\mathcal{T}} \cap \mathscr{V} = \emptyset$ and it contains every non-variable symbol that occurs in a type or a term in $\mathsf{MLTT(1, \Pi, \Sigma)}$;

\item $S_{\mathcal{T}}$ is a subset of the set $(L_{\mathcal{T}} \cup \mathscr{V})^\ast$, whose elements are called \emph{\bfseries sorts} of $\mathcal{T}$;

\item $C_{\mathcal{T}}$ is a subset of the set $L_{\mathcal{T}}^\ast$, whose elements $\mathsf{C}$ are called \emph{\bfseries constants} of $\mathcal{T}$ and equipped with finite sequences $\mathscr{F}_{\mathcal{T}}(\mathsf{C})$ of
the form $(\mathsf{A_1}, \mathsf{A_2}, \dots, \mathsf{A_k}) \in S_{\mathcal{T}}^\ast$ or $((\mathsf{A_1}, \mathsf{A_2}, \dots, \mathsf{A_k}), \mathsf{B}) \in S_{\mathcal{T}}^\ast \times S_{\mathcal{T}}$, called the \emph{\bfseries format} of $\mathsf{C}$; and

\item $A_{\mathcal{T}}$ is a set of judgements of the form $\mathsf{A = A' :Ty}$ or $\mathsf{a = a' : A}$, called \emph{\bfseries axioms} of $\mathcal{T}$, where $\mathsf{A}, \mathsf{A'} \in S_{\mathcal{T}}$ and $\mathsf{a}, \mathsf{a'} \in (L_{\mathcal{T}} \cup \mathscr{V})^\ast$.

\end{itemize}
\end{definition}

\begin{convention}
Let $\mathcal{T}$ be a GMLTT.
A constant $\mathsf{C} \in C_{\mathcal{T}}$ with a format of the form $(\mathsf{A_1}$, $\mathsf{A_2}$, \dots, $\mathsf{A_k})$ is particularly called a \emph{\bfseries type-constant} and written $\mathsf{C} \in C_{\mathcal{T}}(\mathsf{A_1}$, $\mathsf{A_2}$, \dots, $\mathsf{A_k})$.
Similarly, a constant $\mathsf{F}  \in C_{\mathcal{T}}$ with a format of the form $((\mathsf{A_1}$, $\mathsf{A_2}$, \dots, $\mathsf{A_k}), \mathsf{B})$ is particularly called a \emph{\bfseries term-constant} and written $\mathsf{F} \in C_{\mathcal{T}}(\mathsf{A_1}$, $\mathsf{A_2}$, \dots, $\mathsf{A_k}; \mathsf{B})$; we also write $\mathscr{F}^{\mathit{dom}}_{\mathcal{T}}(\mathsf{F})$ for the sequence $(\mathsf{A_1}$, $\mathsf{A_2}$, \dots, $\mathsf{A_k})$ and $\mathscr{F}^{\mathit{cod}}_{\mathcal{T}}(\mathsf{F})$ for the sort $\mathsf{B}$.
We write $C_{\mathcal{T}}^{\mathsf{Ty}}$ and $C_{\mathcal{T}}^{\mathsf{Tm}}$ for the set of all type-constants and the set of all term-constants in $\mathcal{T}$, respectively. 
\end{convention}

\begin{definition}[Theorems of GMLTTs]
Let $\mathcal{T}$ be a GMLTT.
\emph{\bfseries Contexts}, \emph{\bfseries types} and \emph{\bfseries terms in $\bm{\mathcal{T}}$} are judgements of the forms $\mathsf{\vdash \Gamma \ ctx}$, $\mathsf{\Gamma \vdash A \ type}$ and $\mathsf{\Gamma \vdash a : A}$, respectively, obtained by the axioms and rules in Sections~\ref{Contexts}-\ref{DependentPairTypes} plus the following two rules:
\begin{align*}
&\textsc{(Type-Const)} \ \frac{ \ \mathsf{C} \in C_{\mathcal{T}}(\mathsf{A_1}, \mathsf{A_2}, \dots, \mathsf{A_k}) \ \ \ \mathsf{\Delta \equiv x_1 : A_1, x_2 : A_2, \dots, x_k : A_k} \ \ \ \bm{\mathsf{f}} : \mathsf{\Gamma} \to \mathsf{\Delta} \ }{\mathsf{ \ \Gamma \vdash C\bm{\mathsf{f}} \ type} \ } \\
&\textsc{(Term-Const)} \ \frac{ \ \mathsf{F} \in C_{\mathcal{T}}(\mathsf{A_1}, \mathsf{A_2}, \dots, \mathsf{A_k}; \mathsf{B}) \ \ \  \mathsf{\Delta \equiv x_1 : A_1, x_2 : A_2, \dots, x_k : A_k} \ \ \ \bm{\mathsf{f}} : \mathsf{\Gamma} \to \mathsf{\Delta} \ \ \ \mathsf{\Delta \vdash B \ type} \ }{\mathsf{ \ \Gamma \vdash F\bm{\mathsf{f}} : B[\bm{\mathsf{f}}/\mathsf{x_1}, \mathsf{x_2}, \dots, \mathsf{x_k}] } \ }
\end{align*}
We define $\mathcal{T}(\mathsf{\Delta}, \mathsf{\Gamma})$ to be the set of all context morphisms $\mathsf{\Delta} \to \mathsf{\Gamma}$ that consist of terms in $\mathcal{T}$, called \emph{\bfseries context morphisms in $\bm{\mathcal{T}}$}.

\emph{\bfseries Equalities in $\bm{\mathcal{T}}$} are the judgements of the form $\mathsf{\vdash \Gamma = \Gamma' \ ctx}$, $\mathsf{\Gamma \vdash A = A' \ type}$ or $\mathsf{\Gamma \vdash a = a' : A}$ obtained by the axioms and rules in Sections~\ref{Contexts}-\ref{DependentPairTypes} plus the following two rules:
\begin{align*}
&\textsc{(Type-Eq)} \ \frac{ \ \mathsf{\Delta \vdash A \ type} \ \ \ \mathsf{\Delta \vdash A' \ type} \ \ \ \mathsf{\bm{\mathsf{f}} : \Gamma \to \Delta} \ \ \ \mathsf{A = A' : Ty} \in A_{\mathcal{T}} \ }{ \ \mathsf{\Gamma \vdash A[\bm{\mathsf{f}}] = A'[\bm{\mathsf{f}}] \ type} \ } \\
&\textsc{(Term-Eq)} \ \frac{ \ \mathsf{\Delta \vdash a : A} \ \ \ \mathsf{\Delta \vdash a' : A} \ \ \ \mathsf{\bm{\mathsf{f}} : \Gamma \to \Delta} \ \ \ \mathsf{a = a' : A} \in A_{\mathcal{T}} \ }{ \ \mathsf{ \Gamma \vdash a[\bm{\mathsf{f}}] = a'[\bm{\mathsf{f}}] : A[\bm{\mathsf{f}}]} \ }
\end{align*}
Inhabited types and equalities in $\mathcal{T}$ are called \emph{\bfseries theorems in $\bm{\mathcal{T}}$}.
\end{definition}

\begin{notation}
We write $\mathsf{\Gamma \vdash_{\mathcal{T}} J}$ to mean that the judgement $\mathsf{\Gamma \vdash J}$ is derivable in a GMLTT $\mathcal{T}$.
\end{notation}

\begin{definition}[Well-formed GMLTTs]
A GMLTT $\mathcal{T}$ is \emph{\bfseries well-formed} if there is at least one rule to apply for each constant (resp. axiom) of $\mathcal{T}$.
\end{definition}

Well-formedness of a GMLTT $\mathcal{T}$ virtually means that: 
\begin{itemize}

\item For each type-constant $\mathsf{C} \in C_{\mathcal{T}}(\mathsf{A_1}, \mathsf{A_2}, \dots, \mathsf{A_k})$, there is a context morphism $\bm{\mathsf{f}} : \mathsf{\Gamma} \to \mathsf{x_1 : A_1, x_2 : A_2, \dots, x_k : A_k}$ in $\mathcal{T}$;

\item For each term-constant $\mathsf{F} \in C_{\mathcal{T}}(\mathsf{A_1}, \mathsf{A_2}, \dots, \mathsf{A_k}; \mathsf{B})$, there are a context morphism $\bm{\mathsf{f}} : \mathsf{\Gamma} \to \mathsf{x_1 : A_1, x_2 : A_2, \dots, x_k : A_k}$ in $\mathcal{T}$ and a type $\mathsf{x_1 : A_1, x_2 : A_2, \dots, x_k : A_k \vdash B \ type}$ in $\mathcal{T}$;

\item For each axiom of the form $\mathsf{A = A' : Ty} \in A_{\mathcal{T}}$, there are a context morphism $\bm{\mathsf{f}} : \mathsf{\Gamma} \to \mathsf{\Delta}$ in $\mathcal{T}$ and types $\mathsf{\Delta \vdash A \ type}$ and $\mathsf{\Delta \vdash A' \ type}$ in $\mathcal{T}$; and

\item For each axiom of the form $\mathsf{a = a' : A} \in A_{\mathcal{T}}$, there are a context morphism $\bm{\mathsf{f}} : \mathsf{\Gamma} \to \mathsf{\Delta}$ in $\mathcal{T}$ and terms $\mathsf{\Delta \vdash a : A}$ and $\mathsf{\Delta \vdash a' : A}$ in $\mathcal{T}$.

\end{itemize}

\begin{example}
Clearly, the minimal GMLTT $\Phi$ coincides with $\mathsf{MLTT(1, \Pi, \Sigma)}$.
\end{example}

\begin{example}
$\mathsf{MLTT(1, \Pi, \Sigma)}$ equipped with the natural number type $\mathsf{N}$, or $\mathsf{MLTT(1, \Pi, \Sigma, N)}$, can be presented as a well-formed GMLTT $\mathcal{N}$ such that:
\begin{itemize}

\item $L_{\mathcal{N}} \stackrel{\mathrm{df. }}{=} L_{\Phi} \cup \{ \mathsf{N}, \mathsf{zero}, \mathsf{succ}, \mathsf{R^N} \ \! \}$;

\item $S_{\mathcal{N}} \stackrel{\mathrm{df. }}{=} \{ \mathsf{A} \mid \mathsf{\Gamma \vdash_{\mathcal{N}} A \ type} \ \text{for some $\mathsf{\vdash_{\mathcal{N}} \Gamma \ ctx}$} \ \! \}$;

\item $C_{\mathcal{N}}$ contains the type-constant $\mathsf{N} \in C_{\mathcal{N}}(\bm{\epsilon})$ and the term-constants $\mathsf{zero} \in C_{\mathcal{N}}(\bm{\epsilon}; \mathsf{N})$, $\mathsf{succ} \in C_{\mathcal{N}}(\mathsf{N}; \mathsf{N})$ and $\mathsf{R^N_P} \in C_{\mathcal{N}}(\mathsf{A_1}, \mathsf{A_2}, \dots, \mathsf{A_k}, \mathsf{P[zero/y]}, \mathsf{P \Rightarrow P[succ(y)/y]}; \mathsf{P})$ for each type of the form $\mathsf{x_1 : A_1, x_2 : A_2, \dots, x_k : A_k, y : N \vdash_{\mathcal{N}} P \ type}$, where $\mathsf{R^N_P}$ is an abbreviation of $\mathsf{R^N(P)}$; and

\item $A_{\mathcal{N}}$ contains the equations 
\begin{align*}
&\mathsf{R^N_P(\bm{\mathsf{x}}, c_z, c_s) [zero/y] = c_z : P[zero/y]} \\
&\mathsf{R^N_P(\bm{\mathsf{x}}, c_z, c_s) [succ(y)/y] = c_s(R^N_P(\bm{\mathsf{x}}, c_z, c_s)) : P[succ(y)/y]}
\end{align*}
for each triple of a type of the form $\mathsf{\Gamma, y : N \vdash_{\mathcal{N}} P \ type}$ and two terms of the forms $\mathsf{\Gamma \vdash_{\mathcal{N}} c_z : P[zero/y]}$ and $\mathsf{\Gamma, y : N \vdash_{\mathcal{N}} c_s : P \Rightarrow P[succ(y)/y]}$, where $\mathsf{\Gamma \equiv x_1 : A_1, x_2 : A_2, \dots, x_k : A_k}$ and $\bm{\mathsf{x}} \equiv \mathsf{x_1, x_2, \dots, x_k}$.

\end{itemize}

We may recover the conventional rules of $\mathsf{N}$-type as follows.
Given a type $\mathsf{\Gamma, y : N \vdash_{\mathcal{N}} P \ type}$ and terms $\mathsf{\Gamma \vdash_{\mathcal{N}} c_z : P[zero/y]}$ and $\mathsf{\Gamma, y : N \vdash_{\mathcal{N}} c_s : P \Rightarrow P[succ(y)/y]}$, where $\mathsf{\Gamma \equiv x_1 : A_1, x_2 : A_2, \dots, x_k : A_k}$, we obtain $\mathsf{\Gamma \vdash_{\mathcal{N}} N \ type}$ by Type-Const, and $\mathsf{\Gamma \vdash_{\mathcal{N}} zero : N}$ and $\mathsf{\Gamma, y : N \vdash_{\mathcal{N}} succ(y) : N}$ by Term-Const.
Moreover, we get $\mathsf{\Gamma, y : N \vdash_{\mathcal{N}} c_z : P[zero/y]}$ by Weak, whence $\mathsf{\Gamma, y : N \vdash_{\mathcal{N}} R^N_P(\bm{\mathsf{x}}, c_z, c_s) : P}$ by Term-Const, where $\bm{\mathsf{x}} \equiv \mathsf{x_1}, \mathsf{x_2}, \dots, \mathsf{x_k}$.
Finally, we obtain the equations 
\begin{align*}
&\mathsf{\Gamma \vdash_{\mathcal{N}} R^N_P(\bm{\mathsf{x}}, c_z, c_x)[zero/y] = c_z : P[zero/y]} \\
&\mathsf{\Gamma, y : N \vdash_{\mathcal{N}} R^N_P(\bm{\mathsf{x}}, c_z, c_x)[succ(y)/y] = c_s(R^N_P(\bm{\mathsf{x}}, c_z, c_s)) : P[succ(y)/y]}
\end{align*}
by Term-Eq.
\end{example}

\begin{example}
$\mathsf{MLTT(1, \Pi, \Sigma)}$ equipped with identity types $\mathsf{Id}$, or $\mathsf{MLTT(1, \Pi, \Sigma, Id)}$, can be presented as a well-formed GMLTT $\mathcal{I}$ such that:
\begin{itemize}

\item $L_{\mathcal{I}} \stackrel{\mathrm{df. }}{=} L_{\Phi} \cup \{ \mathsf{Id}, \mathsf{refl}, \mathsf{R^{Id}} \ \! \}$;

\item $S_{\mathcal{I}} \stackrel{\mathrm{df. }}{=}  \{ \mathsf{A} \mid \mathsf{\Gamma \vdash_{\mathcal{I}} A \ type} \ \text{for some $\mathsf{\vdash_{\mathcal{I}} \Gamma \ ctx}$} \ \! \}$;

\item $C_{\mathcal{I}}$ contains the type-constant $\mathsf{Id_B} \in C_{\mathcal{I}}(\mathsf{B}, \mathsf{B})$ and the term-constants $\mathsf{refl_B} \in C_{\mathcal{I}}(\mathsf{A_1}, \mathsf{A_2}, \dots, \mathsf{A_k}, \mathsf{B}; \mathsf{Id_B(y, y)})$ and $\mathsf{R^{Id}_P} \in C_{\mathcal{I}}(\mathsf{A_1}, \mathsf{A_2}, \dots, \mathsf{A_k}, \mathsf{P[w/y, w/z, refl_B(w)/v]}, \mathsf{B}, \mathsf{B}, \mathsf{Id_B(y, z)}; \mathsf{P})$ for each pair of types of the forms $\mathsf{\Gamma \vdash_{\mathcal{I}} B \ type}$ and $\mathsf{\Gamma, y:B, z:B, v:Id_B(y, z) \vdash_{\mathcal{I}} P \ type}$, where $\mathsf{\Gamma \equiv x_1 : A_1, x_2 : A_2, \dots, x_k : A_k}$, and $\mathsf{Id_B}$, $\mathsf{refl_B}$ and $\mathsf{R^{Id}_P}$ are abbreviations of $\mathsf{Id(B)}$, $\mathsf{refl(B)}$ and $\mathsf{R^{Id}(P)}$, respectively;

\item $A_{\mathcal{I}}$ contains the equation
\begin{align*}
\mathsf{R^{Id}_P(\bm{\mathsf{x}}, p, w, w, refl_B(\bm{\mathsf{x}}, w)) = p : P[w/y, w/z, refl_B(\bm{\mathsf{x}}, w)/v]}
\end{align*}
for each quadruple of types $\mathsf{\Gamma \vdash_{\mathcal{I}} B \ type}$ and $\mathsf{\Gamma, y:B, z:B, v:Id_B(y, z) \vdash_{\mathcal{I}} P \ type}$, and terms $\mathsf{\Gamma, w : B \vdash_{\mathcal{I}} p : P[w/y, w/z, refl_B(\bm{\mathsf{x}}, w)/v]}$ and $\mathsf{\Gamma \vdash_{\mathcal{I}} b : B}$, where $\mathsf{\Gamma \equiv x_1 : A_1, x_2 : A_2, \dots, x_k : A_k}$ and $\bm{\mathsf{x}} \equiv \mathsf{x_1}, \mathsf{x_2}, \dots, \mathsf{x_k}$.

\end{itemize}

We may recover the conventional rules of $\mathsf{Id}$-types as follows.
Given a type $\mathsf{\Gamma \vdash_{\mathcal{I}} B \ type}$ and terms $\mathsf{\Gamma \vdash_{\mathcal{I}} b : B}$ and $\mathsf{\Gamma \vdash_{\mathcal{I}} b' : B}$, where $\mathsf{\Gamma \equiv x_1 : A_1, x_2 : A_2, \dots, x_k : A_k}$, we obtain the type $\mathsf{\Gamma \vdash_{\mathcal{I}} Id_B(b, b') \ type}$ by Type-Const; also, we have $\mathsf{\Gamma, y:B \vdash_{\mathcal{I}} Id_B(y, y) \ type}$, whence we get the term $\mathsf{\Gamma \vdash_{\mathcal{I}} refl_B(\bm{\mathsf{x}}, b) : Id_B(b, b)}$ by Term-Const, where $\bm{\mathsf{x}} \equiv \mathsf{x_1}, \mathsf{x_2}, \dots, \mathsf{x_k}$.
Moreover, given a type $\mathsf{\Gamma, y:B, z:B, v:Id_B(y, z) \vdash_{\mathcal{I}} P \ type}$ and terms $\mathsf{\Gamma, w : B \vdash_{\mathcal{I}} p : P[w/y, w/z, refl_B(\bm{\mathsf{x}}, w)/v]}$ and $\mathsf{\Gamma \vdash_{\mathcal{I}} q : Id_B(b, b')}$, we get the term $\mathsf{\Gamma, w : B \vdash_{\mathcal{I}} R^{Id}_P(\bm{\mathsf{x}}, p, b, b', q) : P[b/y, b'/z, q/v]}$ by Term-Const and the equality 
\begin{equation*}
\mathsf{\Gamma, w : B \vdash_{\mathcal{I}} R^{Id}_P(\bm{\mathsf{x}}, p, w, w, refl_B(\bm{\mathsf{x}}, w)) = p : P[w/y, w/z, refl_B(\bm{\mathsf{x}}, w)/v]}
\end{equation*}
by Term-Eq.
\end{example}

\begin{convention}
Henceforth, we shall focus on well-formed GMLTTs and call them \emph{MLTTs}.
\end{convention}

\subsubsection{Pre-Syntax}
Contexts, types and terms in MLTTs have been given together with the judgement of typing. 
Alternatively, we may first give a larger class of (possibly \emph{non-valid}) expressions by a simpler induction (on lengths of expressions), which should be called \emph{pre-contexts}, \emph{pre-types} and \emph{pre-terms}, and then carve out the \emph{valid} syntax as its subclass by the typing rules. 

For certain purposes, pre-syntax is useful. 
For instance, the interpretation of MLTTs in CwFs is easier to give on pre-syntax than on syntax \cite{hofmann1997syntax,pitts2001categorical}.
Following this method, we need pre-syntax in Section~\ref{TheoryCategoryCorrespondenceBetweenMLTTsAndCCCwDs}, and therefore we recall it here:

\begin{definition}[Pre-syntax of MLTTs \cite{hofmann1997syntax}]
The expressions called \emph{\bfseries pre-contexts} $\mathsf{\Gamma}$, \emph{\bfseries pre-types} $\mathsf{A}$, \emph{\bfseries pre-terms} $\mathsf{a}$ and \emph{\bfseries pre-context morphisms} $\bm{\mathsf{f}}$ in an MLTT $\mathcal{T}$ are given by the grammar:
\begin{align*}
\mathsf{\Gamma} &\stackrel{\mathrm{df. }}{\equiv} \mathsf{\diamondsuit} \mid \mathsf{\Gamma, x : A} \\
\mathsf{A} &\stackrel{\mathrm{df. }}{\equiv} \mathsf{1} \mid \textstyle \mathsf{\Pi_{x:A_1}A_2} \mid \textstyle \mathsf{\Sigma_{x:A_1}A_2} \mid \mathsf{C}\bm{\mathsf{f}} \\
\mathsf{a} &\stackrel{\mathrm{df. }}{\equiv} \mathsf{x} \mid \mathsf{\star} \mid \mathsf{\lambda x . \ \! a} \mid \mathsf{a_1 a_2} \mid \mathsf{\langle a_1, a_2 \rangle} \mid \mathsf{R^{\Sigma}_A(M, N)} \mid \mathsf{F}\bm{\mathsf{f}} \\
\bm{\mathsf{f}} &\stackrel{\mathrm{df. }}{\equiv} \bm{\epsilon} \mid \bm{\mathsf{f}} \ast \mathsf{a}
\end{align*}
where $\mathsf{x} \in \mathscr{V}$, $\mathsf{C} \in C_{\mathcal{T}}^{\mathsf{Ty}}$ and $\mathsf{F} \in C_{\mathcal{T}}^{\mathsf{Tm}}$.
\end{definition}

\subsubsection{Simply-Typed Lambda-Calculi as Sub-Theories}
\label{ComparisonWithSTLC}
Since we shall compare our semantics of \emph{\bfseries simply-typed $\bm{\lambda}$-calculi (STLCs)} with the standard CCC-semantics, we present them as sub-theories of MLTTs for a simple comparison of the two.

Let us first present the minimal STLC, which is equipped with \emph{finite product types} $\mathsf{1}$ and $\mathsf{\times}$; it corresponds to the non-dependently-typed version of $\mathsf{MLTT(1, \Pi, \Sigma)}$, where $\mathsf{\Rightarrow}$- and $\mathsf{\times}$-types in the former correspond respectively to $\mathsf{\Pi}$- and $\mathsf{\Sigma}$-types in the latter.
We call it the \emph{\bfseries equational theory $\bm{\mathsf{\lambda^=_{1, \times}}}$}.
Roughly, the theory $\mathsf{\lambda^=_{1, \times}}$ is obtained from $\mathsf{MLTT(1, \Pi, \Sigma)}$ by restricting types to \emph{closed} (i.e., with no free variables) ones and eliminating the rules for (judgmental) equalities between contexts and between types as there are only the trivial equalities (i.e., $\alpha$-equivalence). 

Thus, it is a formal system to deduce the following judgements:
\begin{itemize}

\item $\mathsf{\vdash \Gamma \ ctx}$ ($\mathsf{\Gamma}$ is a \emph{\bfseries context});

\item $\mathsf{\vdash A \ type}$ ($\mathsf{A}$ is a (\emph{closed}) \emph{\bfseries type});

\item $\mathsf{\Gamma \vdash a : A}$ ($\mathsf{a}$ is a \emph{\bfseries term} of type $\mathsf{A}$ in the context $\mathsf{\Gamma}$); and

\item $\mathsf{\Gamma \vdash a = a' : A}$ ($\mathsf{a}$ and $\mathsf{a'}$ are \emph{\bfseries equal} terms of type $\mathsf{A}$ in the context $\mathsf{\Gamma}$).

\end{itemize}

Speifically, $\mathsf{\lambda^=_{1, \times}}$ consists of the following axioms and rules:
\begin{align*}
&(\textsc{Ct-Emp})\frac{}{ \ \mathsf{\vdash \diamondsuit \ ctx} \ } \ \ \ (\textsc{Ct-Ext})\frac{ \ \mathsf{\vdash \Gamma \ ctx} \ \ \ \mathsf{\vdash A \ type} \ \ \ \mathsf{x} \in \mathscr{V} \setminus \mathscr{FV}(\mathsf{\Gamma}) \ }{ \ \mathsf{\vdash \Gamma, x : A \ ctx} \ } \\
&(\textsc{Var})\frac{ \ \mathsf{\vdash x_1 : A_1, x_2 : A_2, \dots, x_n : A_n \ ctx} \ \ \ j \in \{ 1, 2, \dots, n \} \ }{ \ \mathsf{x_1 : A_1, x_2 : A_2, \dots, x_n : A_n \vdash x_j : A_j} \ } \ \ \ (\textsc{Tm-EqRefl})\frac{ \ \mathsf{\Gamma \vdash a : A} \ }{ \ \mathsf{\Gamma \vdash a = a : A} \ } \\ 
&(\textsc{Tm-EqSym})\frac{ \ \mathsf{\Gamma \vdash a = a' : A} \ }{ \ \mathsf{\Gamma \vdash a' = a : A} \ } \ \ \ (\textsc{Tm-EqTrans})\frac{ \ \mathsf{\Gamma \vdash a = a' : A} \ \ \ \mathsf{\Gamma \vdash a' = a'' : A} \ }{ \ \mathsf{\Gamma \vdash a = a'' : A} \ } \\
&(\textsc{$\mathsf{1}$-Form})\frac{}{ \ \mathsf{\vdash 1 \ type} \ } \ \ \ (\textsc{$\Rightarrow$-Form}) \frac{ \ \mathsf{\vdash A \ type} \ \ \ \mathsf{\vdash B \ type} \ }{ \ \mathsf{\vdash A \Rightarrow B \ type}  \ } \ \ \ (\textsc{$\times$-Form})\frac{ \ \mathsf{\vdash A \ type} \ \ \ \mathsf{\vdash B \ type} \ }{ \ \mathsf{\vdash A \times B \ type}  \ } \\ 
&(\textsc{$\mathsf{1}$-Intro})\frac{ \ \mathsf{\vdash \Gamma \ ctx} \ }{ \ \mathsf{\Gamma \vdash \star : 1 } \ } \ \ \ (\textsc{$\Rightarrow$-Intro})\frac{ \ \mathsf{\Gamma, x : A \vdash b : B} \ }{ \ \mathsf{\Gamma \vdash \lambda x . b : A \Rightarrow B}  \ } \ \ \ (\textsc{$\times$-Intro})\frac{ \ \mathsf{\Gamma \vdash a : A \ \ \ \Gamma \vdash b : B} \ }{ \ \mathsf{\Gamma \vdash \langle a, b \rangle : A \times B}  \ } \\ 
&(\textsc{$\Rightarrow$-Elim})\frac{ \ \mathsf{\Gamma \vdash f : A \Rightarrow B} \ \ \ \mathsf{\Gamma \vdash a : A} \ }{ \ \mathsf{\Gamma \vdash f(a) : B}  \ } \ \ \ (\textsc{$\times$-Elim})\frac{ \ \mathsf{\Gamma \vdash p : A_1 \times A_2} \ }{ \ \mathsf{\Gamma \vdash \pi_i (p) : A_i} \ } \ \ \ (\textsc{$\mathsf{1}$-Uniq})\frac{ \ \mathsf{\Gamma \vdash c : C} \ }{ \ \mathsf{\Gamma \vdash c = \star : C} \ } \\ 
&(\textsc{$\Rightarrow$-Uniq})\frac{ \ \mathsf{\Gamma \vdash f : A \Rightarrow B} \ \ \ \mathsf{x} \in \mathscr{V} \setminus \mathscr{FV}(\mathsf{\Gamma}) \ }{ \ \mathsf{\Gamma \vdash \lambda x . f(x) = f : A \Rightarrow B} \ } \ \ \ (\textsc{$\times$-Uniq})\frac{ \ \mathsf{\Gamma \vdash p : A \times B} \ }{ \ \mathsf{\Gamma \vdash \langle \pi_1(p), \pi_2(p) \rangle = p : A \times B} \ } \\
&(\textsc{$\Rightarrow$-Comp})\frac{ \ \mathsf{\Gamma, x : A \vdash b : B} \ \ \ \mathsf{\Gamma \vdash a : A} \ }{ \ \mathsf{\Gamma \vdash (\lambda x . b)(a) = b[a/x] : B}  \ } \ \ \ (\textsc{$\times$-Comp})\frac{ \ \mathsf{\Gamma \vdash a_i : A_i} \ }{ \ \mathsf{\Gamma \vdash \pi_i (\langle a_1, a_2 \rangle) = a_i : A_i} \ }  
\end{align*}
where we again omit describing congruence rules.
Following a standard convention, we write $\mathsf{A \Rightarrow B}$ and $\mathsf{A \times B}$ for $\mathsf{\Pi_{x : A}B}$ and $\mathsf{\Sigma_{x : A}B}$, respectively, since there are no dependent types in $\mathsf{\lambda^=_{1, \times}}$.
Note that \textsc{$\times$-Elim} and \textsc{$\times$-Comp} are \emph{admissible} in $\mathsf{MLTT(1, \Pi, \Sigma)}$ by \textsc{$\Sigma$-Elim} and \textsc{$\Sigma$-Comp} as described in \cite{hofmann1997syntax}, and vice versa in $\mathsf{\lambda^=_{1, \times}}$; the other rules are just inherited from $\mathsf{MLTT(1, \Pi, \Sigma)}$.

It is easy to see that the theory $\mathsf{\lambda^=_{1, \times}}$ coincides with the standard equational (with respect to \emph{$\beta \eta$-equivalence}) theory of the \textsf{STLC} equipped with finite product types \cite{crole1993categories,jacobs1999categorical,lambek1988introduction}.


Accordingly, the general class of STLCs is as follows:
\begin{definition}[STLCs]
An \emph{\bfseries STLC} is a quadruple $\mathcal{T} = (L_{\mathcal{T}}, S_{\mathcal{T}}, C_{\mathcal{T}}, A_{\mathcal{T}})$, where:
\begin{itemize}

\item $L_{\mathcal{T}}$ is a finite set of symbols, called the \emph{\bfseries alphabet} of $\mathcal{T}$, such that $L_{\mathcal{T}} \cap \mathscr{V} = \emptyset$ and it contains every non-variable symbol that occurs in a type or a term in $\mathsf{\lambda^=_{1, \times}}$;

\item $S_{\mathcal{T}}$ is a subset of the set $L_{\mathcal{T}}^\ast$, whose elements are called \emph{\bfseries sorts} of $\mathcal{T}$;

\item $C_{\mathcal{T}}$ is a subset of the set $L_{\mathcal{T}}^\ast$, whose elements $\mathsf{C}$ are called \emph{\bfseries constants} of $\mathcal{T}$ and equipped with finite sequences $\mathscr{F}_{\mathcal{T}}(\mathsf{C})$ of
the form $\bm{\epsilon} \in S_{\mathcal{T}}^\ast$ or $((\mathsf{A_1}, \mathsf{A_2}, \dots, \mathsf{A_k}), \mathsf{B}) \in S_{\mathcal{T}}^\ast \times S_{\mathcal{T}}$, called the \emph{\bfseries format} of $\mathsf{C}$; and

\item $A_{\mathcal{T}}$ is a set of judgements of the form $\mathsf{a = a' : A}$, called \emph{\bfseries axioms} of $\mathcal{T}$, where $\mathsf{A} \in S_{\mathcal{T}}$ and $\mathsf{a}, \mathsf{a'} \in (L_{\mathcal{T}} \cup \mathscr{V})^\ast$

\end{itemize}
where the same convention is applied as in the case of MLTTs. 
\end{definition}

\begin{definition}[Theorems of STLCs]
Let $\mathcal{T}$ be an STLC.
\emph{\bfseries Contexts, types and terms in $\bm{\mathcal{T}}$} are judgements of the forms $\mathsf{\vdash \Gamma \ ctx}$, $\mathsf{\vdash A \ type}$ and $\mathsf{\vdash a : A}$, respectively, obtained by the axioms and rules given above in this section plus the following two rules:
\begin{align*}
&\textsc{(Type-Const)} \ \frac{ \ \mathsf{C} \in C_{\mathcal{T}}(\bm{\epsilon}) \ }{\mathsf{ \ \vdash C \ type} \ } \\
&\textsc{(Term-Const)} \ \frac{ \ \mathsf{F} \in C_{\mathcal{T}}(\mathsf{A_1}, \mathsf{A_2}, \dots, \mathsf{A_k}; \mathsf{B}) \ \ \  \mathsf{\Delta \equiv x_1 : A_1, x_2 : A_2, \dots, x_k : A_k} \ \ \ \bm{\mathsf{f}} : \mathsf{\Gamma} \to \mathsf{\Delta} \ \ \ \mathsf{\vdash B \ type} \ }{\mathsf{ \ \Gamma \vdash F\bm{\mathsf{f}} : B} \ }
\end{align*}

\emph{\bfseries Equalities in $\bm{\mathcal{T}}$} are the judgements of the form $\mathsf{\Gamma \vdash a = a' : A}$ obtained by the axioms and rules given above in this section plus the following rule:
\begin{align*}
&\textsc{(Term-Eq)} \ \frac{ \ \mathsf{\Delta \vdash a : A} \ \ \ \mathsf{\Delta \vdash a' : A} \ \ \ \mathsf{\bm{\mathsf{f}} : \Gamma \to \Delta} \ \ \ \mathsf{a = a' : A} \in A_{\mathcal{T}} \ }{ \ \mathsf{ \Gamma \vdash a[\bm{\mathsf{f}}] = a'[\bm{\mathsf{f}}] : A} \ }
\end{align*}
Inhabited types and equalities in $\mathcal{T}$ are also called \emph{\bfseries theorems in $\bm{\mathcal{T}}$}.
\end{definition}

\begin{definition}[Well-formed STLCs]
An STLC $\mathcal{T}$ is \emph{\bfseries well-formed} if there is at least one rule to apply for each constant (resp. axiom) of $\mathcal{T}$.
\end{definition}

\begin{convention}
We shall focus on \emph{well-formed} STLCs and simply call them \emph{STLCs}.
\end{convention}

It is clear that STLCs are just the non-dependently-typed version of MLTTs, and our notion of STLCs coincides with the standard notion of (equational) STTs \cite{lambek1988introduction,crole1993categories,pitts2001categorical}.

\subsection{Categories with Families}
\label{CwFs}
Next, we review \emph{categories with families (CwFs)} and their \emph{semantic type formers} as well as the interpretation of $\mathsf{MLTT(1, \Pi, \Sigma)}$ in CwFs that support $1$-, $\Pi$- and $\Sigma$-types in the strict sense.

\subsubsection{Categories with Families}
\begin{definition}[CwFs \cite{dybjer1996internal,hofmann1997syntax}]
\label{DefCwFs}
A \emph{\bfseries category with families (CwF)} is a tuple 
\begin{equation*}
\mathcal{C} = (\mathcal{C}, \mathit{Ty}, \mathit{Tm}, \_\{\_\}, T, \_.\_, \mathit{p}, \mathit{v}, \langle\_,\_\rangle_\_)
\end{equation*}
where:
\begin{itemize}

\item $\mathcal{C}$ is a category equipped with a specified terminal object $T \in \mathcal{C}$;

\item $\mathit{Ty}$ assigns, to each object $\Gamma \in \mathcal{C}$, a set $\mathit{Ty}(\Gamma)$, called the set of all \emph{\bfseries types} in the \emph{\bfseries context} $\Gamma$;

\item $\mathit{Tm}$ assigns, to each pair of an object $\Gamma \in \mathcal{C}$ and a type $A \in \mathit{Ty}(\Gamma)$, a set $\mathit{Tm}(\Gamma, A)$, called the set of all \emph{\bfseries terms} of type $A$ in the context $\Gamma$;

\item For each morphism $\phi : \Delta \to \Gamma$ in $\mathcal{C}$, $\_\{\_\}$ induces a function $\_\{ \phi \} : \mathit{Ty}(\Gamma) \to \mathit{Ty}(\Delta)$, called the \emph{\bfseries substitution on types}, and a family $(\_\{ \phi \}_A : \mathit{Tm}(\Gamma, A) \to \mathit{Tm}(\Delta, A\{ \phi \}))_{A \in \mathit{Ty}(\Gamma)}$ of functions, called the \emph{\bfseries substitutions on terms};

\item $\_ . \_$ assigns, to each pair of a context $\Gamma \in \mathcal{C}$ and a type $A \in \mathit{Ty}(\Gamma)$, a context $\Gamma . A \in \mathcal{C}$, called the \emph{\bfseries comprehension} of $A$;

\item $\mathit{p}$ associates each pair of a context $\Gamma \in \mathcal{C}$ and a type $A \in \mathit{Ty}(\Gamma)$ with a morphism $\mathit{p}(A) : \Gamma . A \to \Gamma$
in $\mathcal{C}$, called the \emph{\bfseries first projection} associated to $A$;

\item $\mathit{v}$ associates each pair of a context $\Gamma \in \mathcal{C}$ and a type $A \in \mathit{Ty}(\Gamma)$ with a term $\mathit{v}_A \in \mathit{Tm}(\Gamma . A, A\{\mathit{p}(A)\})$
called the \emph{\bfseries second projection} associated to $A$; and

\item $\langle \_, \_ \rangle_\_$ assigns, to each triple of a morphism $\phi : \Delta \to \Gamma$ in $\mathcal{C}$, a type $A \in \mathit{Ty}(\Gamma)$ and a term $g \in \mathit{Tm}(\Delta, A\{ \phi \})$, a morphism $\langle \phi, g \rangle_A : \Delta \to \Gamma . A$
in $\mathcal{C}$, called the \emph{\bfseries extension} of $\phi$ by $g$

\end{itemize}
that satisfies, for all $\Gamma, \Delta, \Theta \in \mathcal{C}$, $A \in \mathit{Ty}(\Gamma)$, $\phi : \Delta \to \Gamma$, $\varphi : \Theta \to \Delta$, $f \in \mathit{Tm}(\Gamma, A)$ and $g \in \mathit{Tm}(\Delta, A\{ \phi \})$, the following equations:
\begin{itemize}

\item \textsc{(Ty-Id)} $A \{ \mathit{id}_\Gamma \} = A$;

\item \textsc{(Ty-Comp)} $A \{ \phi \circ \varphi \} = A \{ \phi \} \{ \varphi \}$;

\item \textsc{(Tm-Id)} $f \{ \mathit{id}_\Gamma \}_A = f$;

\item \textsc{(Tm-Comp)} $f \{ \phi \circ \varphi \}_A = f \{ \phi \}_A \{ \varphi \}_{A\{\phi\}}$;

\item \textsc{(Cons-L)} $\mathit{p}(A) \circ \langle \phi, g \rangle_A = \phi$;

\item \textsc{(Cons-R)} $\mathit{v}_A \{ \langle \phi, g \rangle_A \}_{A\{\mathit{p}(A)\}} = g$;

\item \textsc{(Cons-Nat)} $\langle \phi, g \rangle_A \circ \varphi = \langle \phi \circ \varphi, g \{ \varphi \}_{A\{ \phi \}} \rangle_A$;

\item \textsc{(Cons-Id)} $\langle \mathit{p}(A), \mathit{v}_A \rangle_A = \mathit{id}_{\Gamma . A}$.

\end{itemize}

\end{definition}

\begin{notation}
We often omit the subscript $A$ in $\_ \{ \_ \}_A$ and $\langle \_, \_ \rangle_A$ if it does not cause any ambiguity. 
\end{notation}


CwFs are very close to the syntax, which is best described by the following \emph{term model}:
\begin{definition}[The term model $\mathcal{T}(1, \Pi, \Sigma)$ \cite{hofmann1997syntax}]
\label{DefTermModelT}
The CwF of contexts and context morphisms of $\mathsf{MLTT(1, \Pi, \Sigma)}$ modulo $\mathsf{=}$, called the \emph{\bfseries term model} $\mathcal{T}(1, \Pi, \Sigma)$, is defined as follows:
\begin{itemize}

\item The category $\mathcal{T}(1, \Pi, \Sigma)$ consists of contexts and context morphisms of $\mathsf{MLTT(1, \Pi, \Sigma)}$ as defined in Section~\ref{MLTT} modulo (judgmental) equality $\mathsf{=}$. 
Let us write $[\mathsf{\Gamma}]$ (resp. $[\mathsf{A}]$, $[\mathsf{a}]$) for the equivalence class of a context $\mathsf{\vdash \Gamma \ ctx}$ (resp. a type $\mathsf{\Gamma \vdash A \ type}$, a term $\mathsf{\Gamma \vdash a : A}$).

\item Given $[\mathsf{\Gamma}] \in \mathcal{T}(1, \Pi, \Sigma)$, we define $\mathit{Ty}([\mathsf{\Gamma}]) \stackrel{\mathrm{df. }}{=} \{ [\mathsf{A}] \ \! | \ \! \mathsf{\Gamma \vdash A \ type} \ \! \}$ and $\mathit{Tm}([\mathsf{\Gamma}], [\mathsf{A}]) \stackrel{\mathrm{df. }}{=} \{ [\mathsf{a}] \ \! | \ \! \mathsf{\Gamma \vdash a : A} \ \! \}$. Below, we represent these equivalence classes by their arbitrary representatives.

\item The functions $\_\{ \_ \}$ are generalized substitutions $\_[\_]$ on types and terms in Section~\ref{MLTT}.

\item The terminal object is (the equivalence class of) the empty context $\mathsf{\vdash \diamondsuit \ ctx}$.

\item Given $\mathsf{\vdash \Gamma \ ctx} \in \mathcal{T}(1, \Pi, \Sigma)$ and $\mathsf{\Gamma \vdash A \ type} \in \mathit{Ty}(\mathsf{\Gamma})$, we define $\mathsf{\Gamma . A}  \stackrel{\mathrm{df. }}{=} \mathsf{\vdash \Gamma, x : A \ ctx}$, where $\mathsf{x} \in \mathscr{V} \setminus \mathscr{FV}(\mathsf{\Gamma})$.

\item Assume $\mathsf{\vdash \Gamma = x_1 : A_1, x_2 : A_2, \dots, x_n : A_n \ ctx}$. Then, we define $\mathit{p}(\mathsf{A}) \stackrel{\mathrm{df. }}{=} \langle \mathsf{x_1, x_2, \dots, x_n} \rangle : \mathsf{\Gamma . A} \to \mathsf{\Gamma}$ and $\mathit{v}_{\mathsf{A}} \stackrel{\mathrm{df. }}{=} \mathsf{\Gamma, x : A \vdash x : A} \in \mathit{Tm}(\mathsf{\Gamma . A}, \mathsf{A[\mathit{p}(A)]})$.

\item Given $\phi : \mathsf{\Delta} \to \mathsf{\Gamma}$ and $\mathsf{\Delta \vdash g : A[\phi]} \in \mathit{Tm}(\mathsf{\Gamma}, \mathsf{A[\phi]})$, we define $\langle \phi, \mathsf{g} \rangle_{\mathsf{A}} \stackrel{\mathrm{df. }}{=} (\phi, \mathsf{g}) : \mathsf{\Gamma \to \Delta . A}$.

\end{itemize}
\end{definition}

There are various non-syntactic instances of CwFs, e.g., sets \cite{hofmann1997syntax}, groupoids \cite{hofmann1998groupoid} and games \cite{abramsky2015games,yamada2017game}.
We skip describing these CwFs and leave the details to the references. 

Nevertheless, CwFs model only the fragment of MLTTs that is common to all types; we need to equip them with \emph{semantic $1$-, $\Pi$- and $\Sigma$-type formers} \cite{hofmann1997syntax} in order to model $\mathsf{MLTT(1, \Pi, \Sigma)}$.

\begin{definition}[CwFs with $1$-type \cite{hofmann1997syntax}]
A CwF $\mathcal{C}$ \emph{\bfseries supports $\bm{1}$-type} if:
\begin{itemize}

\item \textsc{(Unit-Form)} For any $\Gamma \in \mathcal{C}$, there is a type $1_{\Gamma} \in \mathit{Ty}(\Gamma)$;

\item \textsc{(Unit-Intro)} For any $\Gamma \in \mathcal{C}$, there is a term $\star_{\Gamma} \in \mathit{Tm}(\Gamma, 1_\Gamma)$;

\item \textsc{(Unit-Subst)} For any morphism $\phi : \Delta \to \Gamma$ in $\mathcal{C}$, we have $1_{\Gamma}\{\phi\} = 1_{\Delta}$; and

\item \textsc{($\star$-Subst)} For any morphism $\phi : \Delta \to \Gamma$ in $\mathcal{C}$, we have $\star_{\Gamma}\{\phi\} = \star_{\Delta}$.

\end{itemize}
It supports $1$-type \emph{\bfseries in the strict sense} if it additionally satisfies: 
\begin{itemize}
\item (\textsc{Unit-Uniq}) $f = \star_\Gamma$ for any $f \in \mathit{Tm}(\Gamma, 1_\Gamma)$.
\end{itemize}
\end{definition}

\begin{definition}[CwFs with $\Pi$-types \cite{hofmann1997syntax}]
A CwF $\mathcal{C}$ \emph{\bfseries supports $\bm{\Pi}$-types} if:
\begin{itemize}

\item \textsc{($\Pi$-Form)} For any $\Gamma \in \mathcal{C}$, $A \in \mathit{Ty}(\Gamma)$ and $B \in \mathit{Ty}(\Gamma . A)$, there is a type $\Pi (A, B) \in \mathit{Ty}(\Gamma)$;

\item \textsc{($\Pi$-Intro)} For any $f \in \mathit{Tm}(\Gamma . A, B)$, there is a term $\lambda_{A, B} (f) \in \mathit{Tm}(\Gamma, \Pi (A, B))$;

\item \textsc{($\Pi$-Elim)} For any $h \in \mathit{Tm}(\Gamma, \Pi (A, B))$ and $a \in \mathit{Tm}(\Gamma, A)$, there is a term $\mathit{App}_{A, B} (h, a) \in \mathit{Tm}(\Gamma, B\{ \overline{a} \})$, where $\overline{a} \stackrel{\mathrm{df. }}{=} \langle \textit{id}_\Gamma, a \rangle_A : \Gamma \to \Gamma . A$;

\item \textsc{($\Pi$-Comp)} We have $\mathit{App}_{A, B} (\lambda_{A, B} (f) , a) = f \{ \overline{a} \}$;

\item \textsc{($\Pi$-Subst)} For any $\Delta \in \mathcal{C}$ and $\phi : \Delta \to \Gamma$ in $\mathcal{C}$, we have $\Pi (A, B) \{ \phi \} = \Pi (A\{\phi\}, B\{\phi^+\}$, where $\phi^+ \stackrel{\mathrm{df. }}{=} \langle \phi \circ \mathit{p}(A\{ \phi \}), \mathit{v}_{A\{\phi\}} \rangle_A : \Delta . A\{\phi\} \to \Gamma . A$;

\item \textsc{($\lambda$-Subst)} For any $g \in \mathit{Tm} (\Gamma. A, B)$, we have $\lambda_{A, B} (g) \{ \phi \} = \lambda_{A\{\phi\}, B\{\phi^+\}} (g \{ \phi^+ \})$; and

\item \textsc{(App-Subst)} We have $\mathit{App}_{A, B} (h, a) \{ \phi \} = \mathit{App}_{A\{\phi\}, B\{\phi^+\}} (h \{ \phi \}, a \{ \phi \})$. 

\end{itemize}
It supports $\Pi$-types \emph{\bfseries in the strict sense} if it additionally satisfies:
\begin{itemize}
\item (\textsc{$\lambda$-Uniq}) $\lambda_{A, B} (\mathit{App}_{A\{\mathit{p}(A)\}, B\{\mathit{p}(A)^+\}}(k, \mathit{v}_A))\{\mathit{p}(A)\} = k$ for any $k \in \mathit{Tm}(\Gamma . A, \Pi(A, B)\{\mathit{p}(A)\})$.
\end{itemize}
\end{definition}

\begin{definition}[CwFs with $\Sigma$-types \cite{hofmann1997syntax}]
\label{DefCwFsWithSigmaTypes}
A CwF $\mathcal{C}$ \emph{\bfseries supports $\bm{\Sigma}$-types} if:
\begin{itemize}

\item \textsc{($\Sigma$-Form)} For any $\Gamma \in \mathcal{C}$, $A \in \mathit{Ty}(\Gamma)$ and $B \in \mathit{Ty}(\Gamma . A)$, there is a type $\Sigma (A, B) \in \mathit{Ty}(\Gamma)$;

\item \textsc{($\Sigma$-Intro)} There is a morphism $\mathit{Pair}_{A, B} : \Gamma . A . B \to \Gamma . \Sigma (A, B)$ in $\mathcal{C}$;

\item \textsc{($\Sigma$-Elim)} Given $P \in \mathit{Ty}(\Gamma . \Sigma (A, B))$ and $f \in \mathit{Tm}(\Gamma . A . B, P\{ \mathit{Pair}_{A, B} \})$, there is a term $\mathcal{R}^{\Sigma}_{A, B, P}(f) \in \mathit{Tm}(\Gamma . \Sigma (A, B), P)$;

\item \textsc{($\Sigma$-Comp)} $\mathcal{R}^{\Sigma}_{A, B, P}(f) \{ \mathit{Pair}_{A, B}\} = f$;

\item \textsc{($\Sigma$-Subst)} For any $\Delta \in \mathcal{C}$ and $\phi : \Delta \to \Gamma$ in $\mathcal{C}$, $\Sigma (A, B) \{ \phi \} = \Sigma (A\{\phi\}, B\{\phi^+\})$;

\item \textsc{(Pair-Subst)} $\mathit{p}(\Sigma (A, B)) \circ \mathit{Pair}_{A, B} = \mathit{p}(A) \circ \mathit{p}(B)$, $\phi^\star \circ \mathit{Pair}_{A\{\phi\}, B\{\phi^+\}} = \mathit{Pair}_{A, B} \circ \phi^{++}$, where $\phi^\star \stackrel{\mathrm{df. }}{=} \langle \phi \circ \mathit{p}(\Sigma(A,B)\{\phi\}), \mathit{v}_{\Sigma(A,B)\{\phi\}} \rangle_{\Sigma(A,B)} : \Delta . \Sigma(A,B)\{\phi\} \to \Gamma . \Sigma(A,B)$, $\phi^{++} \stackrel{\mathrm{df. }}{=} \langle \phi^+ \circ \mathit{p}(B\{\phi^+\}), \mathit{v}_{B\{\phi^+\}} \rangle_B : \Delta . A\{\phi\} . B\{\phi^+\} \to \Gamma . A . B$; and

\item \textsc{($\mathcal{R}^{\Sigma}$-Subst)} $\mathcal{R}^{\Sigma}_{A, B, P}(f) \{\phi^\star \} = \mathcal{R}^{\Sigma}_{A\{\phi\}, B\{\phi^+\}, P\{\phi^\star \}} (f \{ \phi^{++} \})$.
\end{itemize}

It supports $\Sigma$-types \emph{\bfseries in the strict sense} if it satisfies: 
\begin{itemize}
\item (\textsc{$\mathcal{R}^{\Sigma}$-Uniq}) $g = \mathcal{R}^{\Sigma}_{A, B, P}(f)$ for any $g \in \mathit{Tm}(\Gamma . \Sigma (A, B), P)$ such that $g \{ \mathit{Pair}_{A, B}\} = f$.
\end{itemize}
\end{definition}

In light of Definition~\ref{DefTermModelT}, it is straightforward to see what these semantic type formers in the term model $\mathcal{T}(1, \Pi, \Sigma)$ are; thus, we leave their details to \cite{hofmann1997syntax}.

\begin{proposition}[Well-defined $\mathcal{T}(1, \Pi, \Sigma)$ \cite{hofmann1997syntax}]
\label{WellDefinedSyntacticCwFs}
The term model $\mathcal{T}(1, \Pi, \Sigma)$ is a well-defined CwF that supports $1$-, $\Pi$- and $\Sigma$-types in the strict sense.
\end{proposition}

Now, one may see clearly why CwFs equipped with those semantic type formers are not categorical: They are not a strict version of some weak concepts defined by universal properties. 
The work \cite{awodey2016natural} by Awodey is an attempt to address this point.

\subsubsection{Interpretation of MLTTs in CwFs}
\label{InterpretationOfMLTTInCwFs}
Now, let us recall the interpretation of $\mathsf{MLTT(1, \Pi, \Sigma)}$ in CwFs that support $1$-, $\Pi$-, $\Sigma$-types in the strict sense \cite{hofmann1997syntax}.
Since a deduction tree of a judgement in $\mathsf{MLTT(1, \Pi, \Sigma)}$ is not necessarily unique in the presence of the rules \textsc{Ty-Con} and \textsc{Tm-Con}, a priori we cannot define an interpretation by induction on deduction trees. 
For this point, a standard approach is to define an interpretation $\llbracket \_ \rrbracket$ on \emph{pre-syntax} by induction on the lengths of expressions, which is \emph{partial}, and show that it is total and well-defined on every valid or well-typed syntax (i.e., judgement) and preserves judgmental equality as the corresponding semantic equality \cite{streicher2012semantics,hofmann1997syntax,pitts2001categorical}.

By this \emph{soundness} result, a posteriori we may describe the interpretation $\llbracket \_ \rrbracket$ of the syntax by induction on derivation trees of judgements:

\begin{definition}[Interpretation of $\mathsf{MLTT(1, \Pi, \Sigma)}$ in CwFs \cite{hofmann1997syntax}]
\label{DefMLTTInCwFs}
The interpretation $\llbracket \_ \rrbracket$ of $\mathsf{MLTT(1, \Pi, \Sigma)}$ in a CwF $\mathcal{C} = (\mathcal{C}, \mathit{Ty}, \mathit{Tm}, \_\{\_\}, T, \_.\_, \mathit{p}, \mathit{v}, \langle\_,\_\rangle_\_)$ that supports semantic type formers $1 = (1, \star)$, $\Pi = (\Pi, \lambda, \mathit{App})$, $\Sigma = (\Sigma, \mathit{Pair}, R^{\Sigma})$ in the strict sense is defined as follows:
\begin{itemize}

\item (\textsc{Ct-Emp}) $\llbracket \mathsf{\vdash \diamondsuit \ ctx} \rrbracket \stackrel{\mathrm{df. }}{=} T$;

\item (\textsc{Ct-Ext}) $\llbracket \mathsf{\vdash \Gamma, x : A \ ctx} \rrbracket \stackrel{\mathrm{df. }}{=} \llbracket \mathsf{\vdash \Gamma \ ctx} \rrbracket . \llbracket \mathsf{\Gamma \vdash A \ type} \rrbracket$;

\item (\textsc{$\mathsf{1}$-Form}) $\llbracket \mathsf{\Gamma \vdash \mathsf{1} \ type} \rrbracket \stackrel{\mathrm{df. }}{=} 1_{\llbracket \mathsf{\Gamma} \rrbracket}$;

\item (\textsc{$\Pi$-Form}) $\llbracket \mathsf{\Gamma \vdash \Pi_{x : A}B \ type} \rrbracket \stackrel{\mathrm{df. }}{=} \Pi (\llbracket \mathsf{\Gamma \vdash A \ type} \rrbracket, \llbracket \mathsf{\Gamma, x : A \vdash B \ type} \rrbracket)$;

\item (\textsc{$\Sigma$-Form}) $\llbracket \mathsf{\Gamma \vdash \Sigma_{x : A}B \ type} \rrbracket \stackrel{\mathrm{df. }}{=} \Sigma (\llbracket \mathsf{\Gamma \vdash A \ type} \rrbracket, \llbracket \mathsf{\Gamma, x : A \vdash B \ type} \rrbracket)$;

\item (\textsc{Var}) $\llbracket \mathsf{\Gamma, x: A \vdash x : A} \rrbracket \stackrel{\mathrm{df. }}{=} \mathit{v}_{\llbracket \mathsf{A} \rrbracket}$, $\llbracket \mathsf{\Gamma, x: A, \Delta, y : B \vdash x : A} \rrbracket \stackrel{\mathrm{df. }}{=} \llbracket \mathsf{\Gamma, x: A, \Delta \vdash x : A} \rrbracket \{ \mathit{p}(\llbracket \mathsf{\Gamma, x : A, \Delta \vdash B \ type} \rrbracket) \}$;

\item (\textsc{Ty-Con}) $\llbracket \mathsf{\Delta \vdash A \ type} \rrbracket \stackrel{\mathrm{df. }}{=} \llbracket \mathsf{\Gamma \vdash A \ type} \rrbracket$;

\item (\textsc{Tm-Con}) $\llbracket \mathsf{\Delta \vdash a : B} \rrbracket \stackrel{\mathrm{df. }}{=} \llbracket \mathsf{\Gamma \vdash a : A} \rrbracket$;

\item (\textsc{$\mathsf{1}$-Intro}) $\llbracket \mathsf{\Gamma \vdash \star : \mathsf{1}} \rrbracket \stackrel{\mathrm{df. }}{=} \star_{\llbracket \mathsf{\Gamma} \rrbracket}$;

\item (\textsc{$\Pi$-Intro}) $\llbracket \mathsf{\Gamma \vdash \lambda x . \ \! b : \Pi_{x : A}B} \rrbracket \stackrel{\mathrm{df. }}{=} \lambda_{\llbracket \mathsf{A} \rrbracket, \llbracket \mathsf{B} \rrbracket} (\llbracket \mathsf{\Gamma, x : A \vdash b : B} \rrbracket)$;

\item (\textsc{$\Sigma$-Intro}) $\llbracket \mathsf{\Gamma \vdash (a, b) : \Sigma_{x : A}B} \rrbracket \stackrel{\mathrm{df. }}{=} \mathit{Pair}_{\llbracket \mathsf{A} \rrbracket, \llbracket \mathsf{B} \rrbracket} \circ \langle \overline{\llbracket \mathsf{\Gamma \vdash a : A} \rrbracket}, \llbracket \mathsf{\Gamma \vdash b : B[a/x] \rrbracket} \rangle$;

\item (\textsc{$\Pi$-Elim}) $\llbracket \mathsf{\Gamma \vdash f(a) : B[a/x]} \rrbracket \stackrel{\mathrm{df. }}{=} \mathit{App}_{\llbracket \mathsf{A} \rrbracket, \llbracket \mathsf{B} \rrbracket} (\llbracket \mathsf{\Gamma \vdash f : \Pi_{x : A}B} \rrbracket, \llbracket \mathsf{\Gamma \vdash a : A} \rrbracket \rangle)$; and

\item (\textsc{$\Sigma$-Elim}) $\llbracket \mathsf{\Gamma \vdash R^{\Sigma}(C, g, p) : C[p/z]} \rrbracket \stackrel{\mathrm{df. }}{=} R^{\Sigma}_{\llbracket \mathsf{A} \rrbracket, \llbracket \mathsf{B} \rrbracket, \llbracket \mathsf{C} \rrbracket} (\llbracket \mathsf{\Gamma, x : A, y : B \vdash g : C[(x, y)/z]} \rrbracket) \circ \overline{\llbracket \mathsf{\Gamma \vdash p : \Sigma_{x : A} B} \rrbracket}$

\end{itemize}
where the hypotheses of the rules are as presented in Section~\ref{MLTT}, $\overline{\llbracket \mathsf{\Gamma \vdash a : A} \rrbracket} \stackrel{\mathrm{df. }}{=} \langle \mathit{id}_{\llbracket \mathsf{\Gamma} \rrbracket}, \llbracket \mathsf{a} \rrbracket \rangle : \llbracket \mathsf{\Gamma} \rrbracket \to \llbracket \mathsf{\Gamma} \rrbracket . \llbracket \mathsf{A} \rrbracket$ and $\overline{\llbracket \mathsf{\Gamma \vdash p : \Sigma_{x : A} B} \rrbracket} \stackrel{\mathrm{df. }}{=} \langle \mathit{id}_{\llbracket \mathsf{\Gamma} \rrbracket}, \llbracket \mathsf{p} \rrbracket \rangle : \llbracket \mathsf{\Gamma} \rrbracket \to \llbracket \mathsf{\Gamma} \rrbracket . \llbracket \mathsf{\Sigma_{x:A}B} \rrbracket$.
\end{definition}

\begin{theorem}[Soundness of CwFs \cite{hofmann1997syntax}]
\label{ThmSoundnessOfCwFs}
If $\mathsf{\vdash \Gamma = \Gamma' \ ctx}$ (resp. $\mathsf{\Gamma \vdash A = A' \ type}$, $\mathsf{\Gamma \vdash a = a' : A}$) in $\mathsf{MLTT(1, \Pi, \Sigma)}$, then $\llbracket \mathsf{\Gamma} \rrbracket = \llbracket \mathsf{\Gamma'} \rrbracket$ (resp. $\llbracket \mathsf{A} \rrbracket = \llbracket \mathsf{A'} \rrbracket$, $\llbracket \mathsf{a} \rrbracket = \llbracket \mathsf{a'} \rrbracket$) for the interpretation $\llbracket \_ \rrbracket$ of $\mathsf{MLTT(1, \Pi, \Sigma)}$ in any CwF that supports $1$-, $\Pi$- and $\Sigma$-types in the strict sense.
\end{theorem}

By Proposition~\ref{WellDefinedSyntacticCwFs}, this interpretation is also \emph{complete}:
\begin{theorem}[Completeness of CwFs \cite{hofmann1997syntax}]
\label{ThmCompletenessOfCwFs}
Given $\mathsf{\vdash \Gamma \ ctx}$ and $\mathsf{\vdash \Delta \ ctx}$ (resp. $\mathsf{\Gamma \vdash A \ type}$ and $\mathsf{\Gamma \vdash B \ type}$, $\mathsf{\Gamma \vdash a : A}$ and $\mathsf{\Gamma \vdash a' : A}$) in $\mathsf{MLTT(1, \Pi, \Sigma)}$, if $\llbracket \mathsf{\Gamma} \rrbracket = \llbracket \mathsf{\Delta} \rrbracket$ (resp. $\llbracket \mathsf{A} \rrbracket = \llbracket \mathsf{B} \rrbracket$, $\llbracket \mathsf{a} \rrbracket = \llbracket \mathsf{a'} \rrbracket$) for the interpretation $\llbracket \_ \rrbracket$ in any CwF that supports $1$-, $\Sigma$- and $\Pi$-types in the strict sense, then there is a judgement $\mathsf{\vdash \Gamma  = \Delta \ ctx}$ (resp. $\mathsf{\Gamma \vdash A = B \ type}$, $\mathsf{\Gamma \vdash a = a' : A}$) in $\mathsf{MLTT(1, \Pi, \Sigma)}$.
\end{theorem}

For $\mathsf{\lambda^=_{1, \times}}$ is a sub-theory of $\mathsf{MLTT(1, \Pi, \Sigma)}$, any CwF that supports $1$-, $\Pi$- and $\Sigma$-types in the strict sense clearly interprets $\mathsf{\lambda^=_{1, \times}}$, and it is not hard to see that the interpretation is sound and complete as well.

\subsubsection{Interpretation of STLCs in CCCs}
The standard categorical semantics of $\mathsf{\lambda^=_{1, \times}}$ is in CCCs, where contexts and types are both interpreted by objects, and terms by morphisms:
\begin{definition}[Interpretation of $\mathsf{\lambda^=_{1, \times}}$ in CCCs \cite{lambek1988introduction,pitts2001categorical,jacobs1999categorical,crole1993categories}]
\label{DefLambdaInCCCs}
The interpretation $\llbracket \_ \rrbracket$ of the equational theory $\mathsf{\lambda^=_{1, \times}}$ in a CCC $\mathcal{C} = (\mathcal{C}, T, \times, \pi, \Rightarrow, \mathit{ev})$ is given by: 
\begin{itemize}

\item (\textsc{Ct-Emp}) $\llbracket \mathsf{\vdash \diamondsuit \ ctx} \rrbracket \stackrel{\mathrm{df. }}{=} T$;

\item (\textsc{Ct-Ext}) $\llbracket \mathsf{\vdash \Gamma, x : A \ ctx} \rrbracket \stackrel{\mathrm{df. }}{=} \llbracket \mathsf{\vdash \Gamma \ ctx} \rrbracket \times \llbracket \mathsf{\vdash A \ type} \rrbracket$;

\item (\textsc{$\mathsf{1}$-Form}) $\llbracket \mathsf{\vdash \mathsf{1} \ type} \rrbracket \stackrel{\mathrm{df. }}{=} T$;

\item (\textsc{$\Rightarrow$-Form}) $\llbracket \mathsf{\vdash A \Rightarrow B \ type} \rrbracket \stackrel{\mathrm{df. }}{=} \llbracket \mathsf{\vdash A \ type} \rrbracket \Rightarrow \llbracket \mathsf{\vdash B \ type} \rrbracket$;

\item (\textsc{$\times$-Form}) $\llbracket \mathsf{\vdash A \times B \ type} \rrbracket \stackrel{\mathrm{df. }}{=} \llbracket \mathsf{\vdash A \ type} \rrbracket \times \llbracket \mathsf{\vdash B \ type} \rrbracket$;

\item (\textsc{Var}) $\llbracket \mathsf{x_1 : A_1, x_2 : A_2, \dots, x_j : A_j, \dots, x_n : A_n \vdash x_j : A_j} \rrbracket \stackrel{\mathrm{df. }}{=} \pi_j$;

\item (\textsc{$\mathsf{1}$-Intro}) $\llbracket \mathsf{\Gamma \vdash \star : \mathsf{1}} \rrbracket \stackrel{\mathrm{df. }}{=} \ \! !_{\llbracket \mathsf{\Gamma} \rrbracket}$;

\item (\textsc{$\Rightarrow$-Intro}) $\llbracket \mathsf{\Gamma \vdash \lambda x . \ \! b : A \Rightarrow B} \rrbracket \stackrel{\mathrm{df. }}{=} \lambda_{\llbracket \mathsf{A} \rrbracket, \llbracket \mathsf{B} \rrbracket} (\llbracket \mathsf{\Gamma, x : A \vdash b : B} \rrbracket)$;

\item (\textsc{$\times$-Intro}) $\llbracket \mathsf{\Gamma \vdash (a, b) : A \times B} \rrbracket \stackrel{\mathrm{df. }}{=} \langle \llbracket \mathsf{\Gamma \vdash a : A} \rrbracket, \llbracket \mathsf{\Gamma \vdash b : B} \rrbracket \rangle$;

\item (\textsc{$\Rightarrow$-Elim}) $\llbracket \mathsf{\Gamma \vdash f(a) : B} \rrbracket \stackrel{\mathrm{df. }}{=} \mathit{ev}_{\llbracket \mathsf{A} \rrbracket, \llbracket \mathsf{B} \rrbracket} \circ \langle \llbracket \mathsf{\Gamma \vdash f : A \Rightarrow B} \rrbracket, \llbracket \mathsf{\Gamma \vdash a : A} \rrbracket \rangle$; and

\item (\textsc{$\times$-Elim}) $\llbracket \mathsf{\Gamma \vdash \pi_i (p) : A_i} \rrbracket \stackrel{\mathrm{df. }}{=} \pi_i \circ \llbracket \mathsf{\Gamma \vdash p : A_1 \times A_2} \rrbracket$

\end{itemize}
where the hypotheses of the rules are as presented in Section~\ref{MLTT}, and $\pi_j$ is constructed from the first and second projections in the obvious manner.

\end{definition}

\begin{remark}
As each judgement in $\mathsf{\lambda^=_{1, \times}}$ has a \emph{unique} derivation, Definition~\ref{DefLambdaInCCCs} has no problem.
\end{remark}

In light of the rather simple relation between $\mathsf{MLTT(1, \Pi, \Sigma)}$ and $\mathsf{\lambda^=_{1, \times}}$, i.e., the restriction of the former to closed types is the latter, the gap between the interpretations of $\mathsf{\lambda^=_{1, \times}}$ in Definitions~\ref{DefMLTTInCwFs} and \ref{DefLambdaInCCCs} should not be the case.
We shall address this problem in Section~\ref{Equivalence}.

\section{Categories with Dependence}
\label{CwDs}
We have reviewed all the necessary backgrounds; the main contents of the paper begin from the present section.
It is structured roughly as follows.
First, Section~\ref{CwDs} introduces \emph{categories with dependence (CwDs)} and morphisms between CwDs, called \emph{functors with dependence (FwDs)}.
Then, Section~\ref{SCCwDs} defines CwDs with generalized finite products, called \emph{semi-cartesian CwDs (SCCwDs)}, and morphisms between them, called \emph{semi-cartesian FwDs (SCFwDs)}.
The semi-cartesian structure enables us to define morphisms between SCFwDs, called \emph{natural transformations with dependence (NTwDs)}, and moreover the 2-category of SCCwDs (0-cells), SCFwDs (1-cells) and NTwDs (2-cells).
Next, Section~\ref{CCwDs} gives a full generalization of finite products; we call CwDs with them \emph{cartesian CwDs (CCwDs)}. 
Finally Section~\ref{CCCwDs} defines a \emph{closed} structure on CCwDs; we call CCwDs with the closed structure \emph{cartesian closed CwDs (CCCwDs)}, where morphisms between these concepts, \emph{cartesian FwDs (CFwDs)} and \emph{cartesian closed FwDs (CCFwDs)}, are also defined. 
These concepts constitute the 2-category of CCCwDs (0-cells), CCFwDs (1-cells) and NTwDs (2-cells) between them.

Along with these developments, we show that:
\begin{itemize}

\item Strict SCCwDs coincide with CwFs;

\item Strict CCwDs induce CwFs that support $1$- and $\Sigma$-types in the strict sense; and

\item Strict CCCwDs induce CwFs that support $1$-, $\Sigma$- and $\Pi$-types in the strict sense.

\end{itemize}
Thus, in particular, we establish an interpretation of $\mathsf{MLTT}(1, \Pi, \Sigma)$ in CCCwDs; it is easy to see that this interpretation is sound and complete.

\subsection{Categories and Functors with Dependence}
\label{CwDs}
Let us begin with a very simple generalization of categories to capture type dependence in DTTs, called \emph{categories  with dependence (CwDs)}, and morphisms between them, called \emph{functors with dependence (FwDs)}.

\subsubsection{Categories with Dependence}
The idea of CwDs comes from CwFs: To take counterparts of types and terms as \emph{primitive}, rather than to induce or compose them, which stands in sharp contrast to other categorical semantics of DTTs mentioned in the introduction.
\begin{definition}[CwDs]
A \emph{\bfseries category with dependence (CwD)} is a category $\mathcal{C}$ equipped with:
\begin{itemize}

\item \emph{\bfseries Dependent (D-) objects} such that each D-object $A$ is assigned a unique object $\Gamma \in \mathcal{C}$, called the \emph{\bfseries base} of $A$, where $A$ is said to be \emph{\bfseries over} $\Gamma$, and the class of all D-objects over $\Gamma$ is written $\mathscr{D}_{\mathcal{C}}(\Gamma)$; 

\item \emph{\bfseries Dependent (D-) morphisms} such that each D-morphism $f$ is assigned a unique pair of an object $\Gamma \in \mathcal{C}$ and a D-object $A \in \mathscr{D}_{\mathcal{C}}(\Gamma)$, called the \emph{\bfseries domain} and the \emph{\bfseries codomain} of $f$, respectively, where the class of all D-morphisms whose domain and codomain are $\Gamma$ and $A$, respectively, are written $\mathscr{D}_{\mathcal{C}}(\Gamma, A)$; 


\item For each morphism $\phi : \Delta \to \Gamma$ in $\mathcal{C}$ a (class) function $\_ \{ \phi \}_{\mathcal{C}} : \mathscr{D}_{\mathcal{C}}(\Gamma) \to \mathscr{D}_{\mathcal{C}}(\Delta)$ and a family $(\_\{ \phi \}_{\mathcal{C}, A})_{A \in \mathscr{D}_{\mathcal{C}}(\Gamma)}$ of (class) functions $\_\{ \phi \}_{\mathcal{C}, A} : \mathscr{D}_{\mathcal{C}}(\Gamma, A) \to \mathscr{D}_{\mathcal{C}}(\Delta, A\{\phi\}_{\mathcal{C}})$, which are all called \emph{\bfseries dependent (D-) compositions}, that satisfy for any $\Theta, \Delta, \Gamma \in \mathcal{C}$, $A \in \mathscr{D}_{\mathcal{C}}(\Gamma)$, $\phi \in \mathcal{C}(\Delta, \Gamma), \varphi \in \mathcal{C}(\Theta, \Delta)$ and $a \in \mathscr{D}_{\mathcal{C}}(\Gamma, A)$ the equations $A \{ \mathit{id}_\Gamma \}_{\mathcal{C}} = A$, $a \{ \mathit{id}_\Gamma \}_{\mathcal{C}, A} = a$, $A \{ \phi \circ \varphi \}_{\mathcal{C}} = A\{ \phi \}_{\mathcal{C}} \{ \varphi \}_{\mathcal{C}}$ and $a \{ \phi \circ \varphi \}_{\mathcal{C}, A} = a \{ \phi \}_{\mathcal{C}, A} \{ \varphi \}_{\mathcal{C}, A\{\phi\}_{\mathcal{C}}}$.
\end{itemize}

\end{definition}

\begin{notation}
We often omit the subscript $\mathcal{C}$ on $\mathscr{D}_{\mathcal{C}}$ and $\_\{ \_ \}_{\mathcal{C}}$ as well as the subscripts $\mathcal{C}$ and/or $A$ on $\_\{ \_ \}_{\mathcal{C}, A}$ if it brings no ambiguity. 
We write $a : \Gamma \rightarrowtriangle A$ as a notation for $a \in \mathscr{D}(\Gamma, A)$.
Given a CwD $\mathcal{C}$, we write $\mathsf{ob}(\mathcal{C})$ or $\mathcal{C}_0$ (resp. $\mathsf{ar}(\mathcal{C})$ or $\mathcal{C}_1$, $\mathsf{dob}(\mathcal{C})$ or $\mathcal{C}_2$, $\mathsf{dar}(\mathcal{C})$ or $\mathcal{C}_3$) for the class of all objects in $\mathcal{C}$ (resp. all morphisms in $\mathcal{C}$, all D-objects in $\mathcal{C}$, all D-morphisms in $\mathcal{C}$). 
\end{notation}

\begin{remark}
CwDs are nothing new; a CwD is just a substructure of a CwF (Definition~\ref{DefCwFs}), consisting of the underlying category, types, terms and substitutions of the CwF, where we have assigned new names to the last three components.
\end{remark}

\begin{notation}
Let us write $\mathscr{U}(\mathcal{C})$ for the underlying category of a given CwD $\mathcal{C}$.
\end{notation}

In an MLTT, assuming that each type has at least one non-variable term, a type is \emph{constant} in the sense that it does not contain any free variable iff it is invariant (except its context) with respect to a generalized substitution.
Thus, it seems natural to define:
\begin{definition}[Constant CwDs]
Let $\mathcal{C}$ be a CwD.
A D-object $A \in \mathscr{D}_{\mathcal{C}}(\Gamma)$ is \emph{\bfseries constant} if $A \{ \phi_1 \}_{\mathcal{C}} = A \{ \phi_2 \}_{\mathcal{C}}$ for all $\Delta \in \mathcal{C}$ and $\phi_1, \phi_2 : \Delta \to \Gamma$ in $\mathcal{C}$.
A CwD is \emph{\bfseries constant} if its D-objects are all constant. 
\end{definition}

\begin{example}
Each category $\mathcal{C}$ can be seen as a constant CwD $\mathsf{D}(\mathcal{C})$ as follows.
We define $\mathscr{D}_{\mathsf{D}(\mathcal{C})}(\Gamma) \stackrel{\mathrm{df. }}{=} \mathsf{ob}(\mathcal{C})$ and $\mathscr{D}_{\mathsf{D}(\mathcal{C})}(\Delta, \Gamma) \stackrel{\mathrm{df. }}{=} \mathcal{C}(\Delta, \Gamma)$ for all $\Delta, \Gamma \in \mathcal{C}$.
Given $\phi : \Delta \to \Gamma$ in $\mathcal{C}$, we define the D-composition $\_ \{ \phi \} : \mathscr{D}_{\mathsf{D}(\mathcal{C})}(\Gamma) \to \mathscr{D}_{\mathsf{D}(\mathcal{C})}(\Delta)$ to be the identity function on $\mathsf{ob}(\mathcal{C})$, and the D-composition $\_ \{ \phi \}_\Theta : \mathscr{D}_{\mathsf{D}(\mathcal{C})}(\Gamma, \Theta) \to \mathscr{D}_{\mathsf{D}(\mathcal{C})}(\Delta, \Theta)$ for any $\Theta \in \mathscr{D}_{\mathsf{D}(\mathcal{C})}(\Gamma)$ to be the pre-composition function with $\phi$, i.e., $f \mapsto f \circ \phi$ for all $f \in \mathscr{D}_{\mathsf{D}(\mathcal{C})}(\Gamma, \Theta)$.
We call this kind of CwDs \emph{\bfseries categories seen as CwDs}.
Moreover, the operation $\mathscr{U}$, restricted to CwDs $\mathcal{D}$ such that $\mathscr{D}_{\mathcal{D}}(\Gamma) = \mathsf{ob}(\mathcal{D})$, $\mathscr{D}_{\mathcal{D}}(\Delta, \Gamma) = \mathcal{D}(\Delta, \Gamma)$, $\_ \{ \phi \} = \mathit{id}_{\mathsf{ob}(\mathcal{D})}$ and $\_ \{ \phi \}_\Theta = (\_) \circ \phi$ for all $\Delta, \Gamma, \Theta \in \mathcal{D}$ and $\phi : \Delta \to \Gamma$ in $\mathcal{D}$, is clearly the inverse of the operation $\mathsf{D}$.
Hence, CwDs are generalized categories.
\end{example}

\begin{remark}
It is also possible to regard each category $\mathcal{C}$ as a constant CwD $\mathsf{d}(\mathcal{C})$ by $\mathscr{D}_{\mathsf{d}(\mathcal{C})}(\Gamma) \stackrel{\mathrm{df. }}{=} \emptyset$, but it does not reflect our view on CwDs as generalized categories.
This point will be clearer shortly as we proceed. 
\end{remark}

Also, it seems natural to lift smallness and local smallness of categories as follows: 

\begin{definition}[Small CwDs]
A CwD $\mathcal{C}$ is \emph{\bfseries small} if so is the underlying category $\mathscr{U}(\mathcal{C})$, and the class of all D-objects in $\mathcal{C}$ and the class of all D-morphisms in $\mathcal{C}$ are both sets.
\end{definition}

\begin{definition}[Locally small CwDs]
A CwD $\mathcal{C}$ is \emph{\bfseries locally small} if so is the underlying category $\mathscr{U}(\mathcal{C})$, and the class $\mathscr{D}_{\mathcal{C}}(\Gamma, A)$ is a set for any $\Gamma \in \mathcal{C}$ and $A \in \mathscr{D}_{\mathcal{C}}(\Gamma)$.
\end{definition}

\begin{example}
The category $\mathit{Sets}$ of sets and functions gives rise to a non-constant locally small CwD.
A D-object over a set $X$ in $\mathit{Sets}$ is an indexed set $A = \{ A_x \in \mathit{Sets} \mid x \in X \ \! \}$, called a \emph{\bfseries dependent (D-) set} over $X$.
A D-morphism $X \rightarrowtriangle A$ in $\mathit{Sets}$ is a function $f : X \to \bigcup_{x \in X}A_x$ that satisfies $f(x) \in A_x$ for all $x \in X$, called a \emph{\bfseries dependent (D-) function} from $X$ to $A$.
Given a function $\phi : Y \to X$, the D-compositions in $\mathit{Sets}$ are given by $A \{ \phi \} \stackrel{\mathrm{df. }}{=} \{ A_{\phi (y)} \mid y \in Y \ \! \}$ and $f \{ \phi \} \stackrel{\mathrm{df. }}{=} f \circ \phi : (y \in Y) \mapsto f(\phi(y)) \in A_{\phi(y)} = (A\{ \phi \})_y$.
\end{example}

It is straightforward to define the \emph{substructure} relation between CwDs:
\begin{definition}[SubCwDs]
A CwD $\mathcal{C}$ is a \emph{\bfseries subCwD} of another $\mathcal{D}$ if $\mathscr{U}(\mathcal{C})$ is a subcategory of $\mathscr{U}(\mathcal{D})$, and the dependence of $\mathcal{C}$ is the corresponding part of the dependence of $\mathcal{D}$, i.e., for all $\Delta, \Gamma \in \mathcal{C}$, $A \in \mathscr{D}_{\mathcal{C}}(\Gamma)$, $f \in \mathscr{D}_{\mathcal{C}}(\Gamma, A)$ and $\phi : \Delta \to \Gamma$ in $\mathcal{C}$ we have $A \in \mathscr{D}_{\mathcal{D}}(\Gamma)$, $f \in \mathscr{D}_{\mathcal{D}}(\Gamma, A)$, $A\{\phi\}_{\mathcal{C}} = A\{\phi\}_{\mathcal{D}}$ and $f\{\phi\}_{\mathcal{C}} = f\{\phi\}_{\mathcal{D}}$.
\end{definition}

\begin{definition}[Full and wide subCwDs]
A CwD $\mathcal{C}$ is a \emph{\bfseries full subCwD} (resp. \emph{\bfseries wide subCwD}) of a CwD $\mathcal{D}$ if $\mathcal{C}$ is a subCwD of $\mathcal{D}$, $\mathscr{U}(\mathcal{C})$ is a full (resp. wide) subcategory of $\mathscr{U}(\mathcal{D})$, and $\mathscr{D}_{\mathcal{C}}(\Gamma, A) = \mathscr{D}_{\mathcal{D}}(\Gamma, A)$ for all $\Gamma \in \mathcal{C}$ and $A \in \mathscr{D}_{\mathcal{C}}(\Gamma)$ (resp. $\mathscr{D}_{\mathcal{C}}(\Gamma) = \mathscr{D}_{\mathcal{D}}(\Gamma)$ for all $\Gamma \in \mathcal{C}$).
\end{definition}

\begin{example}
Similarly to the CwD $\mathit{Sets}$, there is a non-constant locally small CwD $\mathit{Rel}$ whose underlying category is the category $\mathit{Rel}$ of sets and relations. 
D-objects in $\mathit{Rel}$ are D-sets, D-morphisms $X \rightarrowtriangle A$ in $\mathit{Rel}$ are relations $R \subseteq X \times \bigcup_{x \in X}A_x$ such that $(x, a) \in R \Rightarrow a \in A_x$, D-compositions in $\mathit{Rel}$ are given by $A\{\phi\} \stackrel{\mathrm{df. }}{=} \{ \bigcup_{(y, x) \in \phi} A_x \mid y \in Y \ \! \}$, i.e., $A\{\phi\}_y \stackrel{\mathrm{df. }}{=} \bigcup_{(y, x) \in \phi}A_x$ for all $y \in Y$ and $R\{\phi\} \stackrel{\mathrm{df. }}{=} R \circ \phi = \{ (y, a) \mid \exists x \in X . \ \! (y, x) \in \phi \wedge (x, a) \in R \ \! \}$ for any relation $\phi \subseteq Y \times X$.
It is straightforward to see that $\mathit{Sets}$ is a wide subCwD of $\mathit{Rel}$, lifting the relation between the categories $\mathit{Sets}$ and $\mathit{Rel}$.
\end{example}

\begin{example}
There is a non-constant locally small CwD $\mathit{Mon}$ whose underlying category is the category $\mathit{Mon}$ of monoids and monoid homomorphisms.
A D-object over a monoid $M = (|M|, \cdot_M, e_M)$ in $\mathit{Mon}$ is a triple $A = (|A|, \cdot_A, e_A)$ of an indexed set $|A| = (A_m)_{m \in |M|}$, a binary operation $\cdot_A$ on the union $\bigcup |A| \stackrel{\mathrm{df. }}{=}  \bigcup_{m \in |M|}A_m$ satisfying $a \cdot_A a' \in A_{m \cdot_M m'}$ for all $m, m' \in |M|$, $a \in A_m$ and $a \in A_{m'}$, and an element $e_A \in A_{e_M}$ such that the triple $\bigcup A = (\bigcup |A|, \cdot_A, e_A)$ forms a monoid, called a \emph{\bfseries dependent (D-) monoid} over $M$, and a D-morphism $M \rightarrowtriangle A$ in $\mathit{Mon}$ is a monoid homomorphism $f : M \to \bigcup A$ such that $f(m) \in A_m$ for all $m \in |M|$.
Given a monoid homomorphism $\phi : N \to M$, the D-composition $A\{\phi\} = (|A\{\phi\}|, \cdot_{A\{\phi\}}, e_{A\{\phi\}})$ is given by $A\{\phi\}_n \stackrel{\mathrm{df. }}{=} A_{\phi(n)}$ for each $n \in |N|$, $\cdot_{A\{\phi\}} \stackrel{\mathrm{df. }}{=} \cdot_A \upharpoonright \bigcup |A\{\phi\}|$ and $e_{A\{\phi\}} \stackrel{\mathrm{df. }}{=} e_A$, and the D-composition $f \{\phi\}$ is simply given by the function composition $f \circ \phi$.
\end{example}

One may wonder, as in the case of the CwDs $\mathit{Sets}$, $\mathit{Rel}$ and $\mathit{Mon}$, whether D-morphisms are always certain morphisms in the underlying category; however, it is not the case.
For instance, in the term model $\mathcal{T}(1, \Pi, \Sigma)$ regarded as a CwD, i.e., its types, terms and substitutions are regarded respectively as D-objects, D-morphisms and D-compositions, note that morphisms (or context morphisms) are certain \emph{lists} of D-morphisms (or terms).

There are other non-trivial CwDs whose D-morphisms are not certain morphisms: 
\begin{example}
There is a non-constant locally small CwD $\mathcal{GPD}$ whose underlying category is the category $\mathcal{GPD}$ of groupoids and functors between them \cite{hofmann1998groupoid}.
A D-object over a groupoid $G$ in $\mathcal{GPD}$ is any functor $A : G \to \mathcal{GPD}$, called a \emph{\bfseries dependent (D-) groupoid}. 
In order to define D-morphisms in $\mathcal{GPD}$, it is convenient (though not strictly necessary) to first employ the \emph{Grothendieck construction}  $G . A \in \mathcal{GPD}$ of $G$ and $A$, which is given by:
\begin{itemize}

\item An object of $G . A$ is a pair $(g, a)$ of objects $g \in G$ and $a \in A(g)$;

\item A morphism $(g, a) \to (g', a')$ in $G . A$ is a pair $(\phi, f)$ of a morphism $\phi : g \to g'$ in $G$ and a morphism $f : A(\phi)(a) \to a'$ in $A(g')$;

\item The composition $(g, a) \stackrel{(\phi, f)}{\to} (g', a') \stackrel{(\phi', f')}{\to} (g'', a'')$ in $G . A$ is the pair $(\phi' \circ \phi, f' \circ A(\phi')(f))$; 

\item The identity $\mathit{id}_{(g, a)} : (g, a) \to (g, a)$ in $G . A$ is the pair $(\mathit{id}_g, \mathit{id}_{a})$; and

\item The inverse of $(\phi, f) : (g, a) \to (g', a')$ is the pair $(\phi^{-1}, A(\phi^{-1})(f^{-1})) : (g', a') \to (g, a)$.

\end{itemize}
Then, a D-morphism $G \rightarrowtriangle A$ in $\mathcal{GPD}$ is a pair $\mu = (\mu_0, \mu_1)$ of maps $\mu_0 : \mathsf{ob}(G) \to \bigcup_{g \in G} \mathsf{ob}(A(g))$ and $\mu_1 : \mathsf{ar}(G) \to \bigcup_{g \in G} \mathsf{ar}(A(g))$ such that
\begin{align*}
(g, \mu_0(g)) &\in G . A \\
(\phi, \mu_1(\phi)) &\in G.A((g, \mu_0(g)), (g', \mu_0(g'))) \\
(\psi \circ \phi, \mu_1(\psi \circ \phi)) &= (\psi, \mu_1(\psi)) \circ (\phi, \mu_1(\phi)) \\
(\mathit{id}_g, \mu_1(\mathit{id}_g)) &= \mathit{id}_{(g, \mu_0(g))}
\end{align*}
for all $g, g', g'' \in G$, $\phi : g \to g'$ and $\psi : g' \to g''$ in $G$, i.e., the pairing $\langle \mathit{id}_G, \mu \rangle$ is a functor $G \to G.A$.
Note that it is impossible to turn $\mu$ into a morphism in $\mathcal{GPD}$ for there is no groupoid to serve as the codomain of $\mu$; we have employed $G . A$ to make this point clear.
Finally, given a morphism $F : H \to G$ in $\mathcal{GPD}$, the D-composition $A\{F\} : H \to \mathcal{GPD}$ is the composition $A \circ F$ of functors, and the D-composition $\mu\{F\} : H \rightarrowtriangle A\{F\}$ is given by $\mu\{F\}_i = \mu_i \circ F_i$ for $i = 1, 2$.
\end{example}

\begin{example}
\label{ExCAT}
Similarly to $\mathcal{GPD}$, the category $\mathit{CAT}$ of small categories and functors between them forms a non-constant locally small CwD.
A D-object over a category $\mathcal{C}$ in $\mathit{CAT}$ is a functor $\mathcal{X} : \mathcal{C} \to \mathit{CAT}$, called a \emph{\bfseries dependent (D-) category}, and a D-morphism $\mathcal{C} \rightarrowtriangle \mathcal{X}$ in $\mathit{CAT}$ is a pair $F = (F_0, F_1)$ of maps $F_0 : \mathsf{ob}(\mathcal{C}) \to \bigcup_{\Gamma \in \mathcal{C}} \mathsf{ob}(\mathcal{X}(\Gamma))$ and $F_1 : \mathsf{ar}(\mathcal{C}) \to \bigcup_{\Gamma \in \mathcal{C}} \mathsf{ar}(\mathcal{X}(\Gamma))$ such that
\begin{align*}
(\Gamma, F_0(\Gamma)) &\in \mathcal{C} . \mathcal{X} \\
(\phi, F_1(\phi) &\in \mathcal{C} . \mathcal{X}((\Delta, F_0(\Delta)), (\Gamma, F_0(\Gamma))) \\
(\psi \circ \phi, F_1(\psi \circ \phi)) &= (\psi, F_1(\psi)) \circ (\phi, F_1(\phi)) \\
(\mathit{id}_\Gamma, F_1(\mathit{id}_\Gamma)) &= \mathit{id}_{(\Gamma, F_0(\Gamma))}
\end{align*}
for all $\Delta, \Gamma, \Theta \in \mathcal{C}$, $\phi : \Delta \to \Gamma$ and $\psi : \Gamma \to \Theta$ in $\mathcal{C}$, called a \emph{\bfseries dependent (D-) functor}, where the category $\mathcal{C} . \mathcal{X}$ is the Grothendieck construction of $\mathcal{C}$ and $\mathcal{X}$ (discarding the structure of inverses).
Finally, D-compositions in $\mathit{CAT}$ are defined exactly as in the case of $\mathcal{GPD}$. 
\end{example}

Now, one may wonder what is an appropriate notion of \emph{isomorphisms between D-objects} for it is a fundamental principle in category theory to regard isomorphic objects as essentially the same.
Thinking of the CwD $\mathit{Sets}$ as an example, it sounds fair to say that two D-sets $A$ and $A'$ are isomorphic if they are over sets $X$ and $X'$, respectively, such that there is a bijection $\iota : X \stackrel{\sim}{\to} X'$, and $A_x$ is isomorphic to $A'_{\iota(x)}$ for each $x \in X$ because then the structures of $A$ and $A'$ as D-sets are essentially the same.
However, the second condition does not make sense for an arbitrary CwD (i.e., we cannot always talk about the components $A_x$).
We need additional conditions on CwDs, which will be introduced in Section~\ref{SCCwDs}, to talk about isomorphisms between D-objects.

Next, note that the asymmetry of the domain and the codomain of D-morphisms prohibits us from defining the `opposite CwDs' just by `flipping' morphisms and D-morphisms in a given CwD.
However, we may `flip' only morphisms, not D-morphisms, taking this operation as the `opposite' construction on CwDs.
Thus, we do not need any new concepts for this construction; it suffices to consider CwDs whose underlying categories are opposite categories. 

Nevertheless, it is often convenient to give the explicit definition:
\begin{definition}[Contravariant CwDs]
\label{DefContravariantCwDs}
A \emph{\bfseries contravariant CwD} is a category $\mathcal{C}$ equipped with a pair $(\mathscr{D}_{\mathcal{C}}, \_\{ \_ \}_{\mathcal{C}})$, called the \emph{\bfseries dependence} of $\mathcal{C}$, where:
\begin{itemize}

\item $\mathscr{D}_{\mathcal{C}}$ assigns to each object $\Gamma \in \mathcal{C}$ a class $\mathscr{D}_{\mathcal{C}}(\Gamma)$ of \emph{\bfseries dependent (D-) objects} over $\Gamma$ and to each pair of $\Gamma \in \mathcal{C}$ and $A \in \mathscr{D}_{\mathcal{C}}(\Gamma)$ a class $\mathscr{D}_{\mathcal{C}}(\Gamma, A)$ of \emph{\bfseries dependent (D-) morphisms} from $\Gamma$ to $A$;

\item $\_ \{ \_ \}_{\mathcal{C}}$ assigns to each morphism $\phi : \Delta \to \Gamma$ in $\mathcal{C}$ a (class) function $\_ \{ \phi \}_{\mathcal{C}} : \mathscr{D}_{\mathcal{C}}(\Delta) \to \mathscr{D}_{\mathcal{C}}(\Gamma)$ and a family $(\_\{ \phi \}_{\mathcal{C}, A})_{A \in \mathscr{D}_{\mathcal{C}}(\Gamma)}$ of (class) functions $\_\{ \phi \}_{\mathcal{C}, A} : \mathscr{D}_{\mathcal{C}}(\Delta, B) \to \mathscr{D}_{\mathcal{C}}(\Gamma, B\{\phi\}_{\mathcal{C}})$, which are all called \emph{\bfseries contravariant dependent (D-) compositions} that satisfy:
\begin{align}
B \{ \mathit{id}_\Delta \}_{\mathcal{C}} &= B \\
b \{ \mathit{id}_\Delta \}_{\mathcal{C}, B} &= b \\
B \{ \psi \circ \phi \}_{\mathcal{C}} &= B\{ \phi \}_{\mathcal{C}} \{ \psi \}_{\mathcal{C}} \\
b \{ \psi \circ \phi \}_{\mathcal{C}, B} &= b \{ \phi \}_{\mathcal{C}, B} \{ \psi \}_{\mathcal{C}, B\{\phi\}_{\mathcal{C}}}
\end{align}
for any $\Theta, \Delta, \Gamma \in \mathcal{C}$, $B \in \mathscr{D}_{\mathcal{C}}(\Delta)$, $\phi \in \mathcal{C}(\Delta, \Gamma), \psi \in \mathcal{C}(\Gamma, \Theta)$ and $b \in \mathscr{D}_{\mathcal{C}}(\Delta, B)$.
\end{itemize}
\end{definition}

\begin{notation}
We write $\mathcal{C}^{\mathsf{op}}$ for the opposite category of a given category $\mathcal{C}$, and $\phi^{\mathsf{op}} : \Gamma \to \Delta$ in $\mathcal{C}^{\mathsf{op}}$ for the `flipped' morphism $\phi : \Delta \to \Gamma$ in $\mathcal{C}$.
\end{notation}

\begin{remark}
A contravariant CwD $\mathcal{C}$ is nothing other than a CwD $\mathcal{C}^{\mathsf{op}}$ described in terms of the category $\mathcal{C}$.
\end{remark}

\begin{notation}
The notations for CwDs are also applied to contravariant CwDs. 
\end{notation}

We are now ready to define:
\begin{definition}[Opposites]
Given a CwD $\mathcal{C}$, its \emph{\bfseries opposite} $\mathcal{C}^{\mathsf{op}}$ is the contravariant CwD such that $\mathscr{U}(\mathcal{C}^{\mathsf{op}}) \stackrel{\mathrm{df. }}{=} \mathscr{U}(\mathcal{C})^{\mathsf{op}}$, $\mathscr{D}_{\mathcal{C}^{\mathsf{op}}} \stackrel{\mathrm{df. }}{=} \mathscr{D}_{\mathcal{C}}$ and $\_\{ \_ \}_{\mathcal{C}^{\mathsf{op}}} \stackrel{\mathrm{df. }}{=} \_ \{ \_ \}_{\mathcal{C}} \circ (\_)^{\mathsf{op}}$.
Dually, the \emph{\bfseries opposite} $\mathcal{D}^{\mathsf{op}}$ of a contravariant CwD $\mathcal{D}$ is the CwD defined in the same manner.
\end{definition}

For a contravariant CwD $\mathcal{C}$ is just a CwD $\mathcal{C}^{\mathsf{op}}$, by the \emph{principle of duality} \cite{mac2013categories,awodey2010category}, it suffices to focus only on (\emph{covariant}) CwDs.
However, just for convenience, we occasionally describe contravariant ones $\mathcal{C}^{\mathsf{op}}$ explicitly in terms of $\mathcal{C}$ like Definition~\ref{DefContravariantCwDs}.

\subsubsection{Functors with Dependence}
Next, let us generalize functors by taking into account D-objects and D-morphisms:
\begin{definition}[FwDs]
A \emph{\bfseries functor with dependence (FwD)} between CwDs $\mathcal{C}$ and $\mathcal{D}$ is a quadruple $F = (F_0, F_1, F_2, F_3)$, written $F : \mathcal{C} \to \mathcal{D}$, of maps $F_i : \mathcal{C}_i \to \mathcal{D}_i$ that satisfies:
\begin{enumerate}

\item $\forall \Delta, \Gamma \in \mathcal{C}, \phi \in \mathcal{C}(\Delta, \Gamma) . \ \! F_1(\phi) \in \mathcal{D}(F_0(\Delta), F_0(\Gamma))$;

\item $\forall \Gamma \in \mathcal{C}, A \in \mathscr{D}_{\mathcal{C}}(\Gamma) . \ \! F_2(A) \in \mathscr{D}_{\mathcal{D}}(F_0(\Gamma))$;

\item $\forall \Gamma \in \mathcal{C}, A \in \mathscr{D}_{\mathcal{C}}(\Gamma), a \in \mathscr{D}_{\mathcal{C}}(\Gamma, A) . \ \! F_3(a) \in \mathscr{D}_{\mathcal{D}}(F_0(\Gamma), F_2(A))$;

\item $\forall \Delta, \Gamma, \Theta \in \mathcal{C}, \phi \in \mathcal{C}(\Delta, \Gamma), \psi \in \mathcal{C}(\Gamma, \Theta) . \ \! F_1 (\psi \circ \phi) = F_1(\psi) \circ F_1(\phi)$;

\item $\forall \Gamma \in \mathcal{C} . \ \! F_1(\mathit{id}_\Gamma) = \mathit{id}_{F_0(\Gamma)}$;

\item $\forall \Delta, \Gamma \in \mathcal{C}, \phi \in \mathcal{C}(\Delta, \Gamma), A \in \mathscr{D}_{\mathcal{C}}(\Gamma) . \ \! F_2(A \{ \phi \}_{\mathcal{C}}) = F_2(A) \{ F_1(\phi) \}_{\mathcal{D}}$; and

\item $\forall \Delta, \Gamma \in \mathcal{C}, \phi \in \mathcal{C}(\Delta, \Gamma), A \in \mathscr{D}_{\mathcal{C}}(\Gamma), a \in \mathscr{D}_{\mathcal{C}}(\Gamma, A) . \ \! F_3(a \{ \phi \}_{\mathcal{C}}) = F_3(a)\{F_1(\phi)\}_{\mathcal{D}}$.

\end{enumerate}
\end{definition}

\begin{convention}
We usually abbreviate each $F_i$ ($i = 0, 1, 2, 3$) as $F$, which would not bring any confusion in practice.
\end{convention}

That is, an FwD $F : \mathcal{C} \to \mathcal{D}$ is just a functor $\mathscr{U}(F) \stackrel{\mathrm{df. }}{=} (F_0, F_1) : \mathscr{U}(\mathcal{C}) \to \mathscr{U}(\mathcal{D})$ equipped with maps $F_2$ and $F_3$ on D-objects and D-morphisms, respectively, with analogous structure-preserving properties.

\begin{example}
Given a functor $F : \mathcal{C} \to \mathcal{D}$, there are the obvious FwDs $\mathsf{D}(F) : \mathsf{D}(\mathcal{C}) \to \mathsf{D}(\mathcal{D})$.
We call this kind of FwDs \emph{\bfseries functors seen as FwDs}.
Conversely, FwDs between categories seen as CwDs coincide with functors.
In this sense, FwDs are a generalization of functors.
Note also that there are the trivial FwDs $\mathsf{d}(F) : \mathsf{d}(\mathcal{C}) \to \mathsf{d}(\mathcal{D})$.
\end{example}

\begin{example}
There is an FwD $(\_)^\ast : \mathit{Sets} \to \mathit{Mon}$ that maps:
\begin{align*}
X \in \mathit{Sets} &\mapsto X^\ast \in \mathit{Mon} \\
\phi \in \mathit{Sets}(X, Y) &\mapsto \phi^\ast \in \mathit{Mon}(X^\ast, Y^\ast) \\
A \in \mathscr{D}_{\mathit{Sets}}(X) &\mapsto A^\ast =( \{ \textstyle \prod_{i = 1}^{k} A_{x_i} \ \! | \ \! x_1 x_2 \dots x_k \in X^\ast \}, \ast, \bm{\epsilon})\in \mathscr{D}_{\mathit{Mon}}(X^\ast) \\
f \in \mathscr{D}_{\mathit{Sets}}(X, A) &\mapsto f^\ast \in \mathscr{D}_{\mathit{Mon}}(X^\ast, A^\ast)
\end{align*}
where $X^\ast = (X^\ast, \ast, \bm{\epsilon})$ is the \emph{free monoid} over $X$, $\phi^\ast$ (resp. $f^\ast$) maps each $x_1 x_2 \dots x_k \in X^\ast$ to $\phi(x_1) \phi(x_2) \dots \phi(x_k) \in Y^\ast$ (resp. to $f(x_1) f(x_2) \dots f(x_k) \in \prod_{i = 1}^{k} A_{x_i}$), and $\prod_{i = 1}^{k} A_{x_i}$ is the product of monoids $A_{x_i}$ for $i = 1, 2, \dots, k$. 
\end{example}

\begin{definition}[Fullness and faithfulness of FwDs]
An FwD $F : \mathcal{C} \to \mathcal{D}$ is \emph{\bfseries full} (resp. \emph{\bfseries faithful}) if so is the underlying functor $\mathscr{U}(F)$, and the function $F_{\Gamma, A} : \mathscr{D}_{\mathcal{C}}(\Gamma, A) \to \mathscr{D}_{\mathcal{D}}(F(\Gamma), F(A))$ that maps $f \mapsto F(f)$ for all $f \in \mathscr{D}_{\mathcal{C}}(\Gamma, A)$ is surjective (resp. injective) for each pair of an object $\Gamma \in \mathcal{C}$ and a D-object $A \in \mathscr{D}_{\mathcal{C}}(\Gamma)$ in $\mathcal{C}$.
\end{definition}

Just for convenience, let us write down explicitly the contravariant case:
\begin{definition}[Contravariant FwDs]
A \emph{\bfseries contravariant FwD} from a contravariant CwD $\mathcal{C}$ to a CwD $\mathcal{D}$ is an FwD $F : \mathcal{C}^{\mathsf{op}} \to \mathcal{D}$.
Dually, a contravariant FwD from $\mathcal{D}$ to $\mathcal{C}$ is a contravariant FwD $\mathcal{D}^{\mathsf{op}} \to \mathcal{C}^{\mathsf{op}}$, i.e., an FwD $\mathcal{D} \to \mathcal{C}^{\mathsf{op}}$. 
\end{definition}

\begin{example}
Given a locally small CwD $\mathcal{C}$ and an object $\Delta \in \mathcal{C}$, we obtain an FwD $\mathcal{C}(\Delta, \_) : \mathcal{C} \to \mathit{Sets}$ that maps each object $\Gamma \in \mathcal{C}$ to the set $\mathcal{C}(\Delta, \Gamma)$, each morphism $\alpha \in \mathcal{C}(\Gamma, \Gamma')$ to the function $\mathcal{C}(\Delta, \alpha) : \mathcal{C}(\Delta, \Gamma) \to \mathcal{C}(\Delta, \Gamma')$ that maps $(\phi \in \mathcal{C}(\Delta, \Gamma)) \mapsto \alpha \circ \phi \in \mathcal{C}(\Delta, \Gamma')$, each D-object $A \in \mathscr{D}_{\mathcal{C}}(\Gamma)$ to the D-set $\mathcal{C}(\Delta, A) \stackrel{\mathrm{df. }}{=} (\mathscr{D}_{\mathcal{C}}(\Delta, A\{ \phi \}_{\mathcal{C}}))_{\phi \in \mathcal{C}(\Delta, \Gamma)}$, and each D-morphism $f \in \mathscr{D}_{\mathcal{C}}(\Gamma, A)$ to the D-function $\mathcal{C}(\Delta, f) : \mathcal{C}(\Delta, \Gamma) \rightarrowtriangle \mathcal{C}(\Delta, A)$ that maps $(\phi \in \mathcal{C}(\Delta, \Gamma)) \mapsto f \{ \phi \}_{\mathcal{C}} \in \mathscr{D}_{\mathcal{C}}(\Delta, A\{\phi\}_{\mathcal{C}}) = \mathcal{C}(\Delta, A)_\phi$.

Dually, given a locally small contravariant CwD $\mathcal{D}$ and an object $\Delta \in \mathcal{D}$, we obtain a contravariant FwD $\mathcal{D}(\_, \Delta) : \mathcal{D} \to \mathit{Sets}$ that maps each object $\Gamma \in \mathcal{D}$ to the set $\mathcal{D}(\Gamma, \Delta)$, each morphism $\beta \in \mathcal{D}(\Gamma', \Gamma$) to the function $\mathcal{D}(\beta, \Delta) : \mathcal{D}(\Gamma, \Delta) \to \mathcal{D}(\Gamma', \Delta)$ that maps $(\varphi \in \mathcal{D}(\Gamma, \Delta)) \mapsto \varphi \circ \beta \in \mathcal{D}(\Gamma', \Delta)$, each D-object $A \in \mathscr{D}_{\mathcal{D}}(\Gamma)$ to the D-set $\mathcal{D}(A, \Delta) \stackrel{\mathrm{df. }}{=} (\mathscr{D}_{\mathcal{D}}(\Delta, A\{ \varphi \}_{\mathcal{D}}))_{\varphi \in \mathcal{D}(\Gamma, \Delta)}$, and each D-morphism $f \in \mathscr{D}_{\mathcal{D}}(\Gamma, A)$ to the D-function $\mathcal{D}(f, \Delta) : \mathcal{D}(\Gamma, \Delta) \rightarrowtriangle \mathcal{D}(A, \Delta)$ that maps $(\varphi \in \mathcal{D}(\Gamma, \Delta)) \mapsto f\{\varphi\}_{\mathcal{D}} \in \mathscr{D}_{\mathcal{D}}(\Delta, A\{\varphi\}_{\mathcal{D}}) = \mathcal{D}(A, \Delta)_\varphi$.
\end{example}

As expected, small CwDs and FwDs between them form a category:
\begin{definition}[The category $\mathbb{C}_{\mathbb{D}}$]
\label{TheCategoryCwDs}
The category $\mathbb{C}_{\mathbb{D}}$ is given by:
\begin{itemize}

\item Objects are small CwDs;

\item Morphism $\mathcal{C} \to \mathcal{D}$ are FwDs $\mathcal{C} \to \mathcal{D}$;

\item The composition $\mathcal{C} \stackrel{F}{\to} \mathcal{D} \stackrel{G}{\to} \mathcal{E}$ is the quadruple $(G_0 \circ F_0, G_1 \circ F_1, G_2 \circ F_2, G_3 \circ F_3)$; and

\item The identity $\mathit{id}_{\mathcal{C}} : \mathcal{C} \to \mathcal{C}$ is the quadruple $(\mathit{id}_{\mathcal{C}_0}, \mathit{id}_{\mathcal{C}_1}, \mathit{id}_{\mathcal{C}_2}, \mathit{id}_{\mathcal{C}_3})$.

\end{itemize}
\end{definition}

It is easy to check that $\mathbb{C}_{\mathbb{D}}$ is a well-defined locally small category, and there is the obvious forgetful functor $\mathscr{U} : \mathbb{C}_{\mathbb{D}} \to \mathit{CAT}$, which has the functor $\mathsf{d} : \mathit{CAT} \to \mathbb{C}_{\mathbb{D}}$ as the left adjoint. 

On the other hand, there seems no sensible way to lift the category $\mathbb{C}_{\mathbb{D}}$ to a CwD.
One may define a D-object over a small CwD $\mathcal{C}$ to be an FwD $\mathcal{C} \to \mathit{CAT}$ or a functor $\mathscr{U}(\mathcal{C}) \to \mathbb{C_D}$, but neither works very well.

\subsection{Semi-Cartesian Categories, Functors and Natural Transformations with Dependence}
\label{SCCwDs}
As we have seen above, CwDs and FwDs are natural generalizations of categories and functors, respectively, which subsume several concrete instances.
However, we cannot yet have morphisms between parallel FwDs because we cannot define \emph{morphisms between D-objects} at the moment.
Thus, CwDs and FwDs are in some sense incomplete structures.

On the other hand, we need additional structures on CwDs in order to interpret the \emph{empty context} Ctx-Emp and \emph{context extensions} Ctx-Ext in MLTTs.

The present section addresses these two problems by introducing CwDs with generalized binary products, called \emph{semi-dependent pair (semi-$\Sigma$-) spaces}.
We call CwDs with a terminal object and semi-$\Sigma$-spaces \emph{semi-cartesian CwDs (SCCwDs)}; their strict version coincides with CwFs, and thus they interpret the empty context and context extensions. 
We also define FwDs that preserve the semi-cartesian structure of CwDs, called \emph{semi-cartesian FwDs (SCFwDs)} and morphisms between parallel SCFwDs, called \emph{natural transformations with dependence (NTwDs)}, which constitute 1-cells and 2-cells of the 2-category of SCCwDs (0-cells), respectively.  
Moreover, as promised before, we lift this 2-category to a \emph{2-CwD}, the 2-dimensional generalization of CwDs.

\subsubsection{Semi-Cartesian Categories with Dependence}
\begin{definition}[Semi-$\Sigma$-spaces]
\label{DefSemiDependentPairSpaces}
Let $\mathcal{C}$ be a CwD.
A \emph{\bfseries semi-dependent pair (semi-$\bm{\Sigma}$-) space} of an object $\Gamma \in \mathcal{C}$ and a D-object $A \in \mathscr{D}(\Gamma)$ is an object 
\begin{equation*}
\Gamma . A \in \mathcal{C}
\end{equation*}
together with a morphism 
\begin{equation*}
\pi_1^{\Gamma . A} : \Gamma . A \to \Gamma
\end{equation*}
in $\mathcal{C}$, called the \emph{\bfseries first projection} on $\Gamma . A$, and a D-morphism 
\begin{equation*}
\pi_2^{\Gamma . A} : \Gamma . A \rightarrowtriangle A\{ \pi_1^{\Gamma . A} \}
\end{equation*}
in $\mathcal{C}$, called the \emph{\bfseries second projection} on $\Gamma . A$, such that for any $\Delta \in \mathcal{C}$, $\phi : \Delta \to \Gamma$ and $g : \Delta \rightarrowtriangle A\{\phi\}$ in $\mathcal{C}$ there exists a unique morphism 
\begin{equation*}
\langle \phi, g \rangle : \Delta \to \Gamma . A
\end{equation*}
in $\mathcal{C}$, called the \emph{\bfseries semi-dependent pairing} of $\phi$ and $g$, that satisfies the equations 
\begin{align*}
\pi_1^{\Gamma . A} \circ \langle \phi, g \rangle &= \phi \\
\pi_2^{\Gamma . A} \{ \langle \phi, g \rangle \} &= g.
\end{align*}
\end{definition}

\begin{notation}
We often omit the superscript $\Gamma . A$ on $\pi_i^{\Gamma . A}$ ($i = 1, 2$).
\end{notation}

\begin{definition}[SCCwDs]
\label{DefSCDCs}
A \emph{\bfseries semi-cartesian CwD (SCCwD)} is a CwD $\mathcal{C}$ that has:
\begin{itemize}

\item A terminal object $T \in \mathcal{C}$; and 

\item A semi-$\Sigma$-pair space $\Gamma \stackrel{\pi_1}{\leftarrow} \Gamma . A \stackrel{\pi_2}{\rightarrowtriangle} A\{\pi_1\}$ of any $\Gamma \in \mathcal{C}$ and $A \in \mathscr{D}_{\mathcal{C}}(\Gamma)$.

\end{itemize}
\end{definition}

See the following diagram that depicts the universal property of semi-$\Sigma$-spaces:  
\begin{diagram}
\label{DPS}
\Gamma & &\lTo^\phi & \Delta &  \rDepTo^{g}  & & A \{ \phi \} \\
\uEquals & & & \dDotsTo_{\langle \phi, g \rangle} & & & \uImplies \\
\Gamma & & \lTo^{\pi_1} & \Gamma . A & \rDepTo^{\pi_2} & & \textstyle A\{ \pi_1 \}
\end{diagram}
where the thick arrow $A \{ \pi_1 \} \Rightarrow A \{ \phi \}$ denotes the transformation of the codomain of $\pi_2$ when $\langle \phi, g \rangle$ is composed with $\pi_2$.
Clearly, it generalizes the universal property of binary products in the sense that for a category seen as a CwD the diagram coincides with that of binary products; in other words, for a category seen as a CwD, semi-$\Sigma$-spaces $\Gamma . A$ are just binary products $\Gamma \times A$.
In this sense, SCCwDs are a generalization of cartesian categories. 

On the other hand, unlike binary products, we do not even have a sensible way to formulate, let alone prove, that a terminal object is unit for semi-$\Sigma$-spaces, or semi-$\Sigma$-spaces are associative, due to the asymmetry of semi-$\Sigma$-spaces.
We need \emph{unit D-objects} and \emph{$\Sigma$-spaces} in Section~\ref{CCwDs} to address this point.

Similarly to the case of binary products, we may easily establish:
\begin{proposition}[Uniqueness of semi-$\Sigma$-spaces]
Semi-$\Sigma$-spaces are unique up to isomorphisms. 
\end{proposition}

\begin{lemma}[Semi-dependent pairings as comprehensions]
\label{LemSemiSigmaPairings}
Let $\mathcal{C}$ be an SCCwD, and $\Theta, \Delta, \Gamma \in \mathcal{C}$, $A \in \mathscr{D}(\Gamma)$, $\varphi : \Theta \to \Delta$, $\phi : \Delta \to \Gamma$ and $g : \Delta \rightarrowtriangle A\{\phi\}$ in $\mathcal{C}$. 
Then, we have:
\begin{align}
\label{ConsNat}
\langle \phi, g \rangle \circ \varphi &= \langle \phi \circ \varphi, g \{ \varphi \} \rangle \\
\label{ConsId}
\langle \pi_1^{\Gamma . A}, \pi_2^{\Gamma . A} \rangle &= \mathit{id}_{\Gamma . A}.
\end{align}
\end{lemma}

Note that the equations (\ref{ConsNat}) and (\ref{ConsId}) correspond respectively to the axioms Cons-Nat and Cons-Id for CwFs, and thus Lemma~\ref{LemSemiSigmaPairings} particularly implies that an SCCwD is a CwF except that a terminal object, comprehensions, projections and extensions are not parts of the structure, and comprehensions and extensions are called semi-$\Sigma$-spaces and semi-dependent pairings, respectively.
Hence, CwFs are the strict version of SCCwDs.

Nevertheless, we require some equation for our notion of \emph{strict} SCCwDs:
\begin{definition}[Strict SCCwDs]
A \emph{\bfseries strict SCCwD} is an SCCwD $\mathcal{C}$ equipped with a terminal object $T \in \mathcal{C}$ and a semi-$\Sigma$-space $\Gamma \stackrel{\pi_1}{\leftarrow} \Gamma . A \stackrel{\pi_2}{\rightarrowtriangle} A\{\pi_1\}$ that satisfies:  
\begin{equation}
\label{EqStrictSCCwDs}
\pi_2^{\Gamma . A} \{ \pi_1^{\Gamma . A . A\{\pi_1^{\Gamma . A}\}} \} = \pi_2^{\Gamma . A . A\{\pi_1^{\Gamma . A}\}}
\end{equation}
for each $\Gamma \in \mathcal{C}$ and $A \in \mathscr{D}_{\mathcal{C}}(\Gamma)$.
\end{definition}

We shall need this equation for the proof of Theorem~\ref{ThmPi}.

\begin{example}
\label{ExCCsAsDSCCs}
A category $\mathcal{C}$ seen as a CwD $\mathsf{D}(\mathcal{C})$ is semi-cartesian iff $\mathcal{C}$ is cartesian.
We call this kind of SCCwDs \emph{\bfseries cartesian categories seen as SCCwDs}.
\end{example}

\begin{example}
\label{ExDCCSets}
The CwD $\mathit{Sets}$ is semi-cartesian. 
Any singleton set $\{ \bullet \}$ is a terminal object, and given a set $X$ and an indexed set $A$ over $X$ the set $\{ (x, a) \mid x \in X, a \in A_x \ \! \}$ of \emph{\bfseries semi-dependent pairs} $(x, a)$ is a semi-$\Sigma$-space $X . A$ together with the projection functions as the projections. 
Semi-dependent pairings are just pairings of functions. 
\end{example}

\begin{example}
The CwD $\mathit{Rel}$ is semi-cartesian.
The empty set $\emptyset$ is a terminal (as well as initial) object, and given a set $X$ and an indexed set $A$ over $X$ the disjoint union $X + \bigcup_{x \in X}A_x$ is a semi-$\Sigma$-space $X . A$ together with the first projection $\pi_1 \stackrel{\mathrm{df. }}{=} \{ (\iota_1(x), x) \mid x \in X \ \! \} \subseteq X.A \times X$ and the second projection $\pi_2 \stackrel{\mathrm{df. }}{=} \{ (\iota_2(a), a) \mid a \in \bigcup_{x \in X}A_x \ \! \} \subseteq X.A \times \bigcup_{z \in X.A}A\{\pi_1\}_z$, where $\iota_i$ ($i = 1, 2$) are the `tags' for the disjoint union $X.A$.
Given a set $Y$, a relation $R \subseteq Y \times X$ and a D-relation $S \subseteq Y \times \bigcup_{y \in Y}A\{R\}_y$, the semi-dependent pairing of $R$ and $S$ is given by $\langle R, S \rangle \stackrel{\mathrm{df. }}{=} \{ (y, \iota_1(x)) \mid (y, x) \in R \ \! \} \cup \{ (y, \iota_2(a)) \mid (y, a) \in S \ \! \}$.
\end{example}

\begin{example}
The CwD $\mathit{Mon}$ is semi-cartesian.
A terminal object in $\mathit{Mon}$ or a \emph{\bfseries terminal monoid} is any monoid $T$ whose underlying set is singleton.
Given a monoid $M$ and a D-monoid $A$ over $M$, the semi-$\Sigma$-space $M . A$ is given by $|M.A| \stackrel{\mathrm{df. }}{=} \{ (m, a) \mid m \in |M|, a \in A_m \ \! \}$, $\cdot_{M.A} : ((m, a), (m', a')) \mapsto (m \cdot_M m', a \cdot_A a')$ and $e_{M.A} \stackrel{\mathrm{df. }}{=} (e_M, e_A)$, where projections and semi-dependent pairings are the obvious ones as in $\mathit{Sets}$.
\end{example}

\begin{example}
The CwD $\mathcal{GPD}$ is semi-cartesian.
A terminal object in $\mathcal{GPD}$ is any groupoid $T$ with just one object $\star$ and the identity $\mathit{id}_\star$.
Given a groupoid $G$ and a D-groupoid $A : G \to \mathcal{GPD}$, the semi-$\Sigma$-space $G . A$ is the Grothendieck construction of $G$ and $A$, where projections and semi-dependent pairings are the obvious ones. 
\end{example}

\begin{example}
The CwD $\mathit{CAT}$ is semi-cartesian in the same manner as the case of $\mathcal{GPD}$.
\end{example}

Just for clarity, we explicitly write down the contravariant case because the directions of morphisms in this case are slightly trickier:
\begin{definition}[Contravariant semi-$\Sigma$-spaces]
\label{DefContravariantSemiDependentPairSpaces}
Let $\mathcal{C}$ be a contravariant CwD.
A \emph{\bfseries contravariant semi-$\bm{\Sigma}$-space} of an object $\Gamma \in \mathcal{C}$ and a D-object $A \in \mathscr{D}_{\mathcal{C}}(\Gamma)$ is an object $\Gamma . A \in \mathcal{C}$ together with a morphism $\pi_1 : \Gamma \to \Gamma . A$ in $\mathcal{C}$, called the \emph{\bfseries contravariant first projection} on $\Gamma . A$, and a D-morphism $\pi_2 : \Gamma . A \rightarrowtriangle A\{ \pi_1 \}$ in $\mathcal{C}$, called the \emph{\bfseries contravariant second projection} on $\Gamma . A$, such that for any $\Delta \in \mathcal{C}$, $\varphi : \Gamma \to \Delta$ and $g : \Delta \rightarrowtriangle A\{\varphi\}$ in $\mathcal{C}$ there exists a unique morphism $\langle \varphi, g \rangle : \Gamma . A \to \Delta$ in $\mathcal{C}$, called the \emph{\bfseries contravariant semi-dependent pairing} of $\varphi$ and $g$, that satisfies $\langle \varphi, g \rangle \circ \pi_1 = \varphi$ and $\pi_2 \{ \langle \varphi, g \rangle \} = g$.
\end{definition}

\begin{definition}[Semi-cartesian contravariant CwDs]
\label{DefSCDCs}
A contravariant CwD $\mathcal{C}$ is \emph{\bfseries semi-cartesian} iff it has a terminal object $T \in \mathcal{C}$ and a contravariant semi-$\Sigma$-space $\Gamma \stackrel{\pi_1}{\rightarrow} \Gamma . A \stackrel{\pi_2}{\rightarrowtriangle} A\{\pi_1\}$ of any pair of an object $\Gamma \in \mathcal{C}$ and a D-object $A \in \mathscr{D}_{\mathcal{C}}(\Gamma)$.
\end{definition}

Clearly, the opposite of an SCCwD is a semi-cartesian contravariant CwD; the semi-cartesian contravariant CwDs which we shall consider occur mostly in this manner.

\subsubsection{Semi-Cartesian Functors with Dependence}
Next, following the general principle of category theory that morphisms are structure-preserving, let us define:
\begin{definition}[SCFwDs]
A \emph{\bfseries semi-cartesian FwD (SCFwD)} is an FwD $F : \mathcal{C} \to \mathcal{C'}$ between SCCwDs $\mathcal{C}$ and $\mathcal{C'}$ such that:
\begin{itemize}

\item The object $F(T) \in \mathcal{C'}$ is terminal in $\mathcal{C'}$ for any terminal object $T \in \mathcal{C}$; and 

\item The diagram $F(\Gamma) \stackrel{F(\pi_1^{\Gamma . A})}{\leftarrow} F(\Gamma . A) \stackrel{F(\pi_2^{\Gamma . A})}{\rightarrowtriangle} F(A)\{F(\pi_1^{\Gamma . A})\}$ in $\mathcal{C'}$ is a semi-$\Sigma$-space in $\mathcal{C'}$ for any semi-$\Sigma$-space $\Gamma \stackrel{\pi_1^{\Gamma . A}}{\leftarrow} \Gamma . A \stackrel{\pi_2^{\Gamma . A}}{\rightarrowtriangle} A\{\pi_1^{\Gamma . A}\}$ in $\mathcal{C}$.
\end{itemize}
If $\mathcal{C}$ and $\mathcal{C'}$ are both strict, written $\mathcal{C} = (\mathcal{C}, T, \_ . \_, \pi)$ and $\mathcal{C'} = (\mathcal{C'}, T', \_ .' \_, \pi')$, then $F$ is \emph{\bfseries strict} iff 
\begin{align*}
F(T) &= T' \\
F(\Gamma) \stackrel{F(\pi_1^{\Gamma . A})}{\leftarrow} F(\Gamma . A) \stackrel{F(\pi_2^{\Gamma . A})}{\rightarrowtriangle} F(A)\{F(\pi_1^{\Gamma . A})\} &= F(\Gamma) \stackrel{\pi_1'^{F(\Gamma) .' F(A)}}{\leftarrow} F(\Gamma) .' F(A) \stackrel{\pi_2'^{F(\Gamma) .' F(A)}}{\rightarrowtriangle} F(A)\{\pi_1'^{F(\Gamma) .' F(A)}\}
\end{align*}
for any semi-$\Sigma$-space $\Gamma \stackrel{\pi_1^{\Gamma . A}}{\leftarrow} \Gamma . A \stackrel{\pi_2^{\Gamma . A}}{\rightarrowtriangle} A\{\pi_1^{\Gamma . A}\}$ in $\mathcal{C}$.
\end{definition}

\begin{remark}
Note that the first projection $\pi_1$ is `fliped' under semi-cartesian \emph{contravariant} FwDs.
\end{remark}

Let $F : \mathcal{C} \to \mathcal{C'}$ be an SCFwD.
It is easy to see that there is a unique isomorphism $\iota_{T'}^{F(T)} : T' \stackrel{\sim}{\to} F(T)$ in $\mathcal{C'}$ for any terminal objects $T \in \mathcal{C}$ and $T' \in \mathcal{C'}$, and similarly given semi-$\Sigma$-spaces $\Gamma \stackrel{\pi_1}{\leftarrow} \Gamma . A \stackrel{\pi_2}{\rightarrowtriangle} A\{\pi_1\}$ in $\mathcal{C}$ and $F(\Gamma) \stackrel{\pi_1'}{\leftarrow} F(\Gamma) .' F(A) \stackrel{\pi_2'}{\rightarrowtriangle} F(A)\{\pi_1'\}$ in $\mathcal{C'}$ there is a unique isomorphism $\iota^{F(\Gamma . A)}_{F(\Gamma) .' F(A)} : F(\Gamma) .' F(A) \stackrel{\sim}{\to} F(\Gamma . A)$ in $\mathcal{C'}$.
Moreover, given that $\mathcal{C}$ and $\mathcal{C'}$ are both strict, $F$ is strict iff these isomorphisms are all identities.

\begin{notation}
We often abbreviate the isomorphism $\iota_{T'}^{F(T)}$ (resp. $\iota^{F(\Gamma . A)}_{F(\Gamma) .' F(A)}$) as $\iota$. 
\end{notation}

\begin{example}
Clearly, (resp. strict) \emph{finite-product-preserving} (or \emph{cartesian}) functors $\mathcal{C} \to \mathcal{D}$ coincide with (resp. strict) SCFwDs $\mathsf{D}(\mathcal{C}) \to \mathsf{D}(\mathcal{D})$ for any (resp. strict) cartesian categories $\mathcal{C}$ and $\mathcal{D}$. 
We call this kind of (resp. strict) SCFwDs (resp. \emph{\bfseries strict}) \emph{\bfseries cartesian functors seen as} (resp. \emph{\bfseries strict}) \emph{\bfseries SCFwDs}.
Thus, (resp. strict) SCFwDs are generalized (resp. strict) cartesian functors. 
\end{example}

\begin{example}
Given a locally small SCCwD $\mathcal{C}$ and an object $\Delta \in \mathcal{C}$, the FwD $\mathcal{C}(\Delta, \_) : \mathcal{C} \to \mathit{Sets}$ is semi-cartesian: For any semi-$\Sigma$-space $\Gamma \stackrel{\pi_1}{\leftarrow} \Gamma . A \stackrel{\pi_2}{\rightarrowtriangle} A\{\pi_1\}$ in $\mathcal{C}$, we have the semi-$\Sigma$-space 
\begin{equation*}
\mathcal{C}(\Delta, \Gamma) \stackrel{\mathcal{C}(\Delta, \pi_1)}{\leftarrow} \mathcal{C}(\Delta, \Gamma . A) \stackrel{\mathcal{C}(\Delta, \pi_2)}{\rightarrowtriangle} \mathcal{C}(\Delta, A)\{\mathcal{C}(\Delta, \pi_1)\}
\end{equation*} 
in $\mathit{Sets}$, i.e., $\mathcal{C}(\Delta, \Gamma) . \mathcal{C}(\Delta, A) \cong \mathcal{C}(\Delta, \Gamma . A)$, where semi-dependent pairings of $\mathcal{C}(\Delta, \Gamma . A)$ are the obvious `point-wise' semi-dependent pairings in $\mathcal{C}$. 

Dually, given a locally small semi-cartesian contravariant CwD $\mathcal{D}$ and $\Delta \in \mathcal{D}$, the contravariant FwD $\mathcal{D}(\_, \Delta) : \mathcal{D} \to \mathit{Sets}$ is semi-cartesian: Given a contravariant semi-$\Sigma$-space $\Gamma \stackrel{\pi_1}{\rightarrow} \Gamma . A \stackrel{\pi_2}{\rightarrowtriangle} A\{\pi_1\}$ in $\mathcal{D}$, we have the semi-$\Sigma$-space 
\begin{equation*}
\mathcal{D}(\Gamma, \Delta) \stackrel{\mathcal{D}(\pi_1, \Delta)}{\leftarrow} \mathcal{D}(\Gamma . A, \Delta) \stackrel{\mathcal{D}(\pi_2, \Delta)}{\rightarrowtriangle} \mathcal{D}(A, \Delta)\{\mathcal{D}(\pi_1, \Delta)\}
\end{equation*} 
in $\mathit{Sets}$, i.e., $\mathcal{D}(\Gamma, \Delta) . \mathcal{C}(A, \Delta) \cong \mathcal{C}(\Gamma . A, \Delta)$, where semi-dependent pairings of $\mathcal{C}(\Gamma . A, \Delta)$ are again the `point-wise' contravariant semi-dependent pairings in $\mathcal{D}$.
\end{example}

The following lemma generalizes well-known properties of cartesian functors: 
\begin{lemma}[SCFwD-lemma]
\label{LemSCFwDLemma}
Let $F : \mathcal{C} \to \mathcal{D}$ be an SCFwD, and $\Gamma, \Delta \in \mathcal{C}$, $A \in \mathscr{D}_{\mathcal{C}}(\Gamma)$, $B \in \mathscr{D}_{\mathcal{C}}(\Delta)$, $\phi : \Delta \to \Gamma$, $g : \Delta \rightarrowtriangle A\{\phi\}$ and $h : \Delta . B \rightarrowtriangle A\{\phi \circ \pi_1^{\Delta . B}\}$ in $\mathcal{C}$.
Then, we have:
\begin{enumerate}

\item $F(\pi_1^{\Gamma . A}) \circ \iota^{F(\Gamma . A)}_{F(\Gamma) . F(A)} = \pi_1^{F(\Gamma) . F(A)} : F(\Gamma) . F(A) \to F(\Gamma)$;

\item $F(\pi_2^{\Gamma . A}) \{ \iota^{F(\Gamma . A)}_{F(\Gamma) . F(A)} \} = \pi_2^{F(\Gamma) . F(A)} : F(\Gamma) . F(A) \rightarrowtriangle F(A) \{ \pi_1^{F(\Gamma) . F(A)} \}$; 

\item $F(\langle \phi, g \rangle) = \iota^{F(\Gamma . A)}_{F(\Gamma) . F(A)} \circ \langle F(\phi), F(g) \rangle : F(\Delta) \to F(\Gamma . A)$; and

\item $(\iota^{F(\Gamma . A)}_{F(\Gamma) . F(A)})^{-1} \circ F(\langle \phi \circ \pi_1^{\Delta . B}, h \rangle) \circ \iota^{F(\Delta . B)}_{F(\Delta) . F(B)} = \langle F(\phi) \circ \pi_1^{F(\Delta) . F(B)}, F(h) \{ \iota^{F(\Delta . B)}_{F(\Delta) . F(B)} \} \rangle : F(\Delta) . F(B) \to F(\Gamma) . F(A)$.

\end{enumerate}
\end{lemma}
\begin{proof}
The equations 1 and 2 follow from the definition of the canonical isomorphism $\iota^{F(\Gamma . A)}_{F(\Gamma) . F(A)} : F(\Gamma) . F(A) \stackrel{\sim}{\to} F(\Gamma . A)$, i.e., $\iota^{F(\Gamma . A)}_{F(\Gamma) . F(A)} \stackrel{\mathrm{df. }}{=} \langle F(\pi_1^{\Gamma . A}), F(\pi_2^{\Gamma . A}) \rangle^{-1}$, and for the equation 3 we have:
\begin{align*}
\pi_1^{F(\Gamma) . F(A)} \circ ((\iota^{F(\Gamma . A)}_{F(\Gamma) . F(A)} )^{-1} \circ F(\langle \phi, g \rangle)) &= (\pi_1^{F(\Gamma) . F(A)} \circ (\iota^{F(\Gamma . A)}_{F(\Gamma) . F(A)} )^{-1}) \circ F(\langle \phi, g \rangle) \\
&= F(\pi_1^{\Gamma . A}) \circ F(\langle \phi, g \rangle) \ \text{(by the equation 1)} \\
&= F(\pi_1^{\Gamma . A} \circ \langle \phi, g \rangle) \\
&= F(\phi)
\end{align*}
as well as:
\begin{align*}
\pi_2^{F(\Gamma) . F(A)} \{ (\iota^{F(\Gamma . A)}_{F(\Gamma) . F(A)} )^{-1} \circ F(\langle \phi, g \rangle) \} &= \pi_2^{F(\Gamma) . F(A)} \{ (\iota^{F(\Gamma . A)}_{F(\Gamma) . F(A)} )^{-1} \} \{ F(\langle \phi, g \rangle) \} \\
&= F(\pi_2^{\Gamma . A}) \{ F(\langle \phi, g \rangle) \} \ \text{(by the equation 2)} \\
&= F(\pi_2^{\Gamma . A} \{ \langle \phi, g \rangle \}) \\
&= F(g).
\end{align*}
Hence, by the universal property of semi-dependent pairings, we may conclude that
\begin{equation*}
(\iota^{F(\Gamma . A)}_{F(\Gamma) . F(A)} )^{-1} \circ F(\langle \phi, g \rangle) = \langle F(\phi), F(g) \rangle
\end{equation*}
whence $F(\langle \phi, g \rangle) = \iota^{F(\Gamma . A)}_{F(\Gamma) . F(A)} \circ \langle F(\phi), F(g) \rangle : F(\Delta) \to F(\Gamma . A)$, establishing the equation 3.

Finally, the equation 4 is an immediate corollary of the equations 1 and 3.
\end{proof}


\subsubsection{Natural Transformations with Dependence}
Note that the additional semi-cartesian structure on SCCwDs $\mathcal{C}$ enables us to define a \emph{morphism between D-objects} $D \in \mathscr{D}_{\mathcal{C}}(\Delta)$ and $C \in \mathscr{D}_{\mathcal{C}}(\Gamma)$ to be a pair $(\phi, f)$ of a morphism $\phi : \Delta \to \Gamma$ and a D-morphism $f : \Delta . D \rightarrowtriangle C \{ \phi \circ \pi_1 \}$ in $\mathcal{C}$.
Based on this idea, it seems a reasonable idea to define a \emph{morphism between SCFwDs} $F, G : \mathcal{C} \to \mathcal{D}$ to be a pair $(\eta, t)$ of a family $\eta = (\eta_\Gamma)_{\Gamma \in \mathcal{C}}$ of morphisms $\eta_\Gamma : F(\Gamma) \to G(\Gamma)$ in $\mathcal{D}$ and a family $t = (t_{\Gamma, A})_{\Gamma \in \mathcal{C}, A \in \mathscr{D}_{\mathcal{C}}(\Gamma)}$ of D-morphisms $t_{\Gamma, A} : F(\Gamma) . F(A) \rightarrowtriangle G(A)\{\eta_\Gamma \circ \pi_1\}$ in $\mathcal{D}$ such that the diagrams
\begin{diagram}
F(\Gamma) &&\rTo^{F(\varphi)} && F(\Delta) && F(\Gamma) . F(A) \cong & F(\Gamma . A) &&\rTo^{\langle F(\varphi \circ \pi_1), F(g) \rangle} && F(\Delta) . F(B) \\
&&&&&&&&&&& \dDepTo_{t_{\Delta, B}} \\
\dTo_{\eta_\Gamma} &&&& \dTo_{\eta_{\Delta}} && \dTo_{\langle \eta_\Gamma \circ \pi_1, t_{\Gamma, A} \rangle} &&&&& G(B)\{\eta_\Delta \circ \pi_1\} \\
&&&&&&&&&&& \dImplies \\
G(\Gamma) && \rTo^{G(\varphi)} && G(\Delta) && G(\Gamma) . G(A) \cong & G(\Gamma . A) & \rDepTo^{G(g)} &G(B)\{G(\varphi \circ \pi_1)\}& \rImplies &G(B) \{G(\varphi) \circ \eta_\Gamma \circ \pi_1\}
\end{diagram}
commute for any $\Gamma, \Delta \in \mathcal{C}$, $A \in \mathscr{D}_{\mathcal{C}}(\Gamma)$, $B \in \mathscr{D}_{\mathcal{C}}(\Delta)$, $\varphi : \Gamma \to \Delta$ and $g : \Gamma . A \rightarrowtriangle B\{\varphi \circ \pi_1\}$ in $\mathcal{C}$.

Nevertheless, this concept turns out to be too general as morphisms between SCFwDs. 
For instance, such morphisms between cartesian functors seen as SCFwDs do not coincide with natural transformations, and Theorem~\ref{ThmDSTheorem} would not hold if we adopted them as morphisms between SCFwDs. 

Fortunately, however, the right one is simpler:
\begin{definition}[NTwDs]
A \emph{\bfseries natural transformation with dependence (NTwD)} between SCFwDs $F, G : \mathcal{C} \to \mathcal{D}$ is a pair $(\eta, t)$ of a family $\eta = (\eta_\Gamma)_{\Gamma \in \mathcal{C}}$ of morphisms $\eta_\Gamma : F(\Gamma) \to G(\Gamma)$ in $\mathcal{D}$ and a family $t = (t_{\Gamma, A})_{\Gamma \in \mathcal{C}, A \in \mathscr{D}_{\mathcal{C}}(\Gamma)}$ of D-morphisms $t_{\Gamma, A} : F(\Gamma) . F(A) \rightarrowtriangle G(A)\{\eta_\Gamma \circ \pi_1\}$ in $\mathcal{D}$ such that $\eta$ forms a natural transformation (NT) $\mathscr{U}(F) \Rightarrow \mathscr{U}(G) : \mathscr{U}(\mathcal{C}) \to \mathscr{U}(\mathcal{D})$, and the pair is \emph{\bfseries coherent} in the sense that for any $\Gamma, \Delta \in \mathcal{C}$, $A \in \mathscr{D}_{\mathcal{C}}(\Gamma)$, $B \in \mathscr{D}_{\mathcal{C}}(\Delta)$, $\varphi : \Gamma \to \Delta$ and $g : \Gamma . A \rightarrowtriangle B\{\varphi \circ \pi_1\}$ in $\mathcal{C}$ it satisfies the equation
\begin{equation}
\label{CohNTwDs}
F(\Gamma) . F(A) \stackrel{\iota}{\rightarrow} F(\Gamma . A) \stackrel{\eta_{\Gamma . A}}{\rightarrow} G(\Gamma . A) \stackrel{\iota^{-1}}{\rightarrow} G(\Gamma) . G(A) \stackrel{\pi_2}{\rightarrowtriangle} G(A)\{\pi_1\} \Rightarrow G(A)\{ \eta_\Gamma \circ \pi_1 \} = t_{\Gamma, A}
\end{equation}
where the semi-$\Sigma$-spaces $\Gamma . A \in \mathcal{C}$ and $G(\Gamma) . G(A) \in \mathcal{D}$ are arbitrary.
\end{definition}

Let us see that the coherence axiom (\ref{CohNTwDs}) makes the diagram 
\begin{diagram}
 F(\Gamma) . F(A) \cong & F(\Gamma . A) &&\rTo^{\langle F(\varphi \circ \pi_1), F(g) \rangle} && F(\Delta) . F(B) \\
&&&&& \dDepTo_{t_{\Delta, B}} \\
\dTo_{\langle \eta_\Gamma \circ \pi_1, t_{\Gamma, A} \rangle} &&&&& G(B)\{\eta_\Delta \circ \pi_1\} \\
&&&&& \dImplies \\
G(\Gamma) . G(A) \cong & G(\Gamma . A) & \rDepTo^{G(g)} &G(B)\{G(\varphi \circ \pi_1)\}& \rImplies &G(B) \{G(\varphi) \circ \eta_\Gamma \circ \pi_1\}
\end{diagram}
commutes for any $\Gamma, \Delta \in \mathcal{C}$, $A \in \mathscr{D}_{\mathcal{C}}(\Gamma)$, $B \in \mathscr{D}_{\mathcal{C}}(\Delta)$, $\varphi : \Gamma \to \Delta$ and $g : \Gamma . A \rightarrowtriangle B\{\varphi \circ \pi_1\}$ in $\mathcal{C}$.
First, by the coherence of $(\eta, t)$ and the naturality of $\eta$, we have:
\begin{align*}
G(g) \{ \iota \} \{ \langle \eta_\Gamma \circ \pi_1, t_{\Gamma, A} \rangle \} &= G(g)\{\iota\} \{ \langle \pi_1 \circ (\iota^{-1} \circ \eta_{\Gamma . A} \circ \iota), \pi_2 \{ \iota^{-1} \circ \eta_{\Gamma . A} \circ \iota \} \rangle \} \\
&= G(g)\{\iota \circ \langle \pi_1, \pi_2 \rangle \circ \iota^{-1} \circ \eta_{\Gamma . A} \circ \iota \} \\
&= G(g) \{ \eta_{\Gamma . A} \circ \iota \}
\end{align*}
as well as:
\begin{align*}
t_{\Delta, B} \{ \langle F(\varphi \circ \pi_1), F(g) \rangle \} \{ \iota \} &= \pi_2 \{ \iota^{-1} \circ \eta_{\Delta . B} \circ \iota \} \{ \iota^{-1} \circ F(\langle \varphi \circ \pi_1, g\rangle) \} \{ \iota \} \\
&= \pi_2 \{ \iota^{-1} \circ \eta_{\Delta . B} \circ \iota \circ \iota^{-1} \circ F(\langle \varphi \circ \pi_1, g\rangle) \circ \iota \} \\
&= \pi_2 \{ \iota^{-1} \circ \eta_{\Delta . B} \circ F(\langle \varphi \circ \pi_1, g\rangle) \circ \iota \}
\end{align*}
whence it suffices to show:
\begin{equation*}
G(g) \{ \eta_{\Gamma . A} \circ \iota \} =  \pi_2 \{ \iota^{-1} \circ \eta_{\Delta . B} \circ F(\langle \varphi \circ \pi_1, g\rangle) \circ \iota \}.
\end{equation*}

Then, by the naturality of $\eta$, we have:
\begin{equation*}
G(\langle \varphi \circ \pi_1, g \rangle) \circ \eta_{\Gamma . A} = \eta_{\Delta . B} \circ F(\langle \phi \circ \pi_1, g \rangle)
\end{equation*}
and thus, it is immediate that:
\begin{align*}
G(g) \{ \eta_{\Gamma . A} \circ \iota \} &= \pi_2 \{ \langle G(\varphi \circ \pi_1), G(g) \rangle \circ \eta_{\Gamma . A} \circ \iota \} \\
&= \pi_2 \{ \iota^{-1} \circ G(\langle \varphi \circ \pi_1, g \rangle) \circ \eta_{\Gamma . A} \circ \iota \} \\
&= \pi_2 \{ \iota^{-1} \circ \eta_{\Delta . B} \circ F(\langle \phi \circ \pi_1, g \rangle) \circ \iota \}.
\end{align*}

Next, note that NTwDs $(\eta, t)$ between \emph{strict} SCFwDs $F, G : \mathcal{C} \to \mathcal{D}$ are essentially NTs $\eta : \mathscr{U}(F) \Rightarrow \mathscr{U}(G) : \mathscr{U}(\mathcal{C}) \to \mathscr{U}(\mathcal{D})$ because $t$ may be completely recovered from $\eta$ by the equation $t_{\Gamma, A} = \pi_2 \{\eta_{\Gamma . A} \}$ for any $\Gamma \in \mathcal{C}$ and $A \in \mathscr{D}_{\mathcal{C}}(\Gamma)$.
This point plays an essential role for the commutative diagram in the proof of Theorem~\ref{ThmDSTheorem}.

\begin{notation}
We write $(\eta, t) : F \Rightarrow G$ to mean that the pair $(\eta, t)$ is an NTwD from an SCCwD $F$ to another $G$, or even $(\eta, t) : F \Rightarrow G : \mathcal{C} \to \mathcal{D}$ if we would like to indicate $F, G : \mathcal{C} \to \mathcal{D}$.
\end{notation}

We skip describing \emph{contravariant NTwDs} explicitly because now it should be clear for the reader what they are.
More generally, we shall focus on \emph{covariant} cases in the rest of the paper.

\begin{example}
An NT $\eta : F \Rightarrow G : \mathcal{C} \to \mathcal{D}$ gives rise to an NTwD $\mathsf{D}(\eta) \stackrel{\mathrm{df. }}{=} (\eta, t(\eta)) : \mathsf{D}(F) \Rightarrow \mathsf{D}(G) : \mathsf{D}(\mathcal{C}) \to \mathsf{D}(\mathcal{D})$ by $t(\eta)_{\Gamma, A} \stackrel{\mathrm{df. }}{=} \eta_A \circ \pi_2^{F(\Gamma) . F(A)}$ for all $\Gamma \in \mathcal{C}$ and $A \in \mathscr{D}_{\mathcal{C}}(\Gamma)$, as well as another $\mathsf{d}(\eta) \stackrel{\mathrm{df. }}{=} (\eta, \emptyset) : \mathsf{d}(F) \Rightarrow \mathsf{d}(G) : \mathsf{d}(\mathcal{C}) \to \mathsf{d}(\mathcal{D})$.
We call the first kind \emph{\bfseries NTs seen as NTwDs}.
Conversely, an NTwD $(\epsilon, s) : J \Rightarrow K : \mathcal{L} \to \mathcal{R}$ such that $J$ and $K$ are cartesian functors seen as SCFwDs and $s_{\Gamma, A} = \epsilon_A \circ \pi_2^{F(\Gamma) . F(A)}$ for all $\Gamma \in \mathcal{L}$ and $A \in \mathscr{D}_{\mathcal{L}}(\Gamma)$ coincides with an NT.
Hence, NTwDs are a generalization of NTs.
\end{example}

\begin{example}
Given a locally small SCCwD $\mathcal{C}$ and a morphism $\phi : \Delta \to \Gamma$ in $\mathcal{C}$, we may give an NTwD $(\eta(\phi), t(\phi)) : \mathcal{C}(\Gamma, \_) \Rightarrow \mathcal{C}(\Delta, \_) : \mathcal{C} \to \mathit{Sets}$ by defining for any $\Theta \in \mathcal{C}$ and $C \in \mathscr{D}_{\mathcal{C}}(\Theta)$ 
\begin{align*}
\eta(\phi)_\Theta &\stackrel{\mathrm{df. }}{=} \mathcal{C}(\phi, \Theta) : \mathcal{C}(\Gamma, \Theta) \to \mathcal{C}(\Delta, \Theta) \\ t(\phi)_{\Theta, C} &\stackrel{\mathrm{df. }}{=} \mathcal{C}(\phi, C) \circ \pi_2 : \mathcal{C}(\Gamma, \Theta) . \mathcal{C}(\Gamma, C) \rightarrowtriangle \mathcal{C}(\Delta, C) \{ \mathcal{C}(\phi, \Theta) \circ \pi_1 \}.
\end{align*}
Note that $\eta(\phi)_\Theta$ maps each morphism $\varphi : \Gamma \to \Theta$ in $\mathcal{C}$ to the morphism $\varphi \circ \phi : \Delta \to \Theta$ in $\mathcal{C}$, and $t(\phi)_{\Theta, C}$ maps each pair $(\varphi, g) \in \mathcal{C}(\Gamma, \Theta) . \mathcal{C}(\Gamma, C)$ to the D-morphism $g\{\phi\} \in \mathscr{D}_{\mathcal{C}}(\Delta, C\{\varphi\}\{\phi\})$ in $\mathcal{C}$.
It is easy to see that the pair $(\eta(\phi), t(\phi))$ satisfies the required naturality condition. 
\end{example}

Similarly to the case of NTs, it is straightforward to define \emph{vertical composition} of NTwDs:
\begin{definition}[Vertical composition of NTwDs]
The \emph{\bfseries vertical composition} $(\epsilon, s) \circ (\eta, t) : F \Rightarrow H : \mathcal{C} \to \mathcal{D}$ of NTwDs $(\eta, t) : F \Rightarrow G : \mathcal{C} \to \mathcal{D}$ and $(\epsilon, s) : G \Rightarrow H : \mathcal{C} \to \mathcal{D}$ is given by:
\begin{align*}
(\epsilon, s) \circ (\eta, t) &\stackrel{\mathrm{df. }}{=} (\epsilon \circ \eta, s \circ t) \\
\epsilon \circ \eta &\stackrel{\mathrm{df. }}{=} (\epsilon_\Gamma \circ \eta_\Gamma)_{\Gamma \in \mathcal{C}} \\
s \circ t &\stackrel{\mathrm{df. }}{=} (s_{\Gamma, A} \{ \langle \eta_{\Gamma} \circ \pi_1, t_{\Gamma, A} \rangle \})_{\Gamma \in \mathcal{C}, A \in \mathscr{D}_{\mathcal{C}}(\Gamma)}.
\end{align*}
\end{definition}

Note that the vertical composition on the first components is just the conventional vertical composition on NTs.
It is easy to see that vertical composition of NTwDs is well-defined:
\begin{proposition}[Well-defined VC of NTwDs]
NTwDs are closed under vertical composition.
\end{proposition}
\begin{proof}
It suffices to show that the coherence \ref{CohNTwDs} is preserved under vertical composition. 
Let $(\eta, t) : F \Rightarrow G : \mathcal{C} \to \mathcal{D}$ and $(\epsilon, s) : G \Rightarrow H : \mathcal{C} \to \mathcal{D}$ be NTwDs.
Then, we have:
\begin{align*}
(s \circ t)_{\Gamma, A} &= s_{\Gamma, A} \{ \langle \eta_\Gamma \circ \pi_1, t_{\Gamma, A} \rangle \} \\
&= \pi_2 \{ \iota^{-1} \circ \epsilon_{\Gamma, A} \circ \iota \} \{ \langle \eta_\Gamma \circ \pi_1, \pi_2 \{ \iota^{-1} \circ \eta_{\Gamma, A} \circ \iota \} \rangle \} \\
&= \pi_2 \{ \iota^{-1} \circ \epsilon_{\Gamma, A} \circ \iota \} \{ \langle \pi_1 \circ \iota^{-1} \circ \eta_{\Gamma . A} \circ \iota, \pi_2 \{ \iota^{-1} \circ \eta_{\Gamma, A} \circ \iota \} \rangle \} \\
&= \pi_2 \{ \iota^{-1} \circ \epsilon_{\Gamma, A} \circ \iota \} \{ \langle \pi_1, \pi_2 \rangle \circ \iota^{-1} \circ \eta_{\Gamma . A} \circ \iota \} \\
&= \pi_2 \{ \iota^{-1} \circ \epsilon_{\Gamma, A} \circ \iota \circ \mathit{id}_{F(\Gamma) . F(A)} \circ \iota^{-1} \circ \eta_{\Gamma . A} \circ \iota \} \\
&= \pi_2 \{ \iota^{-1} \circ \epsilon_{\Gamma, A} \circ \eta_{\Gamma . A} \circ \iota \} \\
&= \pi_2 \{ \iota^{-1} \circ (\epsilon \circ \eta)_{\Gamma, A} \circ \iota \} 
\end{align*}
for any $\Gamma \in \mathcal{C}$ and $A \in \mathscr{D}_{\mathcal{C}}(\gamma)$, which establishes the coherence \ref{CohNTwDs} on the vertical composition $(\epsilon, s) \circ (\eta, t) = (\epsilon \circ \eta, s \circ t)$.
\end{proof}

As one may have already expected, we also have \emph{horizontal composition} of NTwDs:
\begin{definition}[Horizontal composition of NTwDs]
The \emph{\bfseries horizontal composition} $(\eta', t') \odot (\eta, t) : F' \circ F \Rightarrow G' \circ G : \mathcal{C} \to \mathcal{E}$ of NTwDs $(\eta, t) : F \Rightarrow G : \mathcal{C} \to \mathcal{D}$ and $(\eta', t') : F' \Rightarrow G' : \mathcal{D} \to \mathcal{E}$ is given by:
\begin{align*}
(\eta', t') \odot (\eta, t) &\stackrel{\mathrm{df. }}{=} (\eta' \odot \eta, t' \odot t) \\
\eta' \odot \eta &\stackrel{\mathrm{df. }}{=} (\eta'_{G(\Gamma)} \circ F'(\eta_\Gamma))_{\Gamma \in \mathcal{C}} = (G'(\eta_\Gamma) \circ \eta'_{F(\Gamma)})_{\Gamma \in \mathcal{C}} \\
t' \odot t &\stackrel{\mathrm{df. }}{=}  (t'_{G(\Gamma), G(A)} \{ \langle F'(\eta_\Gamma) \circ \pi_1, F'(t_{\Gamma, A})\{ \iota^{F' \circ F(\Gamma . A)}_{F' \circ F(\Gamma) . F' \circ F(A)} \} \rangle \})_{\Gamma \in \mathcal{C}, A \in \mathscr{D}_{\mathcal{C}}(\Gamma)}.
\end{align*}
\end{definition}

Again, note that the horizontal composition on the first components is just the horizontal composition on NTs.
It is not hard to see that horizontal composition on NTwDs is well-defined:
\begin{proposition}[Well-defined HC of NTwDs]
NTwDs are closed under horizontal composition.
\end{proposition}
\begin{proof}
Let $(\eta, t) : F \Rightarrow G : \mathcal{C} \to \mathcal{D}$ and $(\eta', t') : F' \Rightarrow G' : \mathcal{D} \to \mathcal{E}$ be NTwDs.
We show that horizontal composition $(\eta', t') \odot (\eta, t) : F' \circ F \Rightarrow G' \circ G : \mathcal{C} \to \mathcal{E}$ satisfies the coherence (\ref{CohNTwDs}):
\begin{align*}
(t' \odot t)_{\Gamma, A} &= t'_{G(\Gamma), G(A)} \{ \langle F'(\eta_\Gamma) \circ \pi_1, F'(t_{\Gamma, A}) \{ \iota \} \rangle \} \\
&= t'_{G(\Gamma), G(A)} \{ \langle F'(\eta_\Gamma \circ \pi_1) \circ \iota, F'(t_{\Gamma, A}) \{ \iota \} \rangle \} \\
&= \pi_2 \{ \iota^{-1} \circ \eta'_{G(\Gamma) . G(A)} \circ \iota \} \{ \langle F' (G(\pi_1) \circ \eta_{\Gamma . A} \circ \iota), F'(\pi_2\{ \iota^{-1} \circ \eta_{\Gamma . A} \circ \iota \}) \rangle \circ \iota \} \\
&= \pi_2 \{ \iota^{-1} \circ \eta'_{G(\Gamma) . G(A)} \circ \iota \} \{ \iota^{-1} \circ F' (\langle G(\pi_1) \circ \eta_{\Gamma . A} \circ \iota, \pi_2\{ \iota^{-1} \circ \eta_{\Gamma . A} \circ \iota \} \rangle) \circ \iota \} \\
&= \pi_2 \{ \iota^{-1} \circ \eta'_{G(\Gamma) . G(A)} \circ \iota \} \{ \iota^{-1} \circ F' (\langle \pi_1 \circ (\iota^{-1} \circ \eta_{\Gamma . A} \circ \iota), \pi_2\{ \iota^{-1} \circ \eta_{\Gamma . A} \circ \iota \} \rangle) \circ \iota \} \\
&= \pi_2 \{ \iota^{-1} \circ \eta'_{G(\Gamma) . G(A)} \circ \iota \} \{ \iota^{-1} \circ F' (\langle \pi_1, \pi_2 \rangle \circ (\iota^{-1} \circ \eta_{\Gamma . A} \circ \iota)) \circ \iota \} \\
&= \pi_2 \{ \iota^{-1} \circ \eta'_{G(\Gamma) . G(A)} \circ \iota \} \{ \iota^{-1} \circ F' (\mathit{id}_{G(\Gamma) . G(A)} \circ \iota^{-1} \circ \eta_{\Gamma . A} \circ \iota) \circ \iota \} \\
&= \pi_2 \{ \iota^{-1} \circ \eta'_{G(\Gamma) . G(A)} \circ \iota \circ \iota^{-1} \circ F' (\iota^{-1}) \circ F'(\eta_{\Gamma . A}) \circ F'(\iota) \circ \iota \} \\
&= \pi_2 \{ \iota^{-1} \circ G'(\iota^{-1}) \circ \eta'_{G(\Gamma . A)} \circ F'(\iota) \circ \iota \circ \iota^{-1} \circ F' (\iota^{-1}) \circ F'(\eta_{\Gamma . A}) \circ F'(\iota) \circ \iota \} \\
&= \pi_2 \{ \iota^{-1} \circ \eta'_{G(\Gamma . A)} \circ F'(\eta_{\Gamma . A}) \circ \iota \} \\
&= \pi_2 \{ \iota^{-1} \circ (\eta' \odot \eta)_{\Gamma . A} \circ \iota \}
\end{align*}
for any $\Gamma \in \mathcal{C}$ and $A \in \mathscr{D}_{\mathcal{C}}(\gamma)$, completing the proof.
\end{proof}

\subsubsection{The 2-Category of Semi-Cartesian Categories with Dependence}
We are now ready to define the 2-category of small SCCwDs, SCFwDs and NTwDs.
Let us first define the following preliminary concept:
\begin{definition}[SCFwD-categories]
Given SCCwDs $\mathcal{C}$ and $\mathcal{D}$, the category $[\mathcal{C}, \mathcal{D}]$, called the \emph{\bfseries SCFwD-category} from $\mathcal{C}$ to $\mathcal{D}$ is given by:
\begin{itemize}

\item Objects are SCFwDs $\mathcal{C} \to \mathcal{D}$;

\item Morphisms $F \to G$ are NTwDs $F \Rightarrow G$;

\item The composition $F \stackrel{(\eta, t)}{\to} G \stackrel{(\epsilon, s)}{\to} H$ is the vertical composition $(\epsilon, s) \circ (\eta, t)$; and

\item The identity $\mathit{id}_F$ is the pair $((\mathit{id}_{F(\Gamma)})_{\Gamma \in \mathcal{C}}, (\pi_2^{F(\Gamma) . F(A)})_{\Gamma \in \mathcal{C}, A \in \mathscr{D}_{\mathcal{C}}(\Gamma)})$.

\end{itemize}
\end{definition}

It is easy to see that SCFwD-categories are well-defined.
We are now ready to define:
\begin{definition}[The 2-category $\mathbb{SCC_D}$]
The 2-category $\mathbb{SCC_D}$ is defined by:
\begin{itemize}

\item 0-cells (or objects) are small SCCwDs;

\item Given 0-cells $\mathcal{C}$ and $\mathcal{D}$, the hom-category $\mathbb{SCC_D}(\mathcal{C}, \mathcal{D})$ is the SCFwD-category $[\mathcal{C}, \mathcal{D}]$;

\item Given 0-cells $\mathcal{C}$, $\mathcal{D}$ and $\mathcal{E}$, the composition functor $[\mathcal{C}, \mathcal{D}] \times [\mathcal{D}, \mathcal{E}] \to [\mathcal{C}, \mathcal{E}]$ maps
\begin{align*}
(F, F') &\mapsto F' \circ F \\
((\eta, t), (\eta', t')) &\mapsto (\eta', t') \odot (\eta, t)
\end{align*}
for all $F, G \in [\mathcal{C}, \mathcal{D}]$, $F', G' \in [\mathcal{D}, \mathcal{E}]$, $(\eta, t) \in [\mathcal{C}, \mathcal{D}](F, G)$ and $(\eta', t') \in [\mathcal{D}, \mathcal{E}](F', G')$; and

\item Given a 0-cell $\mathcal{C}$, the functor $1 \to [\mathcal{C}, \mathcal{C}]$ maps
\begin{align*}
\star &\mapsto \mathit{id}_{\mathcal{C}} \\
\mathit{id}_\star &\mapsto \mathit{id}_{\mathit{id}_{\mathcal{C}}}
\end{align*}
where $\star$ is the unique object of the terminal category $1$.
\end{itemize}
\end{definition}

\begin{theorem}[Well-defined $\mathbb{SCC_D}$]
The structure $\mathbb{SCC_D}$ in fact forms a 2-category.
\end{theorem}
\begin{proof}
Let us just show that $\mathbb{SCC_D}$ satisfies the \emph{interchange law} since the other axioms are easier to verify. 
Let $\mathcal{C}$, $\mathcal{D}$ and $\mathcal{E}$ be SCCwDs, $F, G, H : \mathcal{C} \to \mathcal{D}$ and $F', G', H' : \mathcal{D} \to \mathcal{E}$ SCFwDs, and $(\eta, t) : F \Rightarrow G$, $(\epsilon, s) : G \Rightarrow H$, $(\eta', t') : F' \Rightarrow G'$ and $(\epsilon', s') : G' \Rightarrow H'$ NTwDs. We have to establish the equation:
\begin{equation*}
((\epsilon', s') \odot (\epsilon, s)) \circ ((\eta', t') \odot (\eta, t)) = ((\epsilon', s') \circ (\eta', t')) \odot ((\epsilon, s) \circ (\eta, t))
\end{equation*}
which is equivalent to:
\begin{equation*}
((\epsilon' \odot \epsilon) \circ (\eta' \odot \eta), (s' \odot s) \circ (t' \odot t)) = ((\epsilon' \circ \eta') \odot (\epsilon \circ \eta), (s' \circ t') \odot (s \circ t))
\end{equation*}
which is reduced to:
\begin{equation}
\label{InterchangeLaw}
(s' \odot s) \circ (t' \odot t) = (s' \circ t') \odot (s \circ t)
\end{equation}
because the equation of the first components is just the well-known interchange law for the 2-category $\mathit{CAT}$ of small categories, functors and NTs.

Then, observe that:
\begin{align*}
&((s' \circ t') \odot (s \circ t))_{\Gamma, A} \\
= \ & (s' \circ t')_{H(\Gamma), H(A)} \circ \langle F'(\epsilon_\Gamma \circ \eta_\Gamma) \circ \pi_1, F'((s \circ t)_{\Gamma, A})\{\iota\} \rangle \\
= \ & s'_{H(\Gamma), H(A)} \circ \langle \eta'_{H(\Gamma)} \circ \pi_1, t'_{H(\Gamma), H(A)} \rangle \circ \langle F'(\epsilon_\Gamma \circ \eta_\Gamma) \circ \pi_1, F'((s \circ t)_{\Gamma, A})\{\iota\} \rangle \\
= \ & s'_{H(\Gamma), H(A)} \circ \langle \eta'_{H(\Gamma)} \circ F'(\epsilon_\Gamma \circ \eta_\Gamma) \circ \pi_1, t'_{H(\Gamma), H(A)} \circ \langle F'(\epsilon_\Gamma \circ \eta_\Gamma) \circ \pi_1, F'((s \circ t)_{\Gamma, A})\{\iota\} \rangle \rangle \\
= \ & s'_{H(\Gamma), H(A)} \circ \langle G'(\epsilon_\Gamma \circ \eta_\Gamma) \circ \eta'_{F(\Gamma)} \circ \pi_1, t'_{H(\Gamma), H(A)} \circ \langle F'((\epsilon \circ \eta)_\Gamma \circ \pi_1), F'((s \circ t)_{\Gamma, A}) \rangle \circ \iota \rangle \rangle \\
= \ & s'_{H(\Gamma), H(A)} \circ \langle G'(\epsilon_\Gamma \circ \eta_\Gamma) \circ \eta'_{F(\Gamma)} \circ \pi_1, G'((s \circ t)_{\Gamma, A})\{\iota\} \circ \langle \eta'_{F(\Gamma)} \circ \pi_1, t'_{F(\Gamma), F(A)} \rangle \rangle \\
= \ & s'_{H(\Gamma), H(A)} \circ \langle G'(\epsilon_\Gamma \circ \eta_\Gamma) \circ \eta'_{F(\Gamma)} \circ \pi_1, G'(s_{\Gamma, A})\{\iota\} \circ \langle G'(\eta_\Gamma) \circ \pi_1, G'(t_{\Gamma, A})\{\iota\} \rangle \circ \langle \eta'_{F(\Gamma)} \circ \pi_1, t'_{F(\Gamma), F(A)} \rangle \rangle \\
= \ & s'_{H(\Gamma), H(A)} \circ \langle G'(\epsilon_\Gamma \circ \eta_\Gamma) \circ \eta'_{F(\Gamma)} \circ \pi_1, G'(s_{\Gamma, A})\{\iota\} \circ \langle G'(\eta_\Gamma) \circ \eta'_{F(\Gamma)} \circ \pi_1, G'(t_{\Gamma, A})\{\iota\} \circ \langle \eta'_{F(\Gamma)} \circ \pi_1, t'_{F(\Gamma), F(A)} \rangle \rangle \rangle \\
= \ & s'_{H(\Gamma), H(A)} \circ \langle G'(\epsilon_\Gamma \circ \eta_\Gamma) \circ \eta'_{F(\Gamma)} \circ \pi_1, G'(s_{\Gamma, A})\{\iota\} \circ \langle G'(\eta_\Gamma) \circ \eta'_{F(\Gamma)} \circ \pi_1, t'_{G(\Gamma), G(A)} \circ \langle F'(\eta_\Gamma \circ \pi_1), F'(t_{\Gamma, A}) \rangle \circ \iota \rangle \rangle \\
= \ & s'_{H(\Gamma), H(A)} \circ \langle G'(\epsilon_\Gamma \circ \eta_\Gamma) \circ \eta'_{F(\Gamma)} \circ \pi_1, G'(s_{\Gamma, A})\{\iota\} \circ \langle G'(\eta_\Gamma) \circ \eta'_{F(\Gamma)} \circ \pi_1, t'_{G(\Gamma), G(A)} \circ \langle F'(\eta_\Gamma) \circ \pi_1, F'(t_{\Gamma, A})\{\iota\} \rangle \rangle \rangle \\
= \ & s'_{H(\Gamma), H(A)} \circ \langle G'(\epsilon_\Gamma) \circ G'(\eta_\Gamma) \circ \eta'_{F(\Gamma)} \circ \pi_1, G'(s_{\Gamma, A})\{\iota\} \circ \langle G'(\eta_\Gamma) \circ \eta'_{F(\Gamma)} \circ \pi_1, (t' \odot t)_{\Gamma, A} \rangle \rangle \\
= \ & s'_{H(\Gamma), H(A)} \circ \langle G'(\epsilon_\Gamma) \circ \eta'_{G(\Gamma)} \circ F'(\eta_\Gamma) \circ \pi_1, G'(s_{\Gamma, A})\{\iota\} \circ \langle \eta'_{G(\Gamma)} \circ F'(\eta_\Gamma) \circ \pi_1, (t' \odot t)_{\Gamma, A} \rangle \rangle \\
= \ & s'_{H(\Gamma), H(A)} \circ \langle G'(\epsilon_\Gamma) \circ \pi_1, G'(s_{\Gamma, A})\{\iota\} \rangle \circ \langle \eta'_{G(\Gamma)} \circ F'(\eta_\Gamma) \circ \pi_1, (t' \odot t)_{\Gamma, A} \rangle \\
= \ & s'_{H(\Gamma), H(A)} \circ \langle G'(\epsilon_\Gamma) \circ \pi_1, G'(s_{\Gamma, A})\{\iota\} \rangle \circ \langle (\eta' \odot \eta)_\Gamma \circ \pi_1, (t' \odot t)_{\Gamma, A} \rangle \\
= \ & (s' \odot s)_{\Gamma, A} \circ \langle (\eta' \odot \eta)_\Gamma \circ \pi_1, (t' \odot t)_{\Gamma, A} \rangle \\
= \ & ((s' \odot s) \circ (t' \odot t))_{\Gamma, A}
\end{align*}
for all $\Gamma \in \mathcal{C}$ and $A \in \mathscr{D}_{\mathcal{C}}(\Gamma)$, which establishes the equation (\ref{InterchangeLaw}).
\end{proof}

By this 2-categorical structure of small SCCwDs, we may employ the notions of adjunctions and equivalences between 0-cells in an arbitrary 2-category to define \emph{adjunctions} and \emph{equivalences} between small SCCwDs.
Clearly, we may unpack these definitions and define adjunctions and equivalences between large SCCwDs as well. 

\subsubsection{Categories of Dependent Objects}
This section introduces a 2-functor $(\_)_{\mathscr{D}} : \mathbb{SCC_D} \to \mathit{CAT}$, which will turn out to be very useful for the rest of the paper.

Let us begin with the operation on 0-cells:
\begin{definition}[Categories of D-objects]
Given an SCCwD $\mathcal{C}$, the \emph{\bfseries category of D-objects} in $\mathcal{C}$ is the category $\mathcal{C}_{\mathscr{D}}$ given by: 
\begin{itemize}

\item $\mathsf{ob}(\mathcal{C}_{\mathscr{D}}) \stackrel{\mathrm{df. }}{=} \{ (\Gamma, C) \mid \Gamma \in \mathcal{C}, C \in \mathscr{D}_{\mathcal{C}}(\Gamma) \ \! \}$;

\item $\mathcal{C}_{\mathscr{D}}((\Delta, D), (\Gamma, C)) \stackrel{\mathrm{df. }}{=} \{ (\phi, f) \mid \phi \in \mathcal{C}(\Delta, \Gamma), f \in \mathscr{D}_{\mathcal{C}}(\Delta . D, C\{ \phi \circ \pi_1 \}) \ \! \}$;

\item $(\Delta, D) \stackrel{(\phi, f)}{\to} (\Gamma, C) \stackrel{(\psi, g)}{\to} (\Theta, E) \stackrel{\mathrm{df. }}{=} (\psi \circ \phi, g \{ \langle \phi \circ \pi_1, f \rangle \}_{\mathcal{C}})$;

\item $\mathit{id}_{(\Gamma, C)} \stackrel{\mathrm{df. }}{=} (\mathit{id}_\Gamma, \pi_2)$.

\end{itemize}
\end{definition}

It is straightforward to check that $\mathcal{C}_{\mathscr{D}}$ is a well-defined category for any SCCwD $\mathcal{C}$, and there is the obvious forgetful functor $p_{\mathcal{C}} : \mathcal{C}_{\mathscr{D}} \to \mathcal{C}$.
This construction may be seen as a categorical abstraction of the syntactic \emph{comprehension category} of a DTT \cite{jacobs1999categorical}.

Note that each morphism $(\phi, f) : (\Delta, D) \to (\Gamma, C)$ in $\mathcal{C}_{\mathscr{D}}$ induces a morphism $\langle \phi \circ \pi_1, f \rangle : \Delta . D \to \Gamma . C$ in $\mathcal{C}$; however, we have not taken $\langle \phi \circ \pi_1, f \rangle$ as a morphism $(\Delta, D) \to (\Gamma, C)$ in $\mathcal{C}_{\mathscr{D}}$ since otherwise we cannot define the composition of these morphisms as given above (as semi-$\Sigma$-spaces are unique only up to isomorphisms).

Next, let us lift this construction to SCFwDs as follows:
\begin{definition}[Functors between categories of D-objects]
Each SCFwD $F : \mathcal{C} \to \mathcal{D}$ induces the functor $F_{\mathscr{D}} : \mathcal{C}_{\mathscr{D}} \to  \mathcal{D}_{\mathscr{D}}$ that maps each object $(\Delta, D) \in \mathcal{C}_{\mathscr{D}}$ to the object $(F(\Delta), F(D)) \in \mathcal{D}_{\mathscr{D}}$ and each morphism $(\phi, f)$ in $\mathcal{C}_{\mathscr{D}}$ to the morphism $(F(\phi), F(f) \{ \iota^{F(\Delta . D)}_{F(\Delta) . F(D)} \})$ in $\mathcal{D}_{\mathscr{D}}$.
\end{definition}

By the equation 4 in Lemma~\ref{LemSCFwDLemma} (for the preservation of composition), it is easy to see that the functor $F_{\mathscr{D}} : \mathcal{C}_{\mathscr{D}} \to  \mathcal{D}_{\mathscr{D}}$ is well-defined for any SCFwD $F : \mathcal{C} \to \mathcal{D}$.

It is even simpler for the case of NTwDs:
\begin{definition}[NTwDs as NTs]
An NTwD $(\eta, t) : F \Rightarrow G : \mathcal{C} \to \mathcal{D}$ gives rise to an NT $(\eta, t)_{\mathscr{D}} \stackrel{\mathrm{df. }}{=}  (\eta_\Gamma, t_{\Gamma, A})_{(\Gamma, A) \in \mathcal{C}_{\mathscr{D}}} : F_{\mathscr{D}} \Rightarrow G_{\mathscr{D}} : \mathcal{C}_{\mathscr{D}} \to \mathcal{D}_{\mathscr{D}}$.
\end{definition}

It is easy to see that this construction $(\_)_{\mathscr{D}}$ on NTwDs is well-defined.
It in essence implies that an NTwD $(\eta, t) : F \Rightarrow G : \mathcal{C} \to \mathcal{D}$ can be seen as an NT $(\eta, t)_{\mathscr{D}} : F_{\mathscr{D}} \Rightarrow G_{\mathscr{D}} : \mathcal{C}_{\mathscr{D}} \to \mathcal{D}_{\mathscr{D}}$ such that components of $\eta$ do not depend on D-objects of $\mathcal{C}$, where the two commutative diagrams 
\begin{diagram}
F(\Gamma) &&\rTo^{F(\varphi)} && F(\Delta) && F(\Gamma) . F(A) \cong & F(\Gamma . A) &&\rTo^{\langle F(\varphi \circ \pi_1), F(g) \rangle} && F(\Delta) . F(B) \\
&&&&&&&&&&& \dDepTo_{t_{\Delta, B}} \\
\dTo_{\eta_\Gamma} &&&& \dTo_{\eta_{\Delta}} && \dTo_{\langle \eta_\Gamma \circ \pi_1, t_{\Gamma, A} \rangle} &&&&& G(B)\{\eta_\Delta \circ \pi_1\} \\
&&&&&&&&&&& \dImplies \\
G(\Gamma) && \rTo^{G(\varphi)} && G(\Delta) && G(\Gamma) . G(A) \cong & G(\Gamma . A) & \rDepTo^{G(g)} &G(B)\{G(\varphi \circ \pi_1)\}& \rImplies &G(B) \{G(\varphi) \circ \eta_\Gamma \circ \pi_1\}
\end{diagram}
for $(\eta, t)$ in $\mathcal{D}$ is unified into the following commutative diagram 
\begin{diagram}
(F(\Gamma), F(A)) &&&\rTo^{(F(\varphi), F(g)\{ \iota^{F(\Gamma . A)}_{F(\Gamma) . F(A)}\})} &&& (F(\Delta), F(B)) \\
&&&&&& \\
\dTo_{(\eta_\Gamma, t_{\Gamma, A})} &&&&&& \dTo_{(\eta_\Delta, t_{\Delta, B})} \\
&&&&&& \\
(G(\Gamma), G(A)) &&&\rTo^{(G(\varphi), G(g)\{ \iota^{G(\Gamma . A)}_{G(\Gamma) . G(A)}\})} &&& (G(\Delta), G(B))
\end{diagram}
in $\mathcal{D}_{\mathscr{D}}$ for any $\Gamma, \Delta \in \mathcal{C}$, $A \in \mathscr{D}_{\mathcal{C}}(\Gamma)$, $B \in \mathscr{D}_{\mathcal{C}}(\Delta)$, $\varphi : \Gamma \to \Delta$ and $g : \Gamma . A \rightarrowtriangle B\{\varphi \circ \pi_1\}$ in $\mathcal{C}$.

Thus, we may just apply known results on NTs to NTwDs; for instance:
\begin{lemma}[Naturality of vertical inverses of NTwDs]
Given an NTwD $(\eta, t) : F \Rightarrow G : \mathcal{C} \to \mathcal{D}$, if a pair $(\epsilon, s)$ of a family $\epsilon = (\epsilon_\Gamma)_{\Gamma \in \mathcal{C}}$ of morphisms $\epsilon_\Gamma : G(\Gamma) \to F(\Gamma)$ in $\mathcal{D}$ and a family $s = (s_{\Gamma, A})_{\Gamma \in \mathcal{C}, A \in \mathscr{D}_{\mathcal{C}}(\Gamma)}$ of D-morphisms $s_{\Gamma, A} : G(\Gamma) . G(A) \rightarrowtriangle F(A)\{\epsilon_\Gamma \circ \pi_1\}$ in $\mathcal{D}$ satisfies the equations
\begin{align*}
\epsilon_\Gamma \circ \eta_\Gamma &= \mathit{id}_{F(\Gamma)} \\
\eta_\Gamma \circ \epsilon_\Gamma &= \mathit{id}_{G(\Gamma)} \\
s_{\Gamma, A} \{ \langle \eta_{\Gamma} \circ \pi_1^{F(\Gamma) . F(A)}, t_{\Gamma, A} \rangle \} &= \pi_2^{F(\Gamma) . F(A)} \\
t_{\Gamma, A} \{ \langle \epsilon_{\Gamma} \circ \pi_1^{G(\Gamma) . G(A)}, s_{\Gamma, A} \rangle \} &= \pi_2^{G(\Gamma) . G(A)}
\end{align*}
for all $\Gamma \in \mathcal{C}$ and $A \in \mathscr{D}_{\mathcal{C}}(\Gamma)$, then the pair $(\epsilon, s)$ is an NTwD $G \Rightarrow F$.
\end{lemma}


Now, let us summarize the constructions $(\_)_{\mathscr{D}}$ given so far as follows:
\begin{definition}[The 2-functor $(\_)_{\mathscr{D}}$]
The 2-functor $(\_)_{\mathscr{D}} : \mathbb{SCC_D} \to \mathit{CAT}$ is given by $\mathcal{C} \mapsto \mathcal{C}_{\mathscr{D}}$ for small SCCwDs $\mathcal{C}$, $F \mapsto F_{\mathscr{D}}$ for SCFwDs $F$ and $(\eta, t) \mapsto (\eta, t)_{\mathscr{D}}$ for NTwDs $(\eta, t)$.
\end{definition}

\begin{proposition}[Well-defined $(\_)_{\mathscr{D}}$]
The map $(\_)_{\mathscr{D}}$ is in fact a 2-functor $\mathbb{SCC_D} \to \mathit{CAT}$.
\end{proposition}
\begin{proof}
It suffices to show that $(\_)_{\mathscr{D}}$ preserves composition and identities.
We first show the preservation of composition of 1-cells.
Let $\mathcal{C}$, $\mathcal{D}$ and $\mathcal{E}$ be SCCwDs.
Given SCFwDs $F : \mathcal{C} \to \mathcal{D}$ and $G : \mathcal{D} \to \mathcal{E}$, the object-maps of $(G \circ F)_{\mathscr{D}}$ and $G_{\mathscr{D}} \circ F_{\mathscr{D}}$ clearly coincide.
For their arrow-maps, consider any morphism $(\phi, f) : (\Delta, D) \to (\Gamma, C)$ in $\mathcal{C}$; then, we have: 
\begin{align*}
(G \circ F)_{\mathscr{D}}(\phi, f) &= (G \circ F(\phi), G \circ F(f) \{ \iota^{G \circ F(\Delta . D)}_{G \circ F(\Delta) . G \circ F(D)} \}) \\
&= (G(F(\phi)), G(F(f)) \{ G (\iota^{F(\Delta . D)}_{F(\Delta) . F(D)}) \circ \iota^{G(F(\Delta) . F(D))}_{G(F(\Delta)) . G(F(D))} \}) \\
&= (G(F(\phi)), G(F(f)) \{ G (\iota^{F(\Delta . D)}_{F(\Delta) . F(D)}) \} \{ \iota^{G(F(\Delta) . F(D))}_{G(F(\Delta)) . G(F(D))} \}) \\
&= (G(F(\phi)), G(F(f) \{ \iota^{F(\Delta . D)}_{F(\Delta) . F(D)} \}) \{ \iota^{G(F(\Delta) . F(D))}_{G(F(\Delta)) . G(F(D))} \}) \\
&= G_{\mathscr{D}}(F(\phi), F(f) \{ \iota^{F(\Delta . D)}_{F(\Delta) . F(D)} \}) \\
&= G_{\mathscr{D}} (F_{\mathscr{D}}(\phi, f)) \\
&= G_{\mathscr{D}} \circ F_{\mathscr{D}}(\phi, f).
\end{align*}
where it is easy to show $\iota^{G \circ F(\Delta . D)}_{G \circ F(\Delta) . G \circ F(D)} = G (\iota^{F(\Delta . D)}_{F(\Delta) . F(D)}) \circ \iota^{G(F(\Delta) . F(D))}_{G(F(\Delta)) . G(F(D))}$ by the equations 1, 2 and 3 of Lemma~\ref{LemSCFwDLemma}.
Hence, we have shown that $(G \circ F)_{\mathscr{D}} = G_{\mathscr{D}} \circ F_{\mathscr{D}}$.
Finally, the preservation of 1-cell identities is immediate, and the preservations of vertical and horizontal compositions of 2-cells, and of vertical and horizontal 2-cell identities are rather trivial.
\end{proof}

Given an SCCwD $\mathcal{C}$ and an object $\Gamma \in \mathcal{C}$, we have the subcategory of $\mathcal{C}_{\mathscr{D}}$ whose objects are pairs $(\Gamma, C)$ and morphisms are pairs $(\mathit{id}_\Gamma, f)$, where $C$ and $f$ vary, but $\Gamma$ is fixed.
Then, we may focus on the second components of such pairs, which is equivalent to the following category:
\begin{definition}[Categories of fixed objects]
Given an SCCwD $\mathcal{C}$ and an object $\Gamma \in \mathcal{C}$, the \emph{\bfseries category of a fixed object} $\Gamma$ on $\mathcal{C}$ is the category $\mathcal{C}_\Gamma$ given by: 
\begin{itemize}

\item $\mathsf{ob}(\mathcal{C}_\Gamma) \stackrel{\mathrm{df. }}{=} \mathscr{D}_{\mathcal{C}}(\Gamma)$;

\item $\mathcal{C}_\Gamma(A, B) \stackrel{\mathrm{df. }}{=} \mathscr{D}_{\mathcal{C}}(\Gamma . A, B\{ \pi_1 \})$;

\item $A \stackrel{f}{\to} B \stackrel{g}{\to} C \stackrel{\mathrm{df. }}{=} g \{ \langle \pi_1, f \rangle \}$; and

\item $\mathit{id}_{(\Gamma, A)} \stackrel{\mathrm{df. }}{=} \pi_2$.

\end{itemize}
\end{definition}


This structure will be essential when we prove that our generalizations of binary products and exponentials are certain adjoints, and thus they are unique up to isomorphisms. 

\subsubsection{Isomorphisms between Dependent Objects}
Next, we define \emph{isomorphisms between D-objects} as promised before:
\begin{definition}[Isomorphisms between D-objects]
Let $\mathcal{C}$ be an SCCwD. 
Given $\Gamma, \Delta \in \mathcal{C}$, $A \in \mathscr{D}_{\mathcal{C}}(\Gamma)$ and $B \in \mathscr{D}_{\mathcal{C}}(\Delta)$, an \emph{\bfseries isomorphism} between $A$ and $B$ is an isomorphism $(\Gamma, A) \stackrel{\sim}{\to} (\Delta, B)$ in $\mathcal{C}_{\mathscr{D}}$.
We call $A$ and $B$ \emph{\bfseries isomorphic} and write $A \cong B$ iff there is an isomorphism between them.
\end{definition}

It is easy to see that an inverse of an isomorphism between D-objects is unique.

\begin{example}
In the SCCwD $\mathit{Sets}$, D-sets $A \in \mathscr{D}_{\mathit{Sets}}(X)$ and $B \in \mathscr{D}_{\mathit{Sets}}(Y)$ are isomorphic iff there is a bijection $\varphi : X \stackrel{\sim}{\to} Y$, and the sets $A_x$ and $B_{\varphi(x)}$ are bijective for each $x \in X$.
\end{example}

Note that it is a fundamental property of functors to preserve isomorphisms between objects.
Similarly, SCFwDs preserve isomorphisms between D-objects as well:
\begin{proposition}[Preservation of isomorphisms between D-objects under SCFwDs]
SCFwDs preserve isomorphisms between D-objects.
\end{proposition}
\begin{proof}
Let $F : \mathcal{C} \to \mathcal{D}$ be an SCFwD, and $(\phi, f) : (\Delta, D) \stackrel{\sim}{\to} (\Gamma, C)$ in $\mathcal{C}_{\mathscr{D}}$.
Let $(\psi, g)$ be the inverse of $(\phi, f)$, i.e., $\psi \circ \phi = \mathit{id}_\Delta$, $\phi \circ \psi = \mathit{id}_\Gamma$, $g \{ \langle \phi \circ \pi_1, f \rangle \} = \pi_2$ and $f \{ \langle \psi \circ \pi_1, g \rangle \} = \pi_2$. 
It suffices to establishes the isomorphisms $(F(\phi), F(f)\{\iota\}) : (F(\Delta), F(B)) \cong (F(\Gamma), F(A)) : (F(\psi), F(g)\{\iota\})$ in $\mathcal{D}_{\mathscr{D}}$.
Clearly, $F(\psi) \circ F(\phi) = F(\psi \circ \phi) = F(\mathit{id}_\Delta) = \mathit{id}_{F(\Delta)}$; and we have:
\begin{align*}
F(g) \{ \iota \} \{ \langle F(\phi) \circ \pi_1, F(f)\{\iota\} \rangle \} &= F(g) \{ \iota \} \{ \langle F(\phi) \circ (F(\pi_1) \circ \iota), F(f)\{ \iota \} \rangle \} \ \text{(by Lemma~\ref{LemSCFwDLemma})} \\ 
&= F(g) \{ \iota \} \{ \langle (F(\phi) \circ F(\pi_1)) \circ \iota, F(f)\{ \iota \} \rangle \} \\
&= F(g) \{ \iota \} \{ \langle F(\phi \circ \pi_1) \circ \iota, F(f)\{ \iota \} \rangle \} \\ 
&= F(g) \{ \iota \} \{ \langle F(\phi \circ \pi_1), F(f) \rangle \circ \iota \} \\
&= F(g) \{ \iota \} \{ \iota^{-1} \circ F(\langle \phi \circ \pi_1, f \rangle) \circ \iota \} \ \text{(by Lemma~\ref{LemSCFwDLemma})} \\
&= F(g) \{ \iota \circ \iota^{-1} \circ F(\langle \phi \circ \pi_1, f \rangle) \circ \iota \} \\
&= F(g) \{ F(\langle \phi \circ \pi_1, f \rangle) \} \{ \iota \} \\
&= F(g\{\langle \phi \circ \pi_1, f \rangle\})\{\iota\} \\ 
&= F(\pi_2)\{\iota\} \\
&= \pi_2.
\end{align*}

Similarly, we have $F(\phi) \circ F(\psi) = \mathit{id}_{F(\Gamma)}$ and $F(f)  \{ \iota \} \{ \langle F(\psi) \circ \pi_1, F(g)\{\iota\} \rangle \} = \pi_2$, which completes the proof.
\end{proof}

\subsection{Cartesian Categories and Functors with Dependence}
\label{CCwDs}
We have seen in the last section that strict SCCwDs are CwFs, and thus they give a semantics of MLTTs yet without any specific type constructions.
In this section, we give a generalization of terminal objects and semi-$\Sigma$-spaces, called \emph{unit D-objects} and \emph{$\Sigma$-spaces}, which respectively interpret $\mathsf{1}$- and $\mathsf{\Sigma}$-types in MLTTs.
We also introduce SCFwDs that preserve these structures.

\subsubsection{Cartesian Categories with Dependence}
A naive idea to interpret $\mathsf{1}$-types is to employ D-objects that are to be called \emph{terminal}.
Note that we may define \emph{limits for D-objects and D-morphisms} in any SCCwD $\mathcal{C}$ to be just limits in the category $\mathcal{C}_{\mathscr{D}}$.
Therefore, we may define: 
\begin{definition}[Terminal and initial D-objects]
A D-object $A \in \mathscr{D}_{\mathcal{C}}(\Gamma)$ in an SCCwD $\mathcal{C}$ is \emph{\bfseries terminal} (resp. \emph{\bfseries initial}) iff so is the pair $(\Gamma, A)$ in the category $\mathcal{C}_{\mathscr{D}}$.
\end{definition}

Since a terminal (resp. initial) D-object in an SCCwD is unique up to isomorphisms, we may call it \emph{the} terminal (resp. initial) D-object.

Notice, however, that terminal D-objects cannot interpret $\mathsf{1}$-type because their universal property does not exactly match the uniqueness rule $1$-Uniq.
Thus, we introduce:
\begin{definition}[Unit D-objects]
A D-object $1 \in \mathscr{D}_{\mathcal{C}}(T)$ in an SCCwD $\mathcal{C}$ is \emph{\bfseries unit} iff $T \in \mathcal{C}$ is a terminal object in $\mathcal{C}$, and there is a unique D-morphism $\bm{!}_\Gamma : \Gamma \rightarrowtriangle 1\{!_\Gamma\}$\footnote{Note the notational difference between $!_\Gamma$ and $\bm{!}_\Gamma$.} for each object $\Gamma \in \mathcal{C}$.
\end{definition}

\begin{example}
In a cartesian category $\mathcal{C}$ seen as an SCCwD $\mathsf{D}(\mathcal{C})$, a terminal object $T$ over itself is also a unit D-object.
\end{example}

\begin{proposition}[Uniqueness of unit D-objects]
Unit D-objects are unique up to isomorphisms. 
\end{proposition}
\begin{proof}
Let $1 \in \mathscr{D}_{\mathcal{C}}(T)$ and $1' \in \mathscr{D}_{\mathcal{C}}(T')$ be both unit D-objects in an SCCwD $\mathcal{C}$.
Then, we have $(!_{T'}, \bm{!}_{1'}) : (T, 1) \cong (T', 1') : (!_T, \bm{!}_1)$ in the category $\mathcal{C}_{\mathscr{D}}$, i.e., $1 \cong 1'$ in $\mathcal{C}$.
\end{proof}

It is almost immediate by the definition that unit D-objects capture $\mathsf{1}$-type in MLTTs:
\begin{theorem}[Unit D-objects interpret the $\mathsf{1}$-type]
\label{ThmUnit}
A strict SCCwD equipped with a unit D-object gives rise to a CwF that supports $\mathsf{1}$-type in the strict sense.
\end{theorem}
\begin{proof}
Let $\mathcal{C}$ be a strict SCCwD equipped with a unit D-object $1 \in \mathscr{D}_{\mathcal{C}}(T)$.
For each object $\Gamma \in \mathcal{C}$, we define $1_\Gamma \stackrel{\mathrm{df. }}{=} 1\{!_\Gamma\} \in \mathscr{D}_{\mathcal{C}}(\Gamma)$, which interprets the rule Unit-Intro. 
Then, it is clear how the remaining axioms for a CwF $\mathcal{C}$ that supports $1$-type in the strict sense are satisfied.
\end{proof}

Note that unit D-objects are terminal, but not necessarily vice versa.
In fact, this pattern will occur repeatedly: Limits in the category of D-objects $\mathcal{C}_{\mathscr{D}}$ on an SCCwD $\mathcal{C}$ in general do not give interpretations of type constructions of MLTTs in $\mathcal{C}$, but their stronger version does.

Next, let us consider \emph{(binary) products} of D-objects.
However, if we define the binary product of D-objects $A \in \mathscr{D}_{\mathcal{C}}(\Gamma)$ and $B \in \mathscr{D}_{\mathcal{C}}(\Gamma . A)$ in an SCCwD $\mathcal{C}$ as the binary product $(\Gamma, A) \times (\Gamma . A, B)$ in the category $\mathcal{C}_{\mathscr{D}}$, then it would be  a D-object $A \times B \in \mathscr{D}_{\mathcal{C}}(\Gamma \times \Gamma.A)$ in $\mathcal{C}$, which does not match the rule $\mathsf{\Sigma}$-Form.
Instead, what we need to interpret $\mathsf{\Sigma}$-types is the following:
\begin{definition}[$\Sigma$-spaces]
Let $\mathcal{C}$ be an SCCwD.
Given $\Gamma \in \mathcal{C}$, $A \in \mathscr{D}_{\mathcal{C}}(\Gamma)$ and $B \in \mathscr{D}_{\mathcal{C}}(\Gamma . A)$, a \emph{\bfseries dependent pair ($\bm{\Sigma}$-) space} of $A$ and $B$ in $\mathcal{C}$ is a D-object $\Sigma(A, B) \in \mathscr{D}_{\mathcal{C}}(\Gamma)$ together with D-morphisms
\begin{equation*}
A\{\pi_1^{\Gamma . \Sigma(A, B)}\} \stackrel{\varpi_1^{\Sigma(A, B)}}{\leftarrowtriangle} \Gamma . \Sigma(A, B) \stackrel{\varpi_2^{\Sigma(A, B)}}{\rightarrowtriangle} B \{ \langle \pi_1^{\Gamma . \Sigma(A, B)}, \varpi_1^{\Sigma(A, B)} \rangle \}
\end{equation*}
in $\mathcal{C}$, called the \emph{\bfseries first} and \emph{\bfseries second projections} of $\Sigma(A, B)$, respectively, such that for any $\phi : \Delta \to \Gamma$, $g : \Delta \rightarrowtriangle A \{ \phi \}$ and $h : \Delta \rightarrowtriangle B \{ \langle \phi, g \rangle \}$ in $\mathcal{C}$ there exists a unique D-morphism  
\begin{equation*}
\Lbag g, h \Rbag : \Delta \rightarrowtriangle \Sigma(A, B) \{ \phi \}
\end{equation*}
in $\mathcal{C}$, called the \emph{\bfseries dependent pairing} of $g$ and $h$, that satisfies 
\begin{align}
\label{UP1Sigma}
\varpi_1^{\Sigma(A, B)} \{ \langle \phi, \Lbag g, h \Rbag \rangle \} &= g \\
\label{UP2Sigma}
\varpi_2^{\Sigma(A, B)} \{ \langle \phi, \Lbag g, h \Rbag \rangle \} &= h.
\end{align}
\end{definition}

\begin{notation}
We often omit the superscript $\Sigma(A, B)$ on $\varpi_i^{\Sigma(A, B)}$ ($i = 1, 2$).
\end{notation}

The universal property of a dependent pairing $\Lbag g, h \Rbag : \Delta \rightarrowtriangle \Sigma(A, B) \{ \phi \}$ may be described as the following commutative diagram:
\begin{diagram}
\label{DPS}
A\{\phi\} && & \lDepTo^g && \Delta &  \rDepTo^{h}  & && B \{ \langle \phi, g \rangle \} \\
 \uImplies&&& & \ldTo(3, 2)^{\phi} &  & && & \uImplies \\
 &&\Gamma&&&\dDotsto_{\langle \phi, \Lbag g, h \Rbag \rangle}&&&& \\
 &&&\luTo(3, 2)^{\pi_1}&&&&&& \\
A \{ \pi_1 \} && \lDepTo^{\varpi_1} &&& \Gamma . \Sigma(A, B) && \rDepTo^{\varpi_2} & & B\{ \langle \pi_1, \varpi_1 \rangle \}
\end{diagram}

\begin{proposition}[Uniqueness of $\Sigma$-spaces]
\label{PropUniquenessOfSigmaSpaces}
$\Sigma$-spaces are unique up to isomorphisms.
\end{proposition}
\begin{proof}
Similar to the proof of the uniqueness of semi-$\Sigma$-spaces.
\end{proof}

\begin{definition}[CCwDs]
A \emph{\bfseries cartesian CwD (CCwD)} is an SCCwD $\mathcal{C}$ that has:
\begin{itemize}

\item A unit D-object $1 \in \mathscr{D}_{\mathcal{C}}(T)$;

\item A $\Sigma$-space $A\{ \pi_1 \} \stackrel{\varpi_1}{\leftarrowtriangle} \Gamma . \Sigma(A, B) \stackrel{\varpi_2}{\rightarrowtriangle} B \{ \langle \pi_1, \varpi_1 \rangle \}$ for any triple of $\Gamma \in \mathcal{C}$, $A \in \mathscr{D}_{\mathcal{C}}(\Gamma)$ and $B \in \mathscr{D}_{\mathcal{C}}(\Gamma . A)$.

\end{itemize}
A \emph{\bfseries strict CCwD} is a strict SCCwD $\mathcal{C} = (\mathcal{C}, T, \_ . \_, \pi)$ equipped with a unit D-object $1 \in \mathscr{D}_{\mathcal{C}}(T)$ and a family $\Sigma = (\Sigma(A, B))_{\Gamma \in \mathcal{C}, A \in \mathscr{D}_{\mathcal{C}}(\Gamma), B \in \mathscr{D}_{\mathcal{C}}(\Gamma . A)}$ of $\Sigma$-spaces $\Sigma(A, B) \in \mathscr{D}_{\mathcal{C}}(\Gamma)$ that is \emph{\bfseries coherent} in the sense that for any $\Delta, \Gamma \in \mathcal{C}$, $A \in \mathscr{D}_{\mathcal{C}}(\Gamma)$, $B \in \mathscr{D}_{\mathcal{C}}(\Gamma . A)$ and $\phi : \Delta \to \Gamma$ in $\mathcal{C}$ it satisfies
\begin{align}
\label{Coh0Sigma}
\Sigma(A, B) \{ \phi \} &= \Sigma(A \{ \phi \}, B \{ \phi^+ \}) \\
\label{Coh1Sigma}
\varpi_1^{\Sigma(A, B)} \circ \phi^\star &= \varpi_1^{\Sigma(A\{\phi\}, B\{\phi^+\})} \\
\label{Coh2Sigma}
\varpi_2^{\Sigma(A, B)} \circ \phi^\star &= \varpi_2^{\Sigma(A\{\phi\}, B\{\phi^+\})}
\end{align}
where $\phi^+ \stackrel{\mathrm{df. }}{=} \langle \phi \circ \pi_1, \pi_2 \rangle : \Delta . A \{ \phi \} \to \Gamma . A$ and $\phi^\star \stackrel{\mathrm{df. }}{=} \langle \phi \circ \pi_1, \pi_2 \rangle : \Delta . \Sigma(A, B) \{ \phi \} \to \Gamma . \Sigma(A, B)$.
\end{definition}

\begin{example}
\label{ExCCsSeenAsCDCs}
A cartesian category $\mathcal{C}$ seen as an SCCwD $\mathsf{D}(\mathcal{C})$ is a CCwD as we may define $\Sigma(A, B) \stackrel{\mathrm{df. }}{=} A \times B$, $\varpi_1^{\Sigma(A, B)} \stackrel{\mathrm{df. }}{=} \pi_1^{A \times B} \circ \pi_2^{\Gamma \times (A \times B)}$ and $\varpi_2^{A \times B} \stackrel{\mathrm{df. }}{=} \pi_2^{A \times B} \circ \pi_2^{\Gamma \times (A \times B)}$ for any $\Gamma \in \mathcal{C}$, $A \in \mathscr{D}_{\mathsf{D}(\mathcal{C})}(\Gamma)$ and $B \in \mathscr{D}_{\mathsf{D}(\mathcal{C})}(\Gamma . A)$.
We call it a \emph{\bfseries cartesian category seen as a CCwD}. 
\end{example}

\begin{proposition}[Associativity of $\Sigma$-spaces]
\label{PropAssociativityOfSigmaSpaces}
Let $\mathcal{C}$ be an SCCwD, and $\Gamma \in \mathcal{C}$, $A \in \mathscr{D}_{\mathcal{C}}(\Gamma)$, $B \in \mathscr{D}_{\mathcal{C}}(\Gamma . A)$ and $C \in \mathscr{D}_{\mathcal{C}}(\Gamma . \Sigma(A, B))$.
Then, we have isomorphisms 
\begin{align*}
\mathit{Pair}_{\Gamma . \Sigma(A, B)}^{\Sigma(\Gamma . A, B)} : \Gamma . A . B &\stackrel{\sim}{\to} \Gamma . \Sigma(A, B) \\
\mathit{Triple}_{\Sigma(\Sigma(A, B), C)}^{\Sigma(A, \Sigma(B, C\{\mathit{Pair}_{\Gamma . \Sigma(A, B)}^{\Sigma(\Gamma . A, B)}\}))} : \Sigma(\Sigma(A, B), C) &\stackrel{\sim}{\to} \Sigma(A, \Sigma(B, C\{\mathit{Pair}_{A, B}\})).
\end{align*}
\end{proposition}
\begin{proof}
For brevity, let us omit the subscripts and superscripts, and write $\mathit{Pair}$ and $\mathit{Triple}$ for the isomorphisms to establish.
Let us define: 
\begin{equation*}
\mathit{Pair} \stackrel{\mathrm{df. }}{=} \langle \pi_1 \circ \pi_1, \Lbag \pi_2 \{ \pi_1 \}, \pi_2 \Rbag \rangle : \Gamma . A . B \to \Gamma . \Sigma(A, B)
\end{equation*}
which has $\langle \langle \pi_1, \varpi_1 \rangle, \varpi_2 \rangle : \Gamma . \Sigma(A, B) \to \Gamma . A . B$ as the inverse $\mathit{Pair}^{-1}$.

Similarly, we define: 
\begin{equation*}
\mathit{Triple} \stackrel{\mathrm{df. }}{=} (\mathit{id}_\Gamma, \Lbag \varpi_1\{\langle \pi_1, \varpi_1 \rangle\}, \Lbag \varpi_2\{\langle \pi_1, \varpi_1 \rangle\}, \varpi_2 \Rbag \Rbag) : \Sigma(\Sigma(A, B), C) \to \Sigma(A, \Sigma(B, C\{\mathit{Pair}\}))
\end{equation*}
which has the inverse $\mathit{Triple}^{-1} \stackrel{\mathrm{df. }}{=} (\mathit{id}_\Gamma, \Lbag \Lbag \varpi_1, \varpi_1\{\langle \langle \pi_1, \varpi_1 \rangle, \varpi_2 \rangle \} \Rbag, \varpi_2\{\langle \langle \pi_1, \varpi_1 \rangle, \varpi_2 \rangle \} \Rbag)$.
\end{proof}

\begin{notation}
For \emph{strict} SCCwDs, we rather write $\mathit{Pair}_{A, B}$ and $\mathit{Triple}_{A, B, C}$ for $\mathit{Pair}_{\Gamma . \Sigma(A, B)}^{\Sigma(\Gamma . A, B)}$ and $\mathit{Triple}_{\Sigma(\Sigma(A, B), C)}^{\Sigma(A, \Sigma(B, C\{\mathit{Pair}_{\Gamma . \Sigma(A, B)}^{\Sigma(\Gamma . A, B)}\}))}$, respectively. 
Even for non-strict SCCwDs, however, we often abuse the notation and employ $\mathit{Pair}_{A, B}$ and $\mathit{Triple}_{A, B, C}$ for brevity.
\end{notation}

\begin{proposition}[Unit law of $\Sigma$-spaces]
Let $\mathcal{C}$ be an SCCwD, and $\Gamma \in \mathcal{C}$ and $A \in \mathscr{D}_{\mathcal{C}}(\Gamma)$.
If $1 \in \mathscr{D}_{\mathcal{C}}(T)$ is a unit D-object, then we have $\Sigma(1\{!_\Gamma\}, A\{\pi_1\}) \cong A \cong \Sigma(A, 1\{!_\Gamma\})$.
\end{proposition}
\begin{proof}
Immediate from Proposition~\ref{PropUniquenessOfSigmaSpaces}.
\end{proof}

Thus, as promised before, we have shown that $\Sigma$-spaces have a unit D-object $1$ and satisfy associativity, generalizing the corresponding properties of binary products.

Now, let us show that coherent $\Sigma$-spaces capture strict $\mathsf{\Sigma}$-types in MLTTs:
\begin{theorem}[Coherent $\Sigma$-spaces interpret $\mathsf{\Sigma}$-types]
\label{ThmSigma}
A strict SCCwD equipped with coherent $\Sigma$-spaces induces a CwF that supports $\mathsf{\Sigma}$-types in the strict sense.
\end{theorem}
\begin{proof}
Let $\mathcal{C} = (\mathcal{C}, T, \_ . \_, \pi)$ be a strict SCCwD equipped with a coherent family of $\Sigma$-spaces $\Sigma = (\Sigma, \varpi, \Lbag, \_ \Rbag)$.
We equip the corresponding CwF $\mathcal{C}$ with $\Sigma$-types in the strict sense as follows:
\begin{itemize}

\item \textsc{($\Sigma$-Form)} Given $\Gamma \in \mathcal{C}$, $A \in \mathscr{D}_{\mathcal{C}}(\Gamma)$ and $B \in \mathscr{D}_{\mathcal{C}}(\Gamma . A)$, we have:
\begin{equation*}
\Sigma (A, B) \in \mathscr{D}_{\mathcal{C}}(\Gamma).
\end{equation*}

\item \textsc{($\Sigma$-Intro)} As in the proof of Proposition~\ref{PropAssociativityOfSigmaSpaces}, we define: 
\begin{equation*}
\mathit{Pair}_{A, B} \stackrel{\mathrm{df. }}{=} \langle \pi_1 \circ \pi_1, \Lbag \pi_2 \{ \pi_1 \}, \pi_2 \Rbag \rangle : \Gamma . A . B \to \Gamma . \Sigma(A, B)
\end{equation*}
with the inverse given by:
\begin{equation*}
\mathit{Pair}_{A, B}^{-1} \stackrel{\mathrm{df. }}{=} \langle \langle \pi_1, \varpi_1 \rangle, \varpi_2 \rangle : \Gamma . \Sigma(A, B) \to \Gamma . A . B.
\end{equation*}

\item \textsc{($\Sigma$-Elim)} Given $P \in \mathscr{D}_{\mathcal{C}}(\Sigma(\Gamma, \Sigma(A,B)))$ and $f \in \mathscr{D}_{\mathcal{C}}(\Sigma(\Sigma(\Gamma, A), B), P\{\mathit{Pair}_{A,B}\})$, let us define: 
\begin{equation*}
\mathcal{R}^{\Sigma}_{A, B, P}(f) \stackrel{\mathrm{df. }}{=} f \{ \mathit{Pair}_{A, B}^{-1} \} \in \mathscr{D}_{\mathcal{C}}(\Sigma(\Gamma, \Sigma (A, B)), P).
\end{equation*}

\item \textsc{($\Sigma$-Comp)} We clearly have:
\begin{align*}
\mathcal{R}^{\Sigma}_{A, B, P}(f) \{ \mathit{Pair}_{A, B} \} &= f \{ \mathit{Pair}_{A, B}^{-1} \} \{ \mathit{Pair}_{A, B} \} \\
&= f \{\mathit{Pair}_{A, B}^{-1} \circ \mathit{Pair}_{A, B} \} \\
&= f \{ \mathit{id}_{\Gamma . A . B} \} \\
&= f.
\end{align*}

\item \textsc{($\Sigma$-Subst)} By the axiom (\ref{Coh0Sigma}).

\item \textsc{(Pair-Subst)} We clearly have: 
\begin{align*}
\mathit{p}(\Sigma (A, B)) \circ \mathit{Pair}_{A, B} &= \pi_1 \circ \langle \pi_1 \circ \pi_1, \Lbag \pi_2 \{ \pi_1 \}, \pi_2 \Rbag \rangle \\
&= \pi_1 \circ \pi_1 \\
&= \mathit{p}(A) \circ \mathit{p}(B)
\end{align*}
as well as: 
\begin{align*}
\phi^\star \circ \mathit{Pair}_{A\{\phi\}, B\{\phi^+\}} &= \langle \phi \circ \pi_1, \pi_2 \rangle \circ \langle \pi_1 \circ \pi_1, \Lbag \pi_2 \{ \pi_1 \}, \pi_2 \Rbag \rangle \\
&= \langle \phi \circ \pi_1 \circ \pi_1, \Lbag \pi_2 \{ \pi_1 \}, \pi_2 \Rbag \rangle \\
&= \langle \pi_1 \circ \pi_1, \Lbag \pi_2 \{ \pi_1 \}, \pi_2 \Rbag \rangle \circ \langle \langle \phi \circ \pi_1, \pi_2 \rangle \circ \ \pi_1, \pi_2 \rangle \\
&= \mathit{Pair}_{A, B} \circ \phi^{++} 
\end{align*}
where 
\begin{align*}
\phi^\star &\stackrel{\mathrm{df. }}{=} \langle \phi \circ \pi_1, \pi_2 \rangle : \Delta . \Sigma(A, B)\{ \phi \} \to \Gamma . \Sigma(A, B) \\ 
\phi^{++} &\stackrel{\mathrm{df. }}{=} \langle \phi^+ \circ \pi_1, \pi_2 \rangle : (\Delta . A\{\phi\}) . B\{\phi^+\} \to \Gamma . A . B. 
\end{align*}

\item \textsc{($\mathcal{R}^{\Sigma}$-Subst)} We have:
\begin{align*}
\textstyle \mathcal{R}^{\Sigma}_{A, B, P}(f) \{\phi^\star \} &= f \{ \mathit{Pair}_{A, B}^{-1} \} \{ \langle \phi \circ \pi_1, \pi_2 \rangle \} \\
&= f \{ \langle \langle  \pi_1, \varpi^{\Sigma(A, B)}_1 \rangle, \varpi^{\Sigma(A, B)}_2 \rangle \circ \langle \phi \circ \pi_1, \pi_2 \rangle \} \\
&= f \{ \langle \langle \phi \circ \pi_1, \varpi^{\Sigma(A, B)}_1 \{ \langle \phi \circ \pi_1, \pi_2 \rangle \} \rangle, \varpi^{\Sigma(A, B)}_2 \{ \langle \phi \circ \pi_1, \pi_2 \rangle \} \rangle \} \\
&= f \{ \langle \langle \phi \circ \pi_1, \varpi^{\Sigma(A\{ \phi \}, B\{\phi^+\})}_1 \rangle, \varpi^{\Sigma(A\{ \phi \}, B\{\phi^+\})}_2 \rangle \} \ \text{(because $\Sigma$ is coherent)} \\
&= f \{ \langle \phi^+ \circ \langle \pi_1, \varpi^{\Sigma(A\{ \phi \}, B\{\phi^+\})}_1 \rangle, \varpi^{\Sigma(A\{ \phi \}, B\{\phi^+\})}_2 \rangle \} \\
&= f \{ \langle \phi^+ \circ \pi_1, \pi_2 \rangle \} \{ \langle \langle \pi_1, \varpi^{\Sigma(A\{ \phi \}, B\{\phi^+\})}_1 \rangle, \varpi^{\Sigma(A\{ \phi \}, B\{\phi^+\})}_2 \rangle \} \\
&= f \{ \phi^{++} \} \{ \mathit{Pair}_{A\{\phi\}, B\{\phi^+\}}^{-1} \} \\
&= \mathcal{R}^{\Sigma}_{A\{\phi\}, B\{\phi^+\}, P\{\phi^\star \}} (f \{ \phi^{++} \}).
\end{align*}

\item \textsc{($\Sigma$-Uniq)} Finally, given $g \in \mathscr{D}_{\mathcal{C}}(\Gamma . \Sigma(A, B), P)$ such that $g \{ \mathit{Pair}_{A, B} \} = f$, we have:
\begin{align*}
\mathcal{R}^\Sigma_{A, B, P}(f) &= f \{ \mathit{Pair}_{A, B}^{-1} \} \\
&= g \{ \mathit{Pair}_{A, B} \}\{ \mathit{Pair}_{A, B}^{-1} \} \\
&= g \{ \mathit{Pair}_{A, B}  \circ \mathit{Pair}_{A, B}^{-1} \} \\
&= g\{ \mathit{id}_{\Gamma . \Sigma(A, B)} \} \\
&= g
\end{align*}

\end{itemize}
which completes the proof.
\end{proof}

Note that strict $\Sigma$-types in CwFs are reformulated more concisely as strict $\Sigma$-spaces in SCCwDs. 
In particular, we have rather \emph{induced} the morphism $\mathit{Pair}_{A, B} : \Gamma . A . B \to \Gamma . \Sigma(A, B)$ from projections, which seems rather canonical, while it is taken as \emph{primitive} in NMs \cite{awodey2016natural}.

On the other hand, strict $\Sigma$-types in CwFs do not always induce $\Sigma$-spaces in the corresponding strict SCCwDs as, e.g., there are no obvious projections $\varpi_i$. 
However, we claim that strict CCwDs are a \emph{refinement} of CwFs that support $1$- and $\Sigma$-types in the strict sense, rather than a restriction, as the following instances, including the term models, are formed via CCwDs:

\begin{example}
The SCCwD $\mathit{Sets}$ is cartesian.
A unit D-object in $\mathit{Sets}$ is any singleton D-set $1 = (1_\bullet)_{\bullet \in T}$ such that $1_\bullet = T$.
Given a set $X$ and D-sets $A$ over $X$ and $B$ over $X.A$, the $\Sigma$-space $\Sigma(A, B)$ in $\mathit{Sets}$ is the D-set $(\{ (a, b) \mid a \in A_x, b \in B_{(x, a)} \ \! \})_{x \in X}$ over $X$ equipped with the obvious projections and dependent pairings.
The SCCwD $\mathit{Rel}$ is also cartesian in a similar manner. 
\end{example}

\begin{example}
The SCCwD $\mathcal{GPD}$ is cartesian.
A unit D-object in $\mathcal{GPD}$ is a D-groupoid $1 : T \to \mathcal{GPD}$ that maps $\star \mapsto T$ and $\mathit{id}_\star \mapsto \mathit{id}_T$.
Given D-groupoids $A : \Gamma \to \mathcal{GPD}$ and $B : \Gamma . A \to \mathcal{GPD}$, the $\Sigma$-space $\Sigma(A, B) : \Gamma \to \mathcal{GPD}$ in $\mathcal{GPD}$ maps each object $\gamma \in \Gamma$ to the groupoid $A(\gamma) . B_\gamma$, where the D-groupoid $B_\gamma : A(\gamma) \to \mathcal{GPD}$ maps each object $a \in A(\gamma)$ to the groupoid $B(\gamma, a)$ and each isomorphism $\alpha : a \stackrel{\sim}{\to} a'$ in $A(\gamma)$ to the functor $B(\mathit{id}_\gamma, \alpha) : B(\gamma, a) \to B(\gamma, a')$, and each isomorphism $\phi : \gamma \stackrel{\sim}{\to} \gamma'$ in $\Gamma$ to the functor $A(\phi) . B(\phi, A(\phi)) : A(\gamma) . B_\gamma \to A(\gamma') . B_{\gamma'}$ that maps each object $(a, b) \in A(\gamma) . B_\gamma$ to the groupoid $(A(\phi)(a), B(\phi, A(\phi))(b))$ and each isomorphism $(\alpha, \beta) : (a, b) \to (a', b')$ in $A(\gamma) . B_\gamma$ to the functor $(A(\phi)(\alpha), B(\phi, A(\phi))(\beta)) : (A(\phi)(a), B(\phi, A(\phi))(b)) \to (A(\phi)(a'), B(\phi, A(\phi))(b'))$.
The projections and dependent pairings of $\Sigma(A, B)$ are the obvious ones. 
\end{example}

\begin{example}
In a similar manner, we may show that the SCCwD $\mathit{CAT}$ is also cartesian. 
\end{example}

\begin{example}
The term model $\mathcal{T}(1, \Pi, \Sigma)$ is a strict CCwD.
The unit D-object is $\mathsf{\diamondsuit \vdash 1 \ type}$, and the $\Sigma$-space of $\mathsf{\Gamma \vdash A \ type}$ and $\mathsf{\Gamma, x : A \vdash B \ type}$ is $\mathsf{\Gamma \vdash \Sigma(A, B) \ type}$ equipped with projections
\begin{align*}
&\mathsf{\Gamma, p : \Sigma(A, B) \vdash \pi_1(p) : A} \\
&\mathsf{\Gamma, p : \Sigma(A, B) \vdash \pi_2(p) : B[\pi_1(p)/x]}
\end{align*}
and the dependent pairing of $\mathsf{\Delta \vdash g : A[\bm{\mathsf{d}}]}$ and $\mathsf{\Delta \vdash h : B[\langle \bm{\mathsf{d}}, g \rangle]}$, where $\bm{\mathsf{d}} : \mathsf{\Delta} \to \mathsf{\Gamma}$, is the term
\begin{equation*}
\mathsf{\Delta \vdash \langle g, h \rangle : \Sigma(A, B)[\bm{\mathsf{d}}]}.
\end{equation*}
\end{example}

\subsubsection{Dependent Limits}
We have seen that a unit D-object and $\Sigma$-spaces model $\mathsf{1}$- and $\mathsf{\Sigma}$-types in MLTTs.
However, one may wonder if there is any systematic way to understand these constructions, which would be similar to limits that subsume finite products.
This section briefly addresses this point by introducing the notion of \emph{dependent limits (D-limits)}, which is a natural generalization of limits.
As such, the present section is a detour, and thus the reader may skip it without any problem.

\begin{definition}[D-cones]
Given an FwD $F : \mathcal{I} \to \mathcal{C}$, a \emph{\bfseries dependent (D-) cone} to $F$ is a triple $(\Theta, \phi, f)$ of an object $\Theta \in \mathcal{C}$, a family $\phi = (\phi_i)_{i \in \mathcal{I}}$ of morphisms $\phi_i : \Theta \to F(i)$ in $\mathcal{C}$ and a family $f = (f_{i, X})_{i \in \mathcal{I}, X \in \mathscr{D}_{\mathcal{I}}(i)}$ of D-morphisms $f_{i, X} : \Theta \rightarrowtriangle F(X)\{ \phi_i \}$ such that for any morphism $\mu : i \to j$ in $\mathcal{I}$ the diagram
\begin{diagram}
&&&&F(i) \\
&&&\ruTo(4, 4)^{\phi_i} \\
&&&&\dTo_{F(\mu)} \\
\\
\Theta&&\rTo^{\phi_j}&&F(j)
\end{diagram}
in $\mathcal{C}$ commutes, and for any D-morphism $x : i \rightarrowtriangle X$ in $\mathcal{I}$ the diagram
\begin{diagram}
&&&&&F(i) \\ 
&&&&\ruTo(5,5)^{\phi_i}& \dDepTo_{F(x)} \\ 
&&&&& \\
&&&&&F(X) \\
&&&&& \dImplies \\
\Theta&&\rDepTo(3, 2)^{f_{i, X}}&&&F(X)\{\phi_i\}
\end{diagram}
in $\mathcal{C}$ commutes.
\end{definition}

\begin{definition}[Categories of D-cones]
Given an FwD $F : \mathcal{I} \to \mathcal{C}$, the category $\mathsf{DCones}(F)$ is defined by:
\begin{itemize}

\item Objects are D-cones to $F$;

\item A morphism $(\Theta, \phi, f) \to (\Xi, \psi, g)$ is a morphism $\sigma : \Theta \to \Xi$ in $\mathcal{C}$ such that the diagrams 
\begin{diagram}
\Theta && \rTo^{\phi_i} && F(i) &&& \Theta && \rDepTo^{f_{i, X}} && F(X)\{ \phi_i \} \\
&&&\ruTo(4, 4)_{\psi_i}&&&&&&&\ruImplies(2, 2)& \\
\dTo^\sigma&&&& &&& \dTo^\sigma&&F(X)\{\psi_i\}&& \\
&&&&&&&&\ruDepTo(2, 2)_{g_{i, X}} \\
\Xi&&&&&&& \Xi 
\end{diagram}
in $\mathcal{C}$ commute for all $i \in \mathcal{I}$ and $X \in \mathscr{D}_{\mathcal{I}}(i)$;

\item Composition in $\mathsf{DCones}(F)$ is the composition in $\mathcal{C}$;

\item Identities in $\mathsf{DCones}(F)$ are the identities in $\mathcal{C}$.

\end{itemize}
\end{definition}

It is easy to see that $\mathsf{DCones}(F)$ is a well-defined category for any FwD $F : \mathcal{I} \to \mathcal{C}$.
We are now ready to define:
\begin{definition}[D-limits]
Let $\mathcal{C}$ be a CwD and $\mathcal{I}$ a small CwD. 
A \emph{\bfseries dependent (D-) limit} of shape $\mathcal{I}$ in $\mathcal{C}$ is the terminal object in the category $\mathsf{DCones}(F)$ for any FwD $F : \mathcal{I} \to \mathcal{C}$.
\end{definition}

Clearly, D-limits in a category $\mathcal{C}$ seen as a CwD $\mathsf{D}(\mathcal{C})$ coincide with limits in $\mathcal{C}$; in this sense, D-limits generalize limits. 
On the other hand, by the asymmetry of the domain and codomain of D-morphisms, what should be called \emph{dependent (D-) colimits} would be more complicated than D-limits.
For the lack of space, we leave the concept of D-limits as future work.

\begin{example}
Let $\emptyset$ denote the trivial CwD which has no objects. 
The D-limit of shape $\emptyset$ in a CwD $\mathcal{C}$ is just a terminal object of $\mathcal{C}$.
If we focus on a terminal object of the form $T . A \in \mathcal{C}$, where $T \in \mathcal{C}$ is terminal and $A \in \mathscr{D}_{\mathcal{C}}(T)$, then such D-objects $A$ coincides with unit D-objects:
\begin{enumerate}

\item Given an object $\Gamma \in \mathcal{C}$, there is a D-morphism $\Gamma \stackrel{!_\Gamma}{\to} T . A \stackrel{\pi_2}{\rightarrowtriangle} A\{\pi_1\} \Rightarrow A\{!_\Gamma\}$; 

\item Any D-morphism $f : \Gamma \rightarrowtriangle A\{!_\Gamma\}$ coincides with $\pi_2 \{ !_\Gamma \}$ because
\begin{align*}
f = \pi_2 \{\langle !_\Gamma, f \rangle\} = \pi_2 \{ !_\Gamma \}.
\end{align*}

\end{enumerate}
\end{example}

This pattern applies to $\Sigma$-spaces as well:
\begin{example}
Let $\mathcal{I}$ be an SCCwD with three objects $i, i.X, i.X.Y \in \mathcal{I}$ and two D-objects $X \in \mathscr{D}_{\mathcal{I}}(i)$ and $Y \in \mathscr{D}_{\mathcal{I}}(i.X)$.
The D-limit of shape $\mathcal{I}$ in a SCCwD $\mathcal{C}$ is any diagram 
\begin{diagram}
A\{\psi\} &&\lDepTo_{\varpi_1}&& \Theta &&\rDepTo_{\varpi_2}&& B\{\langle \psi, \varpi_1 \rangle\} \\
&&&&\dTo_\psi&&&& \\
&&&&\Gamma&&&&
\end{diagram}
in $\mathcal{C}$ such that for any diagram
\begin{diagram}
A\{\phi\} &&\lDepTo_{g}&& \Delta &&\rDepTo_{h}&& B\{\langle \phi, g \rangle\} \\
&&&&\dTo_\phi&&&& \\
&&&&\Gamma&&&&
\end{diagram}
in $\mathcal{C}$ there exists a unique morphism $\Upsilon(\phi, g, h) : \Delta \to \Theta$ such that the diagram
\begin{diagram}
A\{\phi\} && & \lDepTo^g && \Delta &  \rDepTo^{h}  & && B \{ \langle \phi, g \rangle \} \\
 \uImplies&&& & \ldTo(3, 2)^{\phi} &  & && & \uImplies \\
 &&\Gamma&&&\dDotsto_{\Upsilon(\phi, g, h)}&&&& \\
 &&&\luTo(3, 2)^{\psi}&&&&&& \\
A \{ \psi \} && \lDepTo^{\varpi_1} &&& \Theta && \rDepTo^{\varpi_2} & & B\{ \langle \psi, \varpi_1 \rangle \}
\end{diagram}
commutes.
Note that we have omitted semi-$\Sigma$-spaces and semi-dependent pairings in the above diagrams because they are redundant. 
If $\Theta$ is of the form $\Gamma . C$ for some $C \in \mathscr{D}_{\mathcal{C}}(\Gamma)$, then $\psi = \pi_1$, and thus the D-object $C$ coincides with the $\Sigma$-space $\Sigma(A, B)$ up to isomorphisms.
\end{example}

We may further study the concept of D-limits in its own right; for instance, similarly to finite limits, how can we characterize finite D-limits in terms of just a few instances of finite D-limits?
Nevertheless, for the lack of space, we leave it as future work.

\subsubsection{Cartesian Functors with Dependence}
It is now clear what would be morphisms between CCwDs:
\begin{definition}[CFwDs]
A \emph{\bfseries cartesian FwD (CFwD)} is an SCFwD $F : \mathcal{C} \to \mathcal{C'}$ between CCwDs $\mathcal{C}$ and $\mathcal{C'}$ such that:
\begin{itemize}

\item The D-object $F(1) \in \mathscr{D}_{\mathcal{C'}}(F(T))$ is unit in $\mathcal{C'}$ for each unit D-object $1 \in \mathscr{D}_{\mathcal{C}}(T)$;

\item The diagram $F(A)\{F(\pi_1)\} \stackrel{F(\varpi_1)}{\leftarrowtriangle} F(\Gamma . \Sigma(A, B)) \stackrel{F(\varpi_2)}{\rightarrowtriangle} F(B)\{ \iota^{F(\Gamma . A)}_{F(\Gamma) . F(A)} \}\{\langle F(\pi_1), F(\varpi_1) \rangle\}$ in $\mathcal{C'}$ is a $\Sigma$-space of $F(A) \in \mathscr{D}_{\mathcal{C'}}(F(\Gamma))$ and $F(B)\{ \iota^{F(\Gamma . A)}_{F(\Gamma) . F(A)} \} \in \mathscr{D}_{\mathcal{C'}}(F(\Gamma) . F(A))$ in $\mathcal{C'}$ for a $\Sigma$-space $A\{\pi_1\} \stackrel{\varpi_1}{\leftarrowtriangle} \Gamma . \Sigma(A, B) \stackrel{\varpi_2}{\rightarrowtriangle} B\{\langle \pi_1, \varpi_1 \rangle\}$ of any $A \in \mathscr{D}_{\mathcal{C}}(\Gamma)$ and $B \in \mathscr{D}_{\mathcal{C}}(\Gamma . A)$ in $\mathcal{C}$.

\end{itemize} 

Moreover, given that $\mathcal{C}$ and $\mathcal{C'}$ are both strict, written $\mathcal{C} = (\mathcal{C}, T, \_ . \_, \pi, 1, \Sigma, \varpi)$ and $\mathcal{C'} = (\mathcal{C'}, T', \_ .' \_, \pi', 1', \Sigma', \varpi')$, $F$ is \emph{\bfseries strict} iff it is a strict SCFwD, $F(1) = 1'$, and the diagram 
\begin{equation*}
F(A)\{F(\pi_1)\} \stackrel{F(\varpi_1)}{\leftarrowtriangle} F(\Gamma . \Sigma(A, B)) \stackrel{F(\varpi_2)}{\rightarrowtriangle} F(B)\{ \iota^{F(\Gamma . A)}_{F(\Gamma) . F(A)} \}\{\langle F(\pi_1), F(\varpi_1) \rangle\} 
\end{equation*}
in $\mathcal{C'}$ coincides with the $\Sigma$-space 
\begin{equation*}
F(A)\{\pi_1'\} \stackrel{\varpi_1'}{\leftarrowtriangle} F(\Gamma) .' \Sigma'(F(A), F(B)) \stackrel{\varpi_2'}{\rightarrowtriangle} F(B)\{\langle \pi_1', \varpi_1' \rangle\}
\end{equation*}
in $\mathcal{C'}$ for any $A \in \mathscr{D}_{\mathcal{C}}(\Gamma)$ and $B \in \mathscr{D}_{\mathcal{C}}(\Gamma . A)$ in $\mathcal{C}$.
\end{definition}

Given a CFwD $F : \mathcal{C} \to \mathcal{D}$, there is an isomorphism 
\begin{equation*}
(\mathit{id}_{F(\Gamma)}, \iota^{F(\Sigma(A, B))}_{\Sigma(F(A), F(B))}) : \Sigma(F(A), F(B)\{ \iota^{F(\Gamma . A)}_{F(\Gamma) . F(A)} \}) \stackrel{\sim}{\to} F(\Sigma(A, B)) 
\end{equation*}
between D-objects in $\mathcal{D}$ for any $\Sigma$-space $\Sigma(A, B) \in \mathscr{D}_{\mathcal{C}}(\Gamma)$ in $\mathcal{C}$.
It is easy to see that $F$ is strict iff $(\mathit{id}_{F(\Gamma)}, \iota^F_{A, B})$ is the identity for any $\Sigma$-space $\Sigma(A, B) \in \mathscr{D}_{\mathcal{C}}(\Gamma)$ in $\mathcal{C}$.

\begin{example}
A cartesian functor $F : \mathcal{C} \to \mathcal{D}$ seen as an SCFwD $\mathsf{D}(F) : \mathsf{D}(\mathcal{C}) \to \mathsf{D}(\mathcal{D})$ is a CFwD because unit D-objects and $\Sigma$-spaces are just finite products.
Let us call this kind of CFwDs \emph{\bfseries cartesian functors seen as CFwDs}.
Thus, CFwDs are generalized cartesian functors.
\end{example}

\begin{example}
Given a locally small CCwD $\mathcal{C}$ and an object $\Delta \in \mathcal{C}$, the SCFwD $\mathcal{C}(\Delta, \_) : \mathcal{C} \to \mathit{Sets}$ is cartesian.
Given a unit D-object $1 \in \mathscr{D}_{\mathcal{C}}(T)$, the D-set $\mathcal{C}(\Delta, 1)$ is unit in $\mathit{Sets}$ because it is a singleton D-set $\mathcal{C}(\Delta, 1) = (\mathcal{C}(\Delta, 1\{!_\Delta\}))_{!_\Delta \in \mathcal{C}(\Delta, T)}$, and the set $\mathcal{C}(\Delta, 1\{!_\Delta\})$ is a singleton set.
Moreover, given a $\Sigma$-space $A\{\pi_1\} \stackrel{\varpi_1}{\leftarrowtriangle} \Gamma . \Sigma(A, B) \stackrel{\varpi_2}{\rightarrowtriangle} B\{\langle \pi_1, \varpi_1 \rangle\}$ of any D-objects $A \in \mathscr{D}_{\mathcal{C}}(\Gamma)$ and $B \in \mathscr{D}_{\mathcal{C}}(\Gamma . A)$ in $\mathcal{C}$, it is not hard to see that the diagram 
\begin{equation*}
\mathcal{C}(\Delta, A)\{\mathcal{C}(\Delta, \pi_1)\} \stackrel{\mathcal{C}(\Delta, \varpi_1)}{\leftarrowtriangle} \mathcal{C}(\Delta, \Gamma . \Sigma(A, B)) \stackrel{\mathcal{C}(\Delta, \varpi_2)}{\rightarrowtriangle} \mathcal{C}(\Delta, B)\{\mathcal{C}(\Delta, \langle \pi_1, \varpi_1 \rangle)\}
\end{equation*}
is a $\Sigma$-space of D-sets $\mathcal{C}(\Delta, A) \in \mathscr{D}_{\mathit{Sets}}(\mathcal{C}(\Delta, \Gamma))$ and $\mathcal{C}(\Delta, B) \in \mathscr{D}_{\mathit{Sets}}(\mathcal{C}(\Delta, \Gamma . A))$ in $\mathit{Sets}$.
\end{example}

\subsubsection{The 2-Category of Cartesian Categories with Dependence}
It is now clear there is the following 2-category:
\begin{definition}[The 2-category $\mathbb{CC_D}$]
The 2-category $\mathbb{CC_D}$ is a sub-2-category of $\mathbb{SCC_D}$ whose 0-cells are small CCwDs and 1-cells are CFwDs.
\end{definition}

\begin{theorem}[Well-defined $\mathbb{CC_D}$]
The structure $\mathbb{CC_D}$ is a well-defined sub-2-category of $\mathbb{SCC_D}$.
\end{theorem}

\subsection{Cartesian Closed Categories and Functors with Dependence}
\label{CCCwDs}
This is the last section on the basic theory of CwDs, in which we introduce a \emph{closed structure} on SCCwDs.
On the one hand it gives a categorical semantics of strict $\mathsf{\Pi}$-types in MLTTs, and on the other hand it provides a categorical generalization of the 2-category $\mathbb{CCC}$ of small CCCs.

\subsubsection{Cartesian Closed Categories with Dependence}
We would like to equip CCwDs with a \emph{closed structure} to form a generalization of CCCs.
Then, as a generalization of exponentials, the following construction seems appropriate:
\begin{definition}[Pseudo-$\Pi$-spaces]
Let $\mathcal{C}$ be an SCCwD, $\Gamma \in \mathcal{C}$, $A \in \mathscr{D}_{\mathcal{C}}(\Gamma)$ and $B \in \mathscr{D}_{\mathcal{C}}(\Gamma . A)$. 
A \emph{\bfseries pseudo-dependent map (pseudo-$\Pi$-) space} from $A$ to $B$ in $\mathcal{C}$ is a D-object 
\begin{equation*}
\Pi(A, B) \in \mathscr{D}_{\mathcal{C}}(\Gamma)
\end{equation*} 
equipped with a D-morphism 
\begin{equation*}
\mathit{dev}_{\Pi(A, B)} : \Gamma . \Pi(A, B) . A\{\pi_1\} \rightarrowtriangle B \{ \pi_1^{+A} \}
\end{equation*}
in $\mathcal{C}$, called the \emph{\bfseries dependent (D-) evaluation} of $\Pi(A, B)$, where 
\begin{equation*}
\pi_1^{+A} \stackrel{\mathrm{df. }}{=} \langle \pi_1 \circ \pi_1, \pi_2 \rangle : \Gamma . \Pi(A, B) . A\{\pi_1\} \to \Gamma . A
\end{equation*}
such that for any D-morphism in $\mathcal{C}$ of the form $f : \Gamma . A \rightarrowtriangle B$ there exists a unique D-morphism 
\begin{equation*}
\Lambda_{\Pi(A, B)}(f) : \Gamma \rightarrowtriangle \Pi(A, B)
\end{equation*}
in $\mathcal{C}$, called the \emph{\bfseries dependent (D-) currying} of $f$, that satisfies the equation
\begin{equation*}
\mathit{dev}_{\Pi(A, B)} \{ (\overline{\Lambda_{\Pi(A, B)} (f)})^{+A\{\pi_1\}} \} = f 
\end{equation*}
where 
\begin{align*}
\overline{\Lambda_{\Pi(A, B)} (f)} &\stackrel{\mathrm{df. }}{=} \langle \mathit{id}_{\Gamma}, \Lambda_{\Pi(A, B)} (f) \rangle : \Gamma \rightarrow \Gamma . \Pi(A, B) \\
(\overline{\Lambda_{\Pi(A, B)} (f)})^{+A\{\pi_1\}} &\stackrel{\mathrm{df. }}{=} \langle \overline{\Lambda_{\Pi(A, B)} (f)} \circ \pi_1, \pi_2 \rangle : \Gamma . A \rightarrow \Gamma . \Pi(A, B) . A\{\pi_1\}.
\end{align*}
\end{definition}

Pseudo-$\Pi$-spaces, D-evaluations and D-currying are, as their names suggest, intended to be generalizations of exponentials, evaluations and currying, respectively.  
To make this point explicit, let us draw the following commutative diagram that depicts the universal property of a pseudo-$\Pi$-space $\Pi(A, B) \in \mathscr{D}_{\mathcal{C}}(\Gamma)$:
\begin{diagram}
\Pi(A, B) & & & \Gamma . \Pi(A, B) . A\{\pi_1\} & && \rDepTo^{\mathit{dev}_{\Pi(A, B)}} & && B\{ \pi_1^{+A} \} \\ \\
\uDotsDepTo_{\Lambda_{\Pi(A, B)}(f)} & & & \uDotsTo_{(\overline{\Lambda_{\Pi(A, B)}(f)})^{+A\{\pi_1\}}} && && & & \dImplies \\ \\
\Gamma & & & \Gamma . A &&& \rDepTo^{f} & && B
\end{diagram}
which can be given in a CCC $\mathcal{C} = (\mathcal{C}, T, \times, p, \Rightarrow, \mathit{ev})$ by:
\begin{diagram}
A \Rightarrow B & & &(\Gamma \times (A \Rightarrow B)) \times A & & \rTo^{\mathit{ev}_{A, B} \circ \langle p_2 \circ p_1, p_2 \rangle} & & & B \\ \\
\uDotsTo_{\lambda_{A, B}(f)} & & & \uDotsTo_{\langle \mathit{id}_\Gamma, (\lambda_{A, B}(f) \rangle \times \mathit{id}_A} & & & & & \dEquals \\ \\
\Gamma & & & \Gamma \times A && \rTo^{f} & & & B
\end{diagram}

Unfortunately, however, pseudo-$\Pi$-spaces are \emph{not} unique up to isomorphisms because the canonical morphisms between pseudo-$\Pi$-spaces from the same D-object $A \in \mathscr{D}_{\mathcal{C}}(\Gamma)$ to the same D-object $B \in \mathscr{D}_{\mathcal{C}}(\Gamma . A)$ (specifically, they are $\imath_{A, B}$ and $\jmath_{A, B}$ in the proof of Proposition~\ref{PropUniquePi} below) are not inverses to each other.
By a similar mechanism, they do \emph{not} give rise to functors. 
In other words, pseudo-$\Pi$-spaces are a categorically incomplete structure. 

To overcome this point, let us define:
\begin{definition}[$\Pi$-spaces]
A pseudo-$\Pi$-space $\Pi(A, B)$ from $A \in \mathscr{D}_{\mathcal{C}}(\Gamma)$ to $B \in \mathscr{D}_{\mathcal{C}}(\Gamma . A)$ in an SCCwD $\mathcal{C}$ is a \emph{\bfseries dependent map ($\Pi$-) space} from $A$ to $B$ in $\mathcal{C}$ iff for any D-object $C \in \mathscr{D}_{\mathcal{C}}(\Gamma)$ the pair $\Pi(A, B) \{ \pi_1 \} = (\Pi(A, B) \{ \pi_1 \}, \mathit{dev}_{\Pi(A, B) \{ \pi_1 \}})$ such that
\begin{align*}
\Pi(A, B) \{ \pi_1 \} &\in \mathscr{D}_{\mathcal{C}}(\Gamma . C) \\
\mathit{dev}_{\Pi(A, B) \{ \pi_1 \}} \stackrel{\mathrm{df. }}{=} \mathit{dev}_{\Pi(A, B)} \{ \pi_1^{+\Pi(A, B) +A\{\pi_1\}} \} &: \Gamma . C . \Pi(A, B)\{\pi_1\} . A\{\pi_1\}\{\pi_1\} \rightarrowtriangle B\{ \pi_1^{+A} \} \{ \pi_1^{+\Pi(A, B) +A\{\pi_1\}} \}
\end{align*}
is a pseudo-$\Pi$-space from $A\{\pi_1\} \in \mathscr{D}_{\mathcal{C}}(\Gamma . C)$ to $B\{\pi_1^{\star A}\} \in \mathscr{D}_{\mathcal{C}}(\Gamma . C . A \{\pi_1\})$ in $\mathcal{C}$ that satisfies
\begin{align}
\label{UP2Pi}
\mathit{dev}_{\Pi(A, B)} \{ \langle \pi_1, \Lambda_{\Pi(A, B)\{\pi_1\}}(h) \rangle^{+A\{\pi_1\}} \} &= h \\
\label{UP3Pi}
\Lambda_{\Pi(A, B)\{\pi\}} (h \{ \langle \pi_1, g \rangle^{+A\{\pi_1\}}\}) &= \Lambda_{\Pi(A, B)\{\pi_1\}}(h)\{ \langle \pi_1, g \rangle \} 
\end{align}
for any $D \in \mathscr{D}_{\mathcal{C}}(\Gamma)$, $h : \Gamma . C . A\{ \pi_1 \} \rightarrowtriangle B\{ \pi_1^{\star A} \}$ and $g : \Gamma . D \rightarrowtriangle C\{\pi_1\}$ in $\mathcal{C}$, where
\begin{align*}
\pi_1^{+\Pi(A, B)} &\stackrel{\mathrm{df. }}{=} \langle \pi_1 \circ \pi_1, \pi_2 \rangle : \Gamma . C . \Pi(A, B)\{\pi_1\} \rightarrow \Gamma . \Pi(A, B) \\
\pi_1^{+\Pi(A, B) +A\{\pi_1\}} &\stackrel{\mathrm{df. }}{=} \langle \pi_1^{+\Pi(A, B)} \circ \pi_1, \pi_2 \rangle : \Gamma . C . \Pi(A, B)\{\pi_1\} . A\{\pi_1\}\{\pi_1\} \rightarrow \Gamma . \Pi(A, B) . A\{\pi_1\} \\
\pi_1^{+A} &\stackrel{\mathrm{df. }}{=} \langle \pi_1 \circ \pi_1, \pi_2 \rangle : \Gamma . \Pi(A, B) . A\{\pi_1\} \rightarrow \Gamma . A \\
\pi_1^{\star A} &\stackrel{\mathrm{df. }}{=} \langle \pi_1 \circ \pi_1, \pi_2 \rangle : \Gamma . C . A\{\pi_1\} \rightarrow \Gamma . A \\
\langle \pi_1, \Lambda_{\Pi(A, B)\{\pi_1\}}(h) \rangle^{+A\{\pi_1\}} &\stackrel{\mathrm{df. }}{=} \langle \langle \pi_1, \Lambda_{\Pi(A, B)\{\pi_1\}}(h) \rangle \circ \pi_1, \pi_2 \rangle : \Gamma . C . A\{\pi_1\} \rightarrow \Gamma . \Pi(A, B) . A\{\pi_1\} \\
\langle \pi_1, g \rangle^{+A\{\pi_1\}} &\stackrel{\mathrm{df. }}{=} \langle \langle \pi_1, g \rangle \circ \pi_1, \pi_2 \rangle : \Gamma . D . A\{\pi_1\} \rightarrow \Gamma . C . A\{\pi_1\}.
\end{align*}
\end{definition}

The additional universal properties (\ref{UP2Pi}) and (\ref{UP3Pi}) of a $\Pi$-space $\Pi(A, B) \in \mathscr{D}_{\mathcal{C}}(\Gamma)$ may be depicted as the following commutative diagrams:
\begin{diagram}
&&& \Pi(A, B)\{ \pi_1 \} & & \Gamma . \Pi(A, B) . A\{\pi_1\} & && \rDepTo^{\mathit{dev}_{\Pi(A, B)}} & && B\{ \pi_1^{+A} \} \\ 
&&\ruDotsDepTo(3, 4)^{\Lambda (h \{ \langle \pi_1, g \rangle^{+}\})}&&&&&&&&& \\
&&& \uDotsDepTo_{\Lambda_{\Pi(A, B)\{\pi_1\}}(h)} & & \uDotsTo_{\langle \pi_1, \Lambda_{\Pi(A, B)\{\pi_1\}}(h) \rangle^{+A\{\pi_1\}}} && && & & \dImplies \\ 
&&&&&&&&&&& \\
\Gamma . D && \rTo^{\langle \pi_1, g \rangle} & \Gamma . C & & \Gamma . C . A\{\pi_1\} &&& \rDepTo^{h} & && B\{ \pi_1^{+A} \}
\end{diagram}
where $\Lambda (h \{ \langle \pi_1, g \rangle^{+}\}) \stackrel{\mathrm{df. }}{=} \Lambda_{\Pi(A, B)\{\pi\}} (h \{ \langle \pi_1, g \rangle^{+A\{\pi_1\}}\})$.

Proposition~\ref{PropUniquePi} below verifies the uniqueness of $\Pi$-spaces up to isomorphisms, where the additional axioms (\ref{UP2Pi}) and (\ref{UP3Pi}) play essential roles. 
Also, it is straightforward to see that $\Pi$-spaces (and pseudo-$\Pi$-spaces) in a cartesian category seen as a CCwD coincide with exponentials, where note that D-evaluations do not depend on objects, and thus they are essentially the same as evaluations.

We then define the promised generalization of CCCs as follows:
\begin{definition}[SCCCwDs]
A \emph{\bfseries semi-cartesian closed CwD (SCCCwD)} is an SCCwD $\mathcal{C}$ that has for any triple of $\Gamma \in \mathcal{C}$, $A \in \mathscr{D}_{\mathcal{C}}(\Gamma)$ and $B \in \mathscr{D}_{\mathcal{C}}(\Gamma . A)$ a $\Pi$-space from $A$ to $B$ in $\mathcal{C}$.
A \emph{\bfseries strict SCCCwD} is a strict SCCwD $\mathcal{C}$ equipped with a family $\Pi = (\Pi(A, B))_{\Gamma \in \mathcal{C}, A \in \mathscr{D}_{\mathcal{C}}(\Gamma), B \in \mathscr{D}_{\mathcal{C}}(\Gamma . A)}$ of $\Pi$-spaces $\Pi(A, B) \in \mathscr{D}_{\mathcal{C}}(\Gamma)$ in $\mathcal{C}$ that is \emph{\bfseries coherent} in the sense that it satisfies the equations
\begin{align}
\label{DcompPi1}
\Pi(A, B)\{ \phi \} &= \Pi(A\{\phi\}, B\{ \phi^{+A} \}) \\
\label{DcompPi2}
\mathit{dev}_{\Pi(A, B)} \{ \phi^{+\Pi(A, B) +A\{\pi_1\}} \} &= \mathit{dev}_{\Pi(A\{\phi\}, B\{\phi^{+A}\})} 
\end{align}
for any $\Delta, \Gamma \in \mathcal{C}$, $A \in \mathscr{D}_{\mathcal{C}}(\Gamma)$, $B \in \mathscr{D}_{\mathcal{C}}(\Gamma . A)$ and $\phi : \Delta \to \Gamma$ in $\mathcal{C}$, where 
\begin{align*}
\phi^{+A} &\stackrel{\mathrm{df. }}{=} \langle \phi \circ \pi_1, \pi_2 \rangle : \Delta . A\{\phi\} \to \Gamma . A \\
\phi^{+\Pi(A, B)} &\stackrel{\mathrm{df. }}{=} \langle \phi \circ \pi_1, \pi_2 \rangle : \Delta . \Pi(A, B)\{\phi\} \to \Gamma . \Pi(A, B) \\
\phi^{+\Pi(A, B) +A\{\pi_1\}} &\stackrel{\mathrm{df. }}{=} \langle \phi^{+\Pi(A, B)} \circ \pi_1, \pi_2 \rangle : \Delta . \Pi(A, B)\{\phi\} . A\{ \phi \} \{ \pi_1 \} \to \Gamma . \Pi(A, B) .A\{\pi_1\}.
\end{align*}
\end{definition}

\begin{definition}[CCCwDs]
A \emph{\bfseries cartesian closed CwD (CCCwD)} is a CCwD that is semi-cartesian closed.
A \emph{\bfseries strict CCCwD} is a strict CCwD that is strictly semi-cartesian closed.
\end{definition}

Note that the coherence axioms (\ref{DcompPi1}) and (\ref{DcompPi2}) on strict SCCCwDs are necessary to interpret strict $\mathsf{\Pi}$-types in MLTTs; see the proof of Theorem~\ref{ThmPi} below.

Now, let us prove some of the expected properties of $\Pi$-spaces, including their uniqueness up to isomorphisms (Proposition~\ref{PropUniquePi}), as a generalization of exponentials.
\begin{lemma}[Transpositions in pseudo-$\Pi$-spaces]
\label{LemTranspositionInPseudoPiSpaces}
Let $\mathcal{C}$ be an SCCwD, $\Gamma \in \mathcal{C}$, $A \in \mathscr{D}_{\mathcal{C}}(\Gamma)$ and $B \in \mathscr{D}_{\mathcal{C}}(\Gamma . A)$. 
If $\Pi(A, B) \in \mathscr{D}_{\mathcal{C}}(\Gamma)$ is a pseudo-$\Pi$-space from $A$ to $B$ in $\mathcal{C}$, then there is a bijection 
\begin{equation*}
\mathscr{D}_{\mathcal{C}}(\Gamma . A, B) \cong \mathscr{D}_{\mathcal{C}}(\Gamma, \Pi(A, B)).
\end{equation*}
\end{lemma}
\begin{proof}
There is the function $\mathscr{D}_{\mathcal{C}}(\Gamma . A, B) \to \mathscr{D}_{\mathcal{C}}(\Gamma, \Pi(A, B))$ that maps $f \mapsto \Lambda_{\Pi(A, B)}(f)$, and also there is the inverse $\mathscr{D}_{\mathcal{C}}(\Gamma, \Pi(A, B)) \to \mathscr{D}_{\mathcal{C}}(\Gamma . A, B)$ that maps $g \mapsto \mathit{dev}_{\Pi(A, B)} \{ (\overline{g})^{+A\{\pi_1\}} \}$, where it is easy to verify that these maps are inverses to each other and thus left to the reader.
\end{proof}

\begin{notation}
By Lemma~\ref{LemTranspositionInPseudoPiSpaces}, it is legitimate to write $\Lambda_{\Pi(A, B)}^{-1} : \mathscr{D}_{\mathcal{C}}(\Gamma, \Pi(A, B)) \to \mathscr{D}_{\mathcal{C}}(\Gamma . A, B)$ for the map 
\begin{equation*}
g \mapsto \mathit{dev}_{\Pi(A, B)} \{ (\overline{g})^{+A\{\pi_1\}} \}.
\end{equation*}
\end{notation}

\begin{lemma}[First $\Pi$-lemma]
\label{LemFirstPiLemma}
Let $\mathcal{C}$ be an SCCwD, $\Gamma \in \mathcal{C}$, $A \in \mathscr{D}_{\mathcal{C}}(\Gamma)$ and $B \in \mathscr{D}_{\mathcal{C}}(\Gamma. A)$.
If $\Pi(A, B) \in \mathscr{D}_{\mathcal{C}}(\Gamma)$ is a $\Pi$-space from $A$ to $B$ in $\mathcal{C}$, then we have:
\begin{equation*}
\Lambda_{\Pi(A, B)\{\pi_1\}}(\mathit{dev}_{\Pi(A, B)}) = \pi_2 : \Gamma . \Pi(A, B) \rightarrowtriangle \Pi(A, B)\{\pi_1\}.
\end{equation*}
\end{lemma}
\begin{proof}
By the universal property of the pseudo-$\Pi$-space $\Pi(A, B) \{ \pi_1 \} \in \mathscr{D}_{\mathcal{C}}(\Gamma . \Pi(A, B))$ in $\mathcal{C}$, it suffices to establish the equation
\begin{equation*}
\mathit{dev}_{\Pi(A, B)\{\pi_1\}}\{ (\overline{\pi_2})^{+A\{\pi_1\}\{\pi_1\}} \} = \mathit{dev}_{\Pi(A, B)}.
\end{equation*}
Then, observe that:
\begin{align*}
\mathit{dev}_{\Pi(A, B)\{\pi_1\}}\{ (\overline{\pi_2})^{+A\{\pi_1\}\{\pi_1\}} \} &= \mathit{dev}_{\Pi(A, B)}\{ \pi_1^{+\Pi(A, B) +A\{\pi_1\}} \} \{ (\overline{\pi_2})^{+A\{\pi_1\}\{\pi_1\}} \} \\
&= \mathit{dev}_{\Pi(A, B)}\{ \langle \langle \pi_1 \circ \pi_1, \pi_2 \rangle \circ \pi_1, \pi_2 \rangle \} \{ \langle \langle \pi_1, \pi_2\{\pi_1\} \rangle, \pi_2 \rangle \} \\
&= \mathit{dev}_{\Pi(A, B)}\{ \langle \langle \pi_1 \circ \pi_1, \pi_2 \rangle \circ \pi_1, \pi_2 \rangle \circ \langle \langle \pi_1, \pi_2\{\pi_1\} \rangle, \pi_2 \rangle \} \\
&= \mathit{dev}_{\Pi(A, B)}\{ \langle \langle \pi_1 \circ \pi_1, \pi_2\{\pi_1\} \rangle, \pi_2 \rangle \} \\
&= \mathit{dev}_{\Pi(A, B)}\{ \langle \langle \pi_1, \pi_2 \rangle \circ \pi_1, \pi_2 \rangle \} \\
&= \mathit{dev}_{\Pi(A, B)}\{ \langle \mathit{id}_{\Gamma . \Pi(A, B)} \circ \pi_1, \pi_2 \rangle \} \\
&= \mathit{dev}_{\Pi(A, B)}\{ \langle \pi_1, \pi_2 \rangle \} \\
&= \mathit{dev}_{\Pi(A, B)}\{ \mathit{id}_{\Gamma . \Pi(A, B) . A\{\pi_1\}} \} \\
&= \mathit{dev}_{\Pi(A, B)}
\end{align*}
which completes the proof.
\end{proof}

\begin{lemma}[Second $\Pi$-lemma]
\label{LemSecondPiLemma}
Let $\mathcal{C} = (\mathcal{C}, T, \_ . \_, \pi, \Pi, \mathit{dev})$ be a strict SCCCwD.
Given $\Delta, \Gamma \in \mathcal{C}$, $\phi : \Delta \to \Gamma$, $A \in \mathscr{D}_{\mathcal{C}}(\Gamma)$, $B \in \mathscr{D}_{\mathcal{C}}(\Gamma . A)$ and $f : \Gamma . A \rightarrowtriangle B$ in $\mathcal{C}$, we have:
\begin{equation*}
\Lambda_{\Pi(A, B)}(f)\{\phi\} = \Lambda_{\Pi(A\{\phi\}, B\{\phi^{+A}\})}(f\{\phi^{+A}\}) : \Gamma \rightarrowtriangle \Pi(A, B)\{\phi\}.
\end{equation*}
\end{lemma}
\begin{proof}
By the universal property of the coherent family $\Pi$ of the $\Pi$-spaces in $\mathcal{C}$, it suffices to show the equation
\begin{equation*}
\mathit{dev}_{\Pi(A\{\phi\}, B\{\phi^{+A}\})} \{ (\overline{\Lambda_{\Pi(A, B)}(f)\{ \phi \}})^{+A\{\phi\}\{\pi_1\}} \} = f \{ \phi^{+A} \}.
\end{equation*}
Then, we have:
\begin{align*}
&\mathit{dev}_{\Pi(A\{\phi\}, B\{\phi^{+A}\})} \{ (\overline{\Lambda_{\Pi(A, B)}(f) \{ \phi \}})^{+A\{\phi\}\{\pi_1\}} \} \\
= \ &\mathit{dev}_{\Pi(A, B)} \{ \phi^{+\Pi(A, B) +A\{\pi_1\}} \} \{ (\overline{\Lambda_{\Pi(A, B)}(f) \{ \phi \}})^{+A\{\phi\}\{\pi_1\}} \} \ \text{(because $\Pi$ is coherent)} \\
= \ &\mathit{dev}_{\Pi(A, B)} \{ \langle \langle \langle \phi \circ \pi_1 \rangle, \pi_2 \rangle \circ \pi_1, \pi_2 \rangle \circ \langle \overline{\Lambda_{\Pi(A, B)}(f) \{ \phi \}} \circ \pi_1, \pi_2 \rangle \} \\
= \ &\mathit{dev}_{\Pi(A, B)} \{ \langle \langle \phi \circ \pi_1, \Lambda_{\Pi(A, B)}(f) \{ \phi \circ \pi_1 \} \rangle, \pi_2 \rangle \} \\
= \ &\mathit{dev}_{\Pi(A, B)} \{ \langle \langle \mathit{id}_\Gamma, \Lambda_{\Pi(A, B)}(f) \rangle \circ \phi \circ \pi_1, \pi_2 \rangle \} \\
= \ &\mathit{dev}_{\Pi(A, B)} \{ \langle \langle \mathit{id}_\Gamma, \Lambda_{\Pi(A, B)}(f) \rangle \circ \pi_1, \pi_2 \rangle \circ \langle \phi \circ \pi_1, \pi_2 \rangle \} \\
= \ &\mathit{dev}_{\Pi(A, B)} \{ \langle \langle \mathit{id}_\Gamma, \Lambda_{\Pi(A, B)}(f) \rangle \circ \pi_1, \pi_2 \rangle \} \{ \langle \phi \circ \pi_1, \pi_2 \rangle \} \\
= \ &f \{ \phi^{+A} \}
\end{align*}
which completes the proof.
\end{proof}

\begin{proposition}[Uniqueness of $\Pi$-spaces]
\label{PropUniquePi}
In any SCCwD $\mathcal{C}$, every $\Pi$-space in $\mathcal{C}$ is unique up to isomorphisms.
More precisely, if $\Gamma \in \mathcal{C}$, $A \in \mathscr{D}_{\mathcal{C}}(\Gamma)$ and $B \in \mathscr{D}_{\mathcal{C}}(\Gamma . A)$, and $\Pi(A, B), \Pi'(A, B) \in \mathscr{D}_{\mathcal{C}}(\Gamma)$ are both $\Pi$-spaces from $A$ to $B$ in $\mathcal{C}$, then $\Pi(A, B) \cong \Pi'(A, B)$.
\end{proposition}
\begin{proof}
First, we have the following two D-morphisms in $\mathcal{C}$:
\begin{align*}
\imath_{A, B} &\stackrel{\mathrm{df. }}{=} \Lambda_{\Pi'(A, B)\{ \pi_1 \}}(\mathit{dev}_{\Pi(A, B)}) : \Gamma . \Pi(A, B) \rightarrowtriangle \Pi'(A, B) \{\pi_1\} \\
\jmath_{A, B} &\stackrel{\mathrm{df. }}{=} \Lambda_{\Pi(A, B)\{ \pi_1 \}}(\mathit{dev}_{\Pi'(A, B)}) : \Gamma . \Pi'(A, B) \rightarrowtriangle \Pi(A, B) \{\pi_1\}.
\end{align*}
It suffices to show that these D-morphisms are inverses to each other. 

Then, observe that:
\begin{align*}
\imath_{A, B} \{ \langle \pi_1, \jmath_{A, B} \rangle \} &= \Lambda_{\Pi'(A, B)\{ \pi_1 \}}(\mathit{dev}_{\Pi(A, B)}) \{ \langle \pi_1, \Lambda_{\Pi(A, B)\{ \pi_1 \}}(\mathit{dev}_{\Pi'(A, B)}) \rangle \} \\
&= \Lambda_{\Pi'(A, B)\{ \pi_1 \}}(\mathit{dev}_{\Pi(A, B)} \{ \langle \pi_1, \Lambda_{\Pi(A, B)\{ \pi_1 \}}(\mathit{dev}_{\Pi'(A, B)}) \rangle^{+A\{\pi_1\}} \}) \ \text{(by the axiom (\ref{UP3Pi}))} \\
&= \Lambda_{\Pi'(A, B)\{ \pi_1 \}}(\mathit{dev}_{\Pi(A, B)} \{ \langle \langle \pi_1 \circ \pi_1, \Lambda_{\Pi(A, B)\{ \pi_1 \}}(\mathit{dev}_{\Pi'(A, B)}) \{\pi_1\} \rangle, \pi_2 \rangle \}) \\
&= \Lambda_{\Pi'(A, B)\{ \pi_1 \}}(\mathit{dev}_{\Pi'(A, B)}) \ \text{(by the axiom (\ref{UP2Pi}))} \\
&= \pi_2 \ \text{(by Lemma~\ref{LemFirstPiLemma})}.
\end{align*}
By symmetry, we also have $\jmath_{A, B} \{ \langle \pi_1, \imath_{A, B} \rangle \} = \pi_2$. 

Moreover, the D-morphism
\begin{equation*}
\mathit{dev}_{\Pi(A, B)} \{ \langle \pi_1, \jmath_{A, B} \rangle^{+A\{\pi_1\}} \} : \Gamma . \Pi'(A, B) . A\{\pi_1\} \rightarrowtriangle B\{ \pi_1^{+A} \} 
\end{equation*}
where
\begin{equation*}
\langle \pi_1, \jmath_{A, B} \rangle^{+A\{\pi_1\}} \stackrel{\mathrm{df. }}{=} \langle \langle \pi_1, \jmath_{A, B} \rangle \circ \pi_1, \pi_2 \rangle : \Gamma . \Pi'(A, B) . A\{\pi_1\} \to \Gamma . \Pi(A, B) . A\{\pi_1\}
\end{equation*}
satisfies
\begin{align*}
\mathit{dev}_{\Pi(A, B)} \{ \langle \pi_1, \jmath_{A, B} \rangle^{+A\{\pi_1\}} \} &= \mathit{dev}_{\Pi(A, B)} \{ \langle \langle \pi_1, \Lambda_{\Pi(A, B)\{ \pi_1 \}}(\mathit{dev}_{\Pi'(A, B)}) \rangle \circ \pi_1, \pi_2 \rangle \} \\
&= \mathit{dev}_{\Pi'(A, B)} \ \text{(by the axiom (\ref{UP2Pi}))}.
\end{align*}
And again by symmetry, $\mathit{dev}_{\Pi'(A, B)} \{ \langle \pi_1, \imath_{A, B} \rangle^{+A\{\pi_1\}} \} = \mathit{dev}_{\Pi(A, B)}$, completing the proof. 
\end{proof}

At this point, we hope that the reader has been convinced that $\Pi$-spaces are a reasonable generalization of exponentials.
Then, let us define a generalization of CCCs:
\begin{definition}[SCCCwDs]
A \emph{\bfseries semi-cartesian closed CwD (SCCCwD)} is an SCCwD $\mathcal{C}$ that has $\Pi$-spaces, i.e., a $\Pi$-space $\Pi(A, B) \in \mathscr{D}_{\mathcal{C}}(\Gamma)$ for any $\Gamma \in \mathcal{C}$, $A \in \mathscr{D}_{\mathcal{C}}(\Gamma)$ and $B \in \mathscr{D}_{\mathcal{C}}(\Gamma . A)$.
An SCCCwD is \emph{\bfseries strict} if it is a strict SCCwD equipped with specified $\Pi$-spaces. 
\end{definition}

\begin{definition}[CCCwDs]
A \emph{\bfseries cartesian closed CwD (CCCwD)} is a CCwD that is semi-cartesian closed.
A CCCwD is \emph{\bfseries strict} if it is a strict CCwD equipped with specified $\Pi$-spaces. 
\end{definition}

\begin{example}
As observed above, a (resp. strict) cartesian category $\mathcal{C}$ seen as a (resp. strict) CCwD $\mathsf{D}(\mathcal{C})$ is closed iff so is $\mathcal{C}$.
Hence, (resp. strict) CCCwDs are a generalization of (resp. strict) CCCs.
Let us call this kind of CCCwDs \emph{\bfseries CCCs seen as CCCwDs}.
\end{example}

\begin{example}
The CCwD $\mathit{Sets}$ is cartesian closed. 
Given $X \in \mathit{Sets}$, $A \in \mathscr{D}_{\mathit{Sets}}(X)$ and $B \in \mathscr{D}_{\mathit{Sets}}(X . A)$, we define $\Pi(A, B)$ to be the D-set 
\begin{equation*}
(\{ g : A_x \to \textstyle \bigcup_{a \in A_x}B_{(x, a)} \mid \forall a_0 \in A_x . \ \! g(a_0) \in B_{(x, a_0)} \ \! \})_{x \in X}.
\end{equation*}

Given a D-function $f : X . A \rightarrowtriangle B$, its D-currying $\Lambda_{A, B}(f) : X \to \Pi(A, B)$ maps 
\begin{equation*}
(x \in X) \mapsto ((a \in A_x) \mapsto f(x, a) \in B_{(x, a)})
\end{equation*}
and the D-evaluation $\mathit{dev}_{A, B} : X . \Pi(A, B) . A\{\pi_1\} \rightarrowtriangle B\{\pi_1^{+A}\}$ maps 
\begin{equation*}
((x, g), a) \in X.\Pi(A, B).A\{\pi_1\} \mapsto g(a) \in B_{(x, a)}.
\end{equation*}

It is easy to see that these $\Pi$-spaces in $\mathit{Sets}$ satisfy the required axioms. 
\end{example}

\begin{example}
The CCwD $\mathcal{GPD}$ is cartesian closed.
Given a groupoid $\Gamma$, and D-groupoids $A : \Gamma \to \mathcal{GPD}$ and $B : \Gamma . A \to \mathcal{GPD}$, there is the D-groupoid $\Pi(A, B) : \Gamma \to \mathcal{GPD}$ defined by:
\begin{itemize}

\item First, as shown in \cite{hofmann1998groupoid}, the set $\mathscr{D}_{\mathcal{GPD}}(\Gamma, A)$ may be seen as the groupoid whose objects are D-groupoid morphisms $f : \Gamma \rightarrowtriangle A$, where note that $\overline{f} \stackrel{\mathrm{df. }}{=} \langle \mathit{id}_\Gamma, f \rangle : \Gamma \to \Gamma .A$ is a functor, and isomorphisms $f \to g$ are families $\eta = (\eta_\gamma)_{\gamma \in \Gamma}$ of isomorphisms $\eta_\gamma : f(\gamma) \to g(\gamma)$ in $A(\gamma)$ such that $\overline{\eta} = (\overline{\eta}_\gamma)_{\gamma \in \Gamma}$, where $\overline{\eta}_\gamma \stackrel{\mathrm{df. }}{=} (\mathit{id}_\gamma, \eta_\gamma) : \overline{f}(\gamma) \to \overline{g}(\gamma)$, forms an NT $\overline{f} \to \overline{g}$.

\item Then, again as shown in \cite{hofmann1998groupoid}, the D-groupoid $\Pi(A, B) : \Gamma \to \mathcal{GPD}$ maps each object $\gamma \in \Gamma$ to the groupoid 
\begin{equation*}
\mathscr{D}_{\mathcal{GPD}}(A(\gamma), B_\gamma)
\end{equation*}
and each isomorphism $\phi : \gamma \stackrel{\sim}{\to} \gamma'$ in $\Gamma$ to the functor 
\begin{equation*}
\Pi(A, B)(\phi) : \mathscr{D}_{\mathcal{GPD}}(A(\gamma), B_\gamma) \to \mathscr{D}_{\mathcal{GPD}}(A(\gamma'), B_{\gamma'})
\end{equation*}
that maps each D-groupoid homomorphism $f \in \mathscr{D}_{\mathcal{GPD}}(A(\gamma), B_\gamma)$ to another
\begin{align*}
\Pi(A, B)(\phi) (f) &\in \mathscr{D}_{\mathcal{GPD}}(A(\gamma'), B_{\gamma'}) \\
a' \in A(\gamma') &\mapsto B(\phi, \mathit{id}_{a'}) (f\{A(\phi^{-1})\}(a')) \in B_{\gamma'}(a') \\
\alpha' \in A(\gamma')(a'_1, a'_2) &\mapsto B(\phi, \mathit{id}_{a'_2}) (f\{A(\phi^{-1})\}(\alpha')) : B(\phi, \mathit{id}_{a'_2})(\Pi(A, B)(\phi) (f)(a_1')) \to \Pi(A, B)(\phi) (f)(a'_2)
\end{align*}
and each isomorphism $\eta : f \stackrel{\sim}{\to} g$ in $\mathscr{D}_{\mathcal{GPD}}(A(\gamma), B_\gamma)$ to the natural isomorphism 
\begin{equation*}
\Pi(A, B)(\phi)(\eta) : \Pi(A, B)(\phi) (f) \stackrel{\sim}{\to} \Pi(A, B)(\phi) (g)
\end{equation*}
whose component on each $a' \in A(\gamma')$ is given by:
\begin{equation*}
\Pi(A, B)(\phi)(\eta)_{a'} \stackrel{\mathrm{df. }}{=} B(\phi, \mathit{id}_{a'})(\eta_{A(\phi^{-1})(a')}).
\end{equation*}

\end{itemize}

The D-evaluation $\mathit{dev}_{A, B} : \Gamma . \Pi(A, B) . A\{\pi_1\} \rightarrowtriangle B\{\pi_1^{+A}\}$ in $\mathcal{GPD}$ is the D-groupoid homomorphism that maps 
\begin{align*}
((\gamma, f), a) \in \Gamma . \Pi(A, B) . A\{\pi_1\} &\mapsto f(a) \in B_\gamma(a) \\
((\phi, \eta), \alpha) \in \Gamma . \Pi(A, B) . A\{\pi_1\}(((\gamma, a), f), ((\gamma', a'), f')) &\mapsto \eta_{a'} \circ B(\phi, \mathit{id}_{a'})(f\{A(\phi^{-1})\}(\alpha)) : B(\phi, \alpha)(f(a)) \to f'(a')
\end{align*}
and the D-currying $\Lambda(h) : \Gamma \rightarrowtriangle \Pi(A, B)$ of a D-groupoid homomorphism $h : \Gamma . A \rightarrowtriangle B$ maps
\begin{align*}
(\gamma \in \Gamma) &\mapsto ((a \in A(\gamma)) \mapsto f(\gamma, a) \in B(\gamma, a)) \\
(\phi \in \Gamma(\gamma, \gamma')) &\mapsto ((\alpha \in A(\gamma')(A(\phi)(a), a')) \mapsto f(\phi, \alpha) \in B(\gamma', a')(B(\phi, \alpha)(f(\gamma, a)), f(\gamma', a'))).
\end{align*}

It is not hard to see that these $\Pi$-spaces in $\mathcal{GPD}$ satisfy the required axioms.
\end{example}

\begin{example}
The term model $\mathcal{T}(1, \Pi, \Sigma)$ is a CCCwD. $\Pi$-spaces are given by \textsc{$\Pi$-Form}, $\mathit{dev}_{A, B}$ by \textsc{$\Pi$-Elim} for terms $\mathsf{\Gamma, f : \Pi(A, B), x : A \vdash f(x) : B[f(x)]}$ and D-currying by \textsc{$\Pi$-Intro}.
\end{example}

Now, let us prove that a coherent family of $\Pi$-spaces interprets strict $\mathsf{\Pi}$-types in MLTTs:
\begin{theorem}[Coherent $\Pi$-spaces interpret $\mathsf{\Pi}$-types]
\label{ThmPi}
A strict SCCCwD induces a CwF that supports $\Pi$-types in the strict sense.
\end{theorem}
\begin{proof}
Let $\mathcal{C} = (\mathcal{C}, T, \_ . \_, \pi, \Pi, \mathit{dev})$ be a strict SCCCwD. 
We equip the corresponding CwF $\mathcal{C}$ with $\Pi$-types in the strict sense as follows:
\begin{itemize}

\item \textsc{($\Pi$-Form)} Given $\Gamma \in \mathcal{C}$, $A \in \mathscr{D}_{\mathcal{C}}(\Gamma)$ and $B \in \mathscr{D}_{\mathcal{C}}(\Sigma(\Gamma, A))$, we have:
\begin{equation*}
\Pi(A, B) \in \mathscr{D}_{\mathcal{C}}(\Gamma).
\end{equation*}
 
\item \textsc{($\Pi$-Intro)} Given $f \in \mathscr{D}_{\mathcal{C}}(\Gamma . A, B)$, we define:
\begin{equation*}
\lambda_{A, B}(f) \stackrel{\mathrm{df. }}{=} \Lambda_{\Pi(A, B)}(f) \in \mathscr{D}_{\mathcal{C}}(\Gamma, \Pi(A, B)).
\end{equation*}

\item \textsc{($\Pi$-Elim)} Recall that there is the inverse $\Lambda_{\Pi(A, B)}^{-1}$ of the D-currying map $\Lambda_{\Pi(A, B)}$ as shown in the proof of Lemma~\ref{LemTranspositionInPseudoPiSpaces}.
Given $g \in \mathscr{D}_{\mathcal{C}}(\Gamma, \Pi(A, B))$ and $a \in \mathscr{D}_{\mathcal{C}}(\Gamma, A)$, we define:
\begin{equation*}
\mathit{App}_{A, B} (g, a) \stackrel{\mathrm{df. }}{=}  \Lambda_{\Pi(A, B)}^{-1}(g) \{ \overline{a} \} \in \mathscr{D}_{\mathcal{C}}(\Gamma, B\{ \overline{a} \}) 
\end{equation*}
where $\overline{a} = \langle \mathit{id}_\Gamma, a \rangle : \Gamma \to \Sigma (\Gamma, A)$.

\item \textsc{($\Pi$-Comp)} It is easy to see that:
\begin{align*}
\mathit{App}_{A, B} (\lambda_{A, B} (f) , a) &= \mathit{App}_{A, B} (\Lambda_{\Pi(A, B)} (f) , a) \\
&= \Lambda_{A, B}^{-1}(\Lambda_{\Pi(A, B)}(f)) \{\overline{a} \} \\
&= f \{ \overline{a} \}.
\end{align*}

\item \textsc{($\Pi$-Subst)} By the axiom (\ref{DcompPi1}).

\item \textsc{($\lambda$-Subst)} By Lemma~\ref{LemSecondPiLemma}.

\item \textsc{(App-Subst)} It is easy to see that:
\begin{align*}
\mathit{App}_{A, B} (g, a) \{ \phi \} &= \Lambda_{\Pi(A, B)}^{-1} (g) \{ \overline{a} \} \{ \phi \} \\
&= \Lambda_{\Pi(A, B)}^{-1} (g) \{ \langle \mathit{id}_\Gamma, a \rangle \circ \phi \} \\
&= \Lambda_{\Pi(A, B)}^{-1} (g) \{ \langle \phi, a \{ \phi \} \rangle \} \\
&= \Lambda_{\Pi(A, B)}^{-1} (g) \{ \phi^{+A} \circ \overline{a \{ \phi \}} \} \\
&= \Lambda_{\Pi(A, B)}^{-1} (g) \{ \phi^{+A} \} \{ \overline{a \{ \phi \}} \} \\
&= \Lambda_{\Pi(A\{ \phi \}, B\{ \phi^{+A} \})}^{-1} (g \{ \phi \}) \{ \overline{a \{ \phi \}} \} \ \text{(by $\lambda$-Subst shown above)} \\
&= \mathit{App}_{A\{ \phi \}, B\{ \phi^{+A} \}} (g \{ \phi \}, a \{ \phi \}) 
\end{align*}
where $\overline{a \{ \phi \}} \stackrel{\mathrm{df. }}{=} \langle \mathit{id}_\Delta, a \{ \phi \} \rangle : \Delta \to \Sigma(\Delta, A \{ \phi \})$.

\item \textsc{($\Pi$-Uniq)} Given $k \in \mathit{Tm}(\Gamma . A, \Pi(A, B)\{\pi_1\})$, observe that:
\begin{align*}
&\lambda_{A, B}(\mathit{App}_{A\{\pi_1\}, B\{\pi_1^{+A}\}}(k, \pi_2))\{\pi_1\} \\
= \ &\Lambda_{\Pi(A, B)}(\mathit{App}_{A\{\pi_1\}, B\{\pi_1^{+A}\}}(k, \pi_2))\{\pi_1\} \\
= \ &\Lambda_{\Pi(A\{\pi_1\}, B\{\pi_1^{+A}\})} (\mathit{App}_{A\{\pi_1\}, B\{\pi_1^{+A}\}}(k, \pi_2)\{\langle \pi_1^{\Gamma . A} \circ \pi_1^{\Gamma . A . A\{\pi_1\}}, \pi_2^{\Gamma . A . A\{\pi_1\}} \rangle\}) \\
= \ &\Lambda_{\Pi(A\{\pi_1\}, B\{\pi_1^{+A}\})} (\mathit{App}_{A\{\pi_1\}, B\{\pi_1^{+A}\}}(k, \pi_2)\{\langle \pi_1^{\Gamma . A} \circ \pi_1^{\Gamma . A . A\{\pi_1\}}, \pi_2^{\Gamma . A}\{\pi_1^{\Gamma . A . A\{\pi_1\}}\} \rangle\}) \ \text{(by the equation~(\ref{EqStrictSCCwDs}))} \\
= \ &\Lambda_{\Pi(A\{\pi_1\}, B\{\pi_1^{+A}\})} (\mathit{App}_{A\{\pi_1\}, B\{\pi_1^{+A}\}}(k, \pi_2) \{ \langle \pi_1^{\Gamma . A}, \pi_2^{\Gamma . A} \rangle \circ \pi_1^{\Gamma . A . A\{\pi_1\}} \}) \ \text{(by the equation~(\ref{ConsNat}))} \\
= \ &\Lambda_{\Pi(A\{\pi_1\}, B\{\pi_1^{+A}\})} (\mathit{App}_{A\{\pi_1\}, B\{\pi_1^{+A}\}}(k, \pi_2) \{ \mathit{id}_{\Gamma . A} \circ \pi_1^{\Gamma . A . A\{\pi_1\}} \}) \ \text{(by the equation~(\ref{ConsId}))} \\
= \ &\Lambda_{\Pi(A\{\pi_1\}, B\{\pi_1^{+A}\})} (\mathit{App}_{A\{\pi_1\}, B\{\pi_1^{+A}\}}(k, \pi_2)\{\pi_1^{\Gamma . A . A\{\pi_1\}}\}) \\
= \ &\Lambda_{\Pi(A\{\pi_1\}, B\{\pi_1^{+A}\})} (\Lambda^{-1}_{\Pi(A\{\pi_1\}, B\{\pi_1^{+A}\})}(k)\{\langle \mathit{id}_{\Gamma . A}, \pi_2^{\Gamma . A} \rangle\}\{\pi_1^{\Gamma . A . A\{\pi_1\}}\}) \\
= \ &\Lambda_{\Pi(A\{\pi_1\}, B\{\pi_1^{+A}\})} (\Lambda^{-1}_{\Pi(A\{\pi_1\}, B\{\pi_1^{+A}\})}(k)\{\langle \pi_1^{\Gamma . A . A\{\pi_1\}}, \pi_2^{\Gamma . A}\{\pi_1^{\Gamma . A . A\{\pi_1\}}\} \rangle\}) \\
= \ &\Lambda_{\Pi(A\{\pi_1\}, B\{\pi_1^{+A}\})} (\Lambda^{-1}_{\Pi(A\{\pi_1\}, B\{\pi_1^{+A}\})}(k)\{\langle \pi_1^{\Gamma . A . A\{\pi_1\}}, \pi_2^{\Gamma . A . A\{\pi_1\}} \rangle\}) \\
= \ &\Lambda_{\Pi(A\{\pi_1\}, B\{\pi_1^{+A}\})} (\Lambda^{-1}_{\Pi(A\{\pi_1\}, B\{\pi_1^{+A}\})}(k)\{\mathit{id}_{\Gamma . A . A\{\pi_1\}}\}) \\
= \ &\Lambda_{\Pi(A\{\pi_1\}, B\{\pi_1^{+A}\})} (\Lambda^{-1}_{\Pi(A\{\pi_1\}, B\{\pi_1^{+A}\})}(k)) \\
= \ &k
\end{align*}

\end{itemize}
which completes the proof.
\end{proof}

Hence, Theorems~\ref{ThmSoundnessOfCwFs}, \ref{ThmUnit}, \ref{ThmSigma} and \ref{ThmPi} implies:
\begin{corollary}[CCCwDs and CwFs]
A strict CCCwD induces a CwF that supports $1$-, $\Sigma$- and $\Pi$-types in the strict sense, and thus it gives a sound interpretation of $\mathsf{MLTT(1, \Pi, \Sigma)}$.
\end{corollary}

Also, since the term model $\mathcal{T}(1, \Pi, \Sigma)$ forms a strict CCCwD, we have:
\begin{theorem}[Completeness theorem]
\label{ThmCompletenessOfCwDs}
The interpretation of $\mathsf{MLTT(1, \Pi, \Sigma)}$ in CCCwDs is complete in the sense of Theorem~\ref{ThmCompletenessOfCwFs}.
\end{theorem}

Since $\Sigma$- and $\Pi$-spaces in CCCwDs are respectively generalizations of binary products and of exponentials in CCCs, respectively, it is expected that the former is left adjoint to the latter.
Let us see that it is in fact the case (up to the first projections) as follows.

First, let us consider $\Sigma$-spaces. 
Clearly, we have the correspondence 
\begin{diagram}[loose,height=.8em,width=0pt,l>=3em,midshaft]
& \ \ \ \Gamma . \Sigma(B, C) &        \rDepTo        & D\{\pi_1^{\Gamma . \Sigma(B, C)}\} \ \ \ \ & \\
& \rdLine &&& \\
&\Gamma . B . C & \rDepTo & D\{\pi_1^{\Gamma . B}\}\{\pi_1^{\Gamma . B . C}\} \ \ \ &
\end{diagram}
where $\Gamma \in \mathcal{C}$, $B, D \in \mathscr{D}_{\mathcal{C}}(\Gamma)$ and $C \in \mathscr{D}_{\mathcal{C}}(\Gamma . B)$, which induces an adjunction
\begin{diagram}
\mathcal{C}_{\Gamma . B} && \pile{\rTo^{\Sigma(B, \_)} \\ \perp \\ \lTo_{\_\{\pi_1^{\Gamma . B}\}}} && \mathcal{C}_{\Gamma}
\end{diagram}
where the functor $\_\{\pi_1^{\Gamma . B}\} : \mathcal{C}_\Gamma \to \mathcal{C}_{\Gamma . B}$ maps
\begin{align*}
A \in \mathscr{D}_{\mathcal{C}}(\Gamma) &\mapsto A\{\pi_1^{\Gamma . B}\} \in \mathscr{D}_{\mathcal{C}}(\Gamma . B) \\
g \in \mathscr{D}_{\mathcal{C}}(\Gamma . A, A'\{\pi_1\}) &\mapsto g \{ \langle \pi_1 \circ \pi_1, \pi_2 \rangle \} \in \mathscr{D}_{\mathcal{C}}(\Gamma . B . A\{\pi_1\}, A'\{\pi_1\}\{\pi_1\})
\end{align*}
and the functor $\Sigma(B, \_) : \mathcal{C}_{\Gamma . B} \to \mathcal{C}_\Gamma$ maps
\begin{align*}
C \in \mathscr{D}_{\mathcal{C}}(\Gamma . B) &\mapsto \Sigma(B, C) \in \mathscr{D}_{\mathcal{C}}(\Gamma) \\
h \in \mathscr{D}_{\mathcal{C}}(\Gamma . B . C, C'\{\pi_1\}) &\mapsto \Lbag \varpi_1, h \{ \mathit{Pair}_{B, C}^{-1} \} \Rbag \in \mathscr{D}_{\mathcal{C}}(\Gamma . \Sigma(B, C), \Sigma(B, C')\{\pi_1\})
\end{align*}
whose functorialities are both easy to verify. 

Next, let us think of $\Pi$-spaces.
Note that Lemma~\ref{LemTranspositionInPseudoPiSpaces} gives the bijective correspondence
\begin{diagram}[loose,height=.8em,width=0pt,l>=3em,midshaft]
& \Gamma . B & \rDepTo & C & \\
& \rdLine &&& \\
& \Gamma &        \rDepTo        & \Pi(B, C) \ \ \ \ \ &
\end{diagram}
for any $\Gamma \in \mathcal{C}$, $B \in \mathcal{C}_\Gamma$ and $C \in \mathcal{C}_{\Gamma . B}$, but it is not an adjunction.
Then, following the term model construction of DTTs in \cite{jacobs1999categorical}, we slightly modify this correspondence into
\begin{diagram}[loose,height=.8em,width=0pt,l>=3em,midshaft]
& \ \ \ \  \Gamma . B . A\{\pi_1^{\Gamma . B} \} & \rDepTo & C\{\pi_1^{\Gamma . B . A\{\pi_1^{\Gamma . B} \}}\} \ \ \ \ & \\
& \rdLine &&& \\
& \Gamma . A &        \rDepTo        & \Pi(B, C)\{ \pi_1^{\Gamma . A} \} \ \ \ \ \ & 
\end{diagram}
for each $A \in \mathcal{C}_\Gamma$, which gives an adjunction
\begin{diagram}
\mathcal{C}_{\Gamma} && \pile{\rTo^{\_\{\pi_1^{\Gamma . B}\}} \\ \perp \\ \lTo_{\Pi(B, \_)}} && \mathcal{C}_{\Gamma . B}
\end{diagram}
where $\Pi(B, \_) : \mathcal{C}_{\Gamma . B} \to \mathcal{C}_\Gamma$ maps
\begin{align*}
C \in \mathscr{D}_{\mathcal{C}}(\Gamma . B) &\mapsto \Pi(B, C) \in \mathscr{D}_{\mathcal{C}}(\Gamma) \\
f \in \mathscr{D}_{\mathcal{C}}(\Gamma . B . C, C'\{\pi_1\}) &\mapsto \Lambda_{\Pi(B\{\pi_1\}, C'\{\langle \pi_1 \circ \pi_1, \pi_2 \rangle\})} (f \{ \langle \langle \pi_1 \circ \pi_1, \pi_2 \rangle, \mathit{dev}_{\Pi(B, C)} \rangle \}) \in \mathscr{D}_{\mathcal{C}}(\Gamma . \Pi(B, C), \Pi(B, C')\{\pi_1\}).
\end{align*}
whose functoriality (specifically the preservation of composition) is by the axioms (\ref{UP2Pi}) and (\ref{UP3Pi}).

Therefore, we may compose the two adjunctions to obtain the adjunction: 
\begin{diagram}
\mathcal{C}_{\Gamma . B} && \pile{\rTo^{\Sigma(B, \_)\{\pi_1^{\Gamma . B}\}} \\ \perp \\ \lTo_{\Pi(B, \_\{\pi_1^{\Gamma . B}\})}} && \mathcal{C}_{\Gamma . B}
\end{diagram}
which particularly implies the correspondence 
\begin{diagram}[loose,height=.8em,width=0pt,l>=3em,midshaft]
& \ \ \ \Gamma . B . \Sigma(B, C)\{\pi_1\} &        \rDepTo        & D\{ \pi_1 \} & \\
& \rdLine &&& \\
&\ \ \ \Gamma . B . B\{\pi_1\} . C\{\pi_1\} & \rDepTo & D\{\pi_1\}\{\pi_1\} & \\
& \rdLine &&& \\
& \Gamma . B . C &        \rDepTo        & \Pi(B, D\{ \pi_1 \}) \{ \pi_1 \}\ \ \ \ &
\end{diagram}
between $\Sigma$ and $\Pi$ up to the first projections.

\subsubsection{Cartesian Closed Functors with Dependence}
Again, following the general recipe that morphisms are structure-preserving, let us define:
\begin{definition}[CCFwDs]
A \emph{\bfseries cartesian closed FwD (CCFwD)} is a CFwD $F : \mathcal{C} \to \mathcal{D}$ between CCCwDs $\mathcal{C}$ and $\mathcal{D}$ such that the diagram 
\begin{equation*}
F(\Gamma . \Pi(A, B) . A\{\pi_1\}) \stackrel{F(\mathit{dev}_{\Pi(A, B)})}{\rightarrowtriangle} F(B) \{ F(\pi_1^{+A}) \}
\end{equation*}
forms a $\Pi$-space from $F(A) \in \mathscr{D}_{\mathcal{D}}(F(\Gamma))$ to $F(B)\{\iota^{F(\Gamma . A)}_{F(\Gamma) . F(A)}\} \in \mathscr{D}_{\mathcal{D}}(F(\Gamma) . F(A))$ in $\mathcal{D}$ for any $\Gamma \in \mathcal{C}$, $A \in \mathscr{D}_{\mathcal{C}}(\Gamma)$ and $B \in \mathscr{D}_{\mathcal{C}}(\Gamma . A)$ in $\mathcal{C}$.
If $\mathcal{C}$ and $\mathcal{D}$ are both strict, $F$ is \emph{\bfseries strict} iff it is a strict CFwD, and the diagram
\begin{equation*}
F(\Gamma . \Pi(A, B) . A\{\pi_1\}) \stackrel{F(\mathit{dev}_{\Pi(A, B)})}{\rightarrowtriangle} F(B) \{ F(\pi_1^{+A}) \}
\end{equation*}
coincides with the $\Pi$-space 
\begin{equation*}
F(\Gamma) . \Pi(F(A), F(B)\{\iota^{F(\Gamma . A)}_{F(\Gamma) . F(A)}\}) . F(A)\{\pi_1\} \stackrel{\mathit{dev}_{\Pi(F(A), F(B)\{\iota^{F(\Gamma . A)}_{F(\Gamma) . F(A)}\})}}{\rightarrowtriangle} F(B)\{\iota^{F(\Gamma . A)}_{F(\Gamma) . F(A)}\}\{\pi_1^{+F(A)}\}
\end{equation*}
for any $\Gamma \in \mathcal{C}$, $A \in \mathscr{D}_{\mathcal{C}}(\Gamma)$ and $B \in \mathscr{D}_{\mathcal{C}}(\Gamma . A)$ in $\mathcal{C}$.
\end{definition}

\begin{example}
A (resp. strict) cartesian closed functor $F : \mathcal{C} \to \mathcal{D}$ seen as a CFwD $\mathsf{D}(F) : \mathsf{D}(\mathcal{C}) \to \mathsf{D}(\mathcal{D})$ is clearly (resp. strictly) closed.
\end{example}

\begin{example}
Given a locally small CCCwD $\mathcal{C}$ and an object $\Delta \in \mathcal{C}$, the CFwD $\mathcal{C}(\Delta, \_) : \mathcal{C} \to \mathit{Sets}$ is closed.
In fact, given a $\Pi$-space $\Gamma . \Pi(A, B) . A\{\pi_1\} \stackrel{\mathit{dev}_{\Pi(A, B)}}{\rightarrowtriangle} B \{ \pi_1^{+A} \}$ in $\mathcal{C}$, the diagram
\begin{equation*}
\mathcal{C}(\Delta, \Gamma . \Pi(A, B) . A\{\pi_1\}) \stackrel{\mathcal{C}(\Delta, \mathit{dev}_{\Pi(A, B)})}{\rightarrowtriangle} \mathcal{C}(\Delta, B \{ \pi_1^{+A} \})
\end{equation*}
in $\mathit{Sets}$ forms a $\Pi$-space from the D-set $\mathcal{C}(\Delta, A) \in \mathscr{D}_{\mathit{Sets}}(\mathcal{C}(\Delta, \Gamma))$ to another $\mathcal{C}(\Delta, B) \in \mathscr{D}_{\mathit{Sets}}(\mathcal{C}(\Delta, \Gamma.A))$.
\end{example}

\subsubsection{The 2-Category of Cartesian Closed Categories with Dependence}
Finally, let us summarize what we have established in this section as the following 2-category:
\begin{definition}[The 2-category $\mathbb{CCC_D}$]
The 2-category $\mathbb{CCC_D}$ is the sub-2-category of $\mathbb{CC_D}$ whose 0-cells are small CCCwDs and 1-cells are CCFwDs.
\end{definition}

\begin{theorem}[Well-defined $\mathbb{CCC_D}$]
$\mathbb{CCC_D}$ is a well-defined sub-2-category of $\mathbb{CC_D}$.
\end{theorem}
\begin{proof}
It suffices to show that the composition $G \circ F : \mathcal{C} \to \mathcal{E}$ of any CCFwDs $F : \mathcal{C} \to \mathcal{D}$ and $G : \mathcal{D} \to \mathcal{E}$ is again a CCFwD.
But then, it is immediate because the equation
\begin{equation*}
G(\iota^{F(\Gamma . A)}_{F(\Gamma) . F(A)}) \circ \iota^{G(F(\Gamma) . F(A))}_{G(F(\Gamma)) . G(F(A))} = \iota^{G \circ F(\Gamma . A)}_{G \circ F(\Gamma) . G \circ F(A)}
\end{equation*}
holds for any $\Gamma \in \mathcal{C}$ and $A \in \mathscr{D}_{\mathcal{C}}(\Gamma)$ in $\mathcal{C}$ by the universal property of semi-$\Sigma$-spaces.
\end{proof}

We have established a generalization of CCCs, whose strict version gives a categorical yet direct semantics of MLTTs.
In the rest of the paper, we shall make these relations with CCCs and with MLTTs more precise by establishing certain 2-equivalences.

\section{Equivalence between CtxCCCs and ConCtxCCCwDs}
\label{Equivalence}
This section is devoted to establish a 2-equivalence between the 2-category of \emph{contextual} CCCs and the 2-category of constant \emph{contextual} CCCwDs, where contextuality refers to strictness plus \emph{reachability} \cite{pitts2001categorical} of any object from a terminal object via binary products or semi-$\Sigma$-spaces.

The rest of the section proceeds roughly as follows.
We first define contextual CCCs and contextual CCCwDs in Section~\ref{ContextualCCCs} and Section~\ref{ContextualCCCwDs}, respectively.
And then, we establish the promised 2-equivalence (Theorem~\ref{ThmDSTheorem}) in Section~\ref{EqBetweenCtxCCCsAndConCtxCCCwDs}.

\subsection{Contextual Cartesian Closed Categories}
\label{ContextualCCCs}
Let us begin with introducing \emph{contextual} CCCs, which are similar to \emph{contextual categories} given by Cartmell \cite{cartmell1986generalized}:
\begin{definition}[CtxCCCs]
A \emph{\bfseries contextual CCC (CtxCCC)} is a strict CCC, i.e., a CCC $\mathcal{C}$ equipped with specified finite products $(T, \times, p)$ and exponentials $(\Rightarrow, \mathit{ev})$, together with a class $\mathscr{A}_{\mathcal{C}} \subseteq \mathsf{ob}(\mathcal{C})$ of objects in $\mathcal{C}$, whose elements are called \emph{\bfseries atomic objects} in $\mathcal{C}$, such that $T \in \mathscr{A}_{\mathcal{C}}$, and each object $\Delta \in \mathcal{C}$ has a unique finite sequence $\Delta_1, \Delta_2, \dots, \Delta_{\#(\Delta)} \in \mathscr{A}_{\mathcal{C}}^\ast$ of atomic objects in $\mathcal{C}$ that satisfies 
\begin{equation*}
\Delta = T \times \Delta_1 \times \Delta_2 \times \dots \times \Delta_{\#(\Delta)}
\end{equation*}
where $\times$ is left associative.
We call the natural number $\#(\Delta)$ the \emph{\bfseries length} of $\Delta$.
\end{definition}

Note that in STTs each context is of the form $\mathsf{\diamondsuit, x_1 : A_1, x_2 : A_2, \dots, x_k : A_k}$.
Atomic objects in contextual CCCs are a categorical abstraction of such atomic constituents $\mathsf{x_i : A_i}$ of contexts.

\begin{definition}[CtxCCFs]
A \emph{\bfseries contextual cartesian closed functor (CtxCCF)} is a strict cartesian closed functor $F : \mathcal{C} \to \mathcal{D}$ between CtxCCCs $\mathcal{C}$ and $\mathcal{D}$ that preserves atomic objects, i.e., 
\begin{equation*}
\Gamma \in \mathscr{A}_{\mathcal{C}} \Rightarrow F(\Gamma) \in \mathscr{A}_{\mathcal{C}}. 
\end{equation*}
\end{definition}

\begin{definition}[The 2-category $\mathsf{Ctx}\mathbb{CCC}$]
The 2-category $\mathsf{Ctx}\mathbb{CCC}$ is the sub-2-category of the 2-category $\mathbb{CCC}$ of small CCCs whose 0-cells are small CtxCCCs, 1-cells are CtxCCFs, and 2-cells are NTs.
\end{definition}

Clearly, $\mathsf{Ctx}\mathbb{CCC}$ is a well-defined 2-category.

\subsection{Contextual Cartesian Closed Categories with Dependence}
\label{ContextualCCCwDs}
Next, let us define the contextual analogue of SCCwDs:
\begin{definition}[CtxSCCwDs]
A \emph{\bfseries contextual SCCwD (CtxSCCwD)} is a strict SCCwD $\mathcal{C} = (\mathcal{C}, T, \_ . \_, \pi)$ such that: 
\begin{itemize}

\item Each constant D-object $A \in \mathscr{D}_{\mathcal{C}}(\Gamma)$ satisfies 
\begin{equation*}
A = \underline{A}\{ !_\Gamma \}
\end{equation*}
for a unique D-object $\underline{A} \in \mathscr{D}_{\mathcal{C}}(T)$, called the \emph{\bfseries core} of $A$; and 

\item Each object $\Delta \in \mathcal{C}$ has a unique finite sequence $D_1 \in \mathscr{D}_{\mathcal{C}}(T), D_2 \in \mathscr{D}_{\mathcal{C}}(T . D_1), \dots, D_{\sharp(\Delta)} \in \mathscr{D}_{\mathcal{C}}(T . D_1 . D_2 \dots D_{\sharp(\Delta)-1})$ of D-objects that satisfies 
\begin{equation*}
\Delta = T . D_1 . D_2 \dots D_{\#(\Delta)}
\end{equation*}
where we call the natural number $\#(\Delta)$ the \emph{\bfseries length} of $\Delta$.
\end{itemize}
\end{definition}

\begin{lemma}[CtxSCCwD-lemma]
\label{LemCtxSCCwDLemma}
Let $F : \mathcal{C} \to \mathcal{C'}$ be a strict SCFwD between CtxSCCwDs $\mathcal{C}$ and $\mathcal{C'}$. Then, we have: 
\begin{enumerate}

\item $A\{ \phi \} = \underline{A}\{!_\Delta\}$; 

\item $F(!_\Gamma) = \ !_{F(\Gamma)}$; and

\item $F(\underline{A}) = \underline{F(A)}$

\end{enumerate}
for any objects $\Delta, \Gamma \in \mathcal{C}$, morphism $\phi : \Delta \to \Gamma$ in $\mathcal{C}$ and constant D-object $A \in \mathscr{D}_{\mathcal{C}}(\Gamma)$.
\end{lemma}
\begin{proof}
The first and the second equations are immediate, and the third equation follows from the second. 
\end{proof}

\begin{definition}[CtxCCCwDs]
A \emph{\bfseries contextual CCCwD (CtxCCCwD)} is a strict CCCwD such that the underlying strict SCCwD is contextual. 
\end{definition}

\begin{definition}[The 2-category $\mathsf{Ctx}\mathbb{CCC_D}$]
The 2-category $\mathsf{Ctx}\mathbb{CCC_D}$ is the sub-2-category of $\mathbb{CCC_D}$ whose 0-cells are small CtxCCCwDs, 1-cells are strict CCFwDs, and 2-cells are NTwDs.
\end{definition}

\begin{definition}[The 2-category $\mathsf{ConCtx}\mathbb{CCC_D}$]
The 2-category $\mathsf{ConCtx}\mathbb{CCC_D}$ is the full sub-2-category of $\mathsf{Ctx}\mathbb{CCC_D}$ whose 0-cells are all constant.
\end{definition}

\subsection{Equivalence between CtxCCCs and ConCtxCCCwDs}
\label{EqBetweenCtxCCCsAndConCtxCCCwDs}
Finally, we define 2-functors $\mathsf{D} : \mathsf{Ctx}\mathbb{CCC} \rightleftarrows \mathsf{ConCtx}\mathbb{CCC_D} : \mathsf{S}$, and show that they constitute a 2-equivalence between the 2-categories. 

Let us begin with defining the functor $\mathsf{D} : \mathsf{Ctx}\mathbb{CCC} \to \mathsf{ConCtx}\mathbb{CCC_D}$, which just summarizes the operation $\mathsf{D}$ described so far: 
\begin{definition}[The 2-functor $\mathsf{D}$]
The 2-functor $\mathsf{D} : \mathsf{Ctx}\mathbb{CCC} \to \mathsf{ConCtx}\mathbb{CCC_D}$ maps: 
\begin{itemize}

\item Each CtxCCC $\mathcal{C} = (\mathcal{C}, \mathscr{A}_{\mathcal{C}}, T, \times, p, \Rightarrow, \mathit{ev})$ to the constant CtxCCCwD $\mathsf{D}(\mathcal{C})$ such that:
\begin{align*}
\mathscr{D}_{\mathsf{D}(\mathcal{C})}(\Gamma) &\stackrel{\mathrm{df. }}{=} \mathscr{A}_{\mathcal{C}} \\
\mathscr{D}_{\mathsf{D}(\mathcal{C})}(\Gamma, \Theta) &\stackrel{\mathrm{df. }}{=} \mathcal{C}(\Gamma, \Theta) \\
\Theta \{ \phi \}_{\mathsf{D}(\mathcal{C})} &\stackrel{\mathrm{df. }}{=} \Theta \\
\psi \{ \phi \}_{\mathsf{D}(\mathcal{C})} &\stackrel{\mathrm{df. }}{=} \psi \circ \phi \\
\Delta \stackrel{\pi_1}{\leftarrow}\Delta . \Gamma \stackrel{\pi_2}{\rightarrowtriangle} \Gamma \{ \pi_1 \} &\stackrel{\mathrm{df. }}{=} \Delta \stackrel{p_1}{\leftarrow} \Delta \times \Gamma \stackrel{p_2}{\rightarrow} \Gamma \\
1 &\stackrel{\mathrm{df. }}{=} T \\
\Theta\{\pi_1\} \stackrel{\varpi_1}{\leftarrowtriangle} \Gamma . \Sigma(\Theta . \Xi) \stackrel{\varpi_2}{\rightarrowtriangle} \Xi \{ \langle \pi_1, \varpi_1 \rangle \} &\stackrel{\mathrm{df. }}{=} \Theta \stackrel{p_1 \circ p_2}{\leftarrow} \Gamma \times (\Theta \times \Xi) \stackrel{p_2 \circ p_2}{\rightarrow} \Xi \\
\Gamma . \Pi(\Theta, \Xi) \Theta \{ \pi_1 \} \stackrel{\mathit{dev_{\Pi(\Theta, \Xi)}}}{\rightarrowtriangle} \Xi\{\pi_1^{+\Theta}\} &\stackrel{\mathrm{df. }}{=} \Gamma \times (\Theta \Rightarrow \Xi) \times \Theta \stackrel{\langle p_2 \circ p_1, p_2 \rangle}{\rightarrow} (\Theta \Rightarrow \Xi) \times \Theta \stackrel{\mathit{ev}_{\Theta, \Xi}}{\rightarrow} \Xi
\end{align*}
for all $\Delta, \Gamma \in \mathsf{D}(\mathcal{C})$, $\Theta \in \mathscr{D}_{\mathsf{D}(\mathcal{C})}(\Gamma)$, $\Xi \in \mathscr{D}_{\mathsf{D}(\mathcal{C})}(\Gamma . \Theta)$, $\psi \in  \mathscr{D}_{\mathsf{D}(\mathcal{C})}(\Gamma, \Theta)$ and $\phi : \Delta \to \Gamma$ in $\mathsf{D}(\mathcal{C})$; 

\item Each CtxCCF $F : \mathcal{C} \to \mathcal{D}$ to the strict CCFwD $\mathsf{D}(F) : \mathsf{D}(\mathcal{C}) \to \mathsf{D}(\mathcal{D})$ given by:
\begin{align*}
\mathsf{D}(F)_0 &\stackrel{\mathrm{df. }}{=} \mathsf{D}(F)_2 \stackrel{\mathrm{df. }}{=} F_0 \\
\mathsf{D}(F)_1 &\stackrel{\mathrm{df. }}{=} \mathsf{D}(F)_3 \stackrel{\mathrm{df. }}{=} F_1
\end{align*}

\item Each NT $\eta : F \Rightarrow G : \mathcal{C} \to \mathcal{D}$ to the NTwD $\mathsf{D}(\eta) \stackrel{\mathrm{df. }}{=} (\eta, t^{\mathsf{D}(\eta)}) : \mathsf{D}(F) \Rightarrow \mathsf{D}(G) : \mathsf{D}(\mathcal{C}) \to \mathsf{D}(\mathcal{D})$, where $t^{\mathsf{D}(\eta)}$ is given by:
\begin{equation*}
t^{\mathsf{D}(\eta)}_{\Gamma, \Theta} \stackrel{\mathrm{df. }}{=} F(\Gamma) \times F(\Theta) \stackrel{p_2}{\to} F(\Theta) \stackrel{\eta_\Theta}{\to} G(\Theta)
\end{equation*}
for all $\Gamma \in \mathsf{D}(\mathcal{C})$ and $\Theta \in \mathscr{D}_{\mathsf{D}(\mathcal{C})}(\Gamma)$.
\end{itemize}
\end{definition}

\begin{remark}
The domain of $\mathsf{D}$ must be $\mathsf{Ctx}\mathbb{CCC}$, not $\mathbb{CCC}$, since the domain of the functor $\mathsf{S}$ given below must be $\mathsf{ConCtx}\mathbb{CCC_D}$. 
\end{remark}

\begin{lemma}[Well-defined $\mathsf{D}$]
\label{WellDefinedD}
The 2-functor $\mathsf{D}$ is well-defined.
\end{lemma}
\begin{proof}
Straightforward.
\end{proof}

Next, we need a preliminary operation before defining the 2-functor $\mathsf{S}$:
\begin{definition}[Right-shifting]
Let $\mathcal{C}$ be a constant CtxCCCwD, where $T \in \mathcal{C}$ is the specified terminal object.
Given an object $\Delta \in \mathcal{C}$ such that $\Delta \neq T$, we define a D-object $\mathscr{R}_{\mathcal{C}}(\Delta) \in \mathscr{D}_{\mathcal{C}}(T)$, called the \emph{\bfseries right-shifting} of $\Delta$, by induction on the length of $\Delta$:
\begin{equation*}
\mathscr{R}_{\mathcal{C}}(\Delta) \stackrel{\mathrm{df. }}{=} \begin{cases} C \ &\text{if $\Delta = T . C$}; \\ \Sigma(\mathscr{R}_{\mathcal{C}}(\Gamma . A), \underline{B}\{!_{T.\mathscr{R}_{\mathcal{C}}(\Gamma . A)}\}) &\text{if $\Delta = \Gamma . A . B$} \end{cases}
\end{equation*}
together with a morphism $r_{\mathcal{C}}(\Delta) : T . \mathscr{R}_{\mathcal{C}}(\Delta) \to \Delta$ and a D-morphism $\ell_{\mathcal{C}}(\Delta) : \Delta \rightarrowtriangle \mathscr{R}_{\mathcal{C}}(\Delta)\{!_\Delta\}$ given by:
\begin{align*}
r_{\mathcal{C}}(\Delta) &\stackrel{\mathrm{df. }}{=} \begin{cases} T.C \stackrel{\mathit{id}_{T.C}}{\to} T.C \ &\text{if $\Delta = T.C$}; \\ T.\Sigma(\mathscr{R}_{\mathcal{C}}(\Gamma.A), \underline{B}\{!_{T.\mathscr{R}_{\mathcal{C}}(\Gamma . A)}\}) \stackrel{\langle r_{\mathcal{C}}(\Gamma.A) \circ \langle \pi_1, \varpi_1 \rangle, \varpi_2 \rangle}{\to} \Gamma . A . B &\text{if $\Delta = \Gamma. A . B$} \end{cases} \\
\ell_{\mathcal{C}}(\Delta) &\stackrel{\mathrm{df. }}{=} \begin{cases} T.C \stackrel{\pi_2}{\to} C\{!_{T.C}\} \ &\text{if $\Delta = T.C$}; \\ \Gamma . A . B \stackrel{\Lbag \ell_{\mathcal{C}}(\Gamma.A) \circ \pi_1, \varpi_2 \Rbag}{\to} \Sigma(\mathscr{R}_{\mathcal{C}}(\Gamma.A), \underline{B}\{!_{T.\mathscr{R}_{\mathcal{C}}(\Gamma . A)}\})\{!_{\Gamma.A.B}\} &\text{if $\Delta = \Gamma. A . B$.} \end{cases}
\end{align*}
\end{definition}

\begin{notation}
We often omit the subscript $\mathcal{C}$ on $\mathscr{R}_{\mathcal{C}}$, $r_{\mathcal{C}}$ and $\ell_{\mathcal{C}}$ if it does not bring any confusion. 
\end{notation}

\begin{lemma}[Right-shifting-lemma]
Given a constant CtxCCCwD $\mathcal{C}$, where $T \in \mathcal{C}$ is the specified terminal object, the operations $\mathscr{R}_{\mathcal{C}}$, $r_{\mathcal{C}}$ and $\ell_{\mathcal{C}}$ are all well-defined, and we have:
\begin{align*}
r_{\mathcal{C}}(\Delta) \circ \langle !_\Delta, \ell_{\mathcal{C}}(\Delta) \rangle &=\mathit{id}_{\Delta} \\
\langle !_\Delta, \ell_{\mathcal{C}}(\Delta) \{ r_{\mathcal{C}}(\Delta) \} \rangle &= \mathit{id}_{T . \mathscr{R}_{\mathcal{C}}(\Delta)}
\end{align*}
for all $\Delta \in \mathcal{C}$ such that $\Delta \neq T$.
\end{lemma}
\begin{proof}
By induction on the length of $\Delta$.
\end{proof}

Now, we are ready to define:
\begin{definition}[The 2-functor $\mathsf{S}$]
The 2-functor $\mathsf{S} : \mathsf{ConCtx}\mathbb{CCC_D} \to \mathsf{Ctx}\mathbb{CCC}$ maps each constant CtxCCCwD $\mathcal{D} = (\mathcal{D}, T, \_.\_, \pi, 1, \Sigma, \varpi, \Pi, \mathit{dev})$ to the CtxCCC $\mathsf{S}(\mathcal{D})$ given by:
\begin{itemize}

\item The underlying category is $\mathcal{D}$ equipped with the terminal object $T \in \mathcal{D}$;

\item $\mathscr{A}_{\mathsf{S}(\mathcal{D})} \stackrel{\mathrm{df. }}{=} \{ T \} \cup \{ T . \underline{A} \mid A \in \mathsf{dob}(\mathcal{D}) \}$;

\item Given $\Delta, \Gamma \in \mathsf{S}(\mathcal{D})$, we define the product of $\Delta$ and $\Gamma$ by:
\begin{equation*}
\Delta \stackrel{p_1}{\leftarrow} \Delta \times \Gamma \stackrel{p_2}{\rightarrow} \Gamma  \stackrel{\mathrm{df. }}{=} \Delta \stackrel{\pi_1}{\leftarrow} \Delta . \mathscr{R}_{\mathcal{C}}(\Gamma)\{!_\Delta\} \stackrel{\langle !_{\Delta . \mathscr{R}_{\mathcal{C}}(\Gamma)\{!_\Delta\}}, \pi_2 \rangle}{\rightarrow} T . \mathscr{R}_{\mathcal{C}}(\Gamma) \stackrel{r_{\mathcal{C}}(\Gamma)}{\rightarrow} \Gamma 
\end{equation*}

\item Given $\Gamma, \Delta \in \mathsf{S}(\mathcal{D})$, we define the exponential $\Gamma^\Delta \times \Delta \stackrel{\mathit{ev}_{\Delta, \Gamma}}{\rightarrow} \Gamma$, where $\Gamma^\Delta \stackrel{\mathrm{df. }}{=} \Delta \Rightarrow \Gamma$, by induction on the length of $\Gamma$:
\begin{align*}
&\Gamma^\Delta \times \Delta \stackrel{\mathit{ev}_{\Delta, \Gamma}}{\rightarrow} \Gamma \\ \stackrel{\mathrm{df. }}{=} &\begin{cases} T \times \Delta \stackrel{!}{\rightarrow} T \ &\text{if $\Gamma = T$} \\ (\Gamma'^\Delta) . \Pi(\mathscr{R}(\Delta)\{!\}, \underline{A}\{!\}) . \mathscr{R}(\Delta)\{!\} \stackrel{\langle \mathit{ev}_{\Delta, \Gamma'} \circ \langle \pi_1 \circ \pi_1, \pi_2 \rangle, \mathit{dev}_{\mathscr{R}(\Delta)\{!\}, A\{!\}} \circ \langle \langle !, \pi_2 \rangle, \pi_2\{\pi_1\}\rangle \rangle}{\rightarrow} \Gamma' . A &\text{if $\Gamma = \Gamma' . A$} \end{cases}
\end{align*}

\end{itemize}
and maps each CtxCCF $F$ and NTwD $(\eta, t)$ to their respective restrictions 
\begin{align*}
\mathsf{S}(F) &\stackrel{\mathrm{df. }}{=} (F_0, F_1) \\
\mathsf{S}(\eta, t) &\stackrel{\mathrm{df. }}{=} \eta
\end{align*}
to 0-cells and 1-cells. 
\end{definition}

\begin{lemma}[Well-defined $\mathsf{S}$]
The 2-functors $\mathsf{S}$ is well-defined.
\end{lemma}
\begin{proof}
Let $\mathcal{D} = (\mathcal{D}, T, \_.\_, \pi, 1, \Sigma, \varpi, \Pi, \mathit{dev})$ be a constant CtxCCCwD. 
It suffices to show that the binary products and the exponentials in $\mathsf{S}(\mathcal{D})$ are well-defined since then it is clearly contextual.

Let us first consider binary products.
Given $\Theta, \Delta, \Gamma \in \mathsf{S}(\mathcal{D})$, $\phi : \Theta \to \Delta$ and $\psi : \Theta \to \Gamma$ in $\mathsf{S}(\mathcal{D})$, we define the pairing $(\phi, \psi) : \Theta \to \Delta \times \Gamma$ of $\phi$ and $\psi$ to be 
\begin{equation*}
\langle \phi, \ell_{\mathcal{D}}(\Gamma)\{\psi\} \rangle : \Theta \to \Delta . \mathscr{R}_{\mathcal{C}}(\Gamma)\{!_\Delta\}. 
\end{equation*}
Then, we have:
\begin{align*}
p_1 \circ \langle \phi, \ell_{\mathcal{D}}(\Gamma)\{\psi\} \rangle &= \pi_1 \circ \langle \phi, \ell_{\mathcal{D}}(\Gamma)\{\psi\} \rangle \\
&= \phi
\end{align*}
and
\begin{align*}
p_2 \circ \langle \phi, \ell_{\mathcal{D}}(\Gamma)\{\psi\} \rangle &= (r_{\mathcal{D}}(\Gamma) \circ \langle !, \pi_2 \rangle) \circ \langle \phi, \ell_{\mathcal{D}}(\Gamma)\{\psi\} \rangle \\
&= r_{\mathcal{D}}(\Gamma) \circ (\langle !, \pi_2 \rangle \circ \langle \phi, \ell_{\mathcal{D}}(\Gamma)\{\psi\} \rangle) \\
&= r_{\mathcal{D}}(\Gamma) \circ \langle !, \ell_{\mathcal{D}}(\Gamma)\{\psi\} \rangle \\
&= r_{\mathcal{D}}(\Gamma) \circ (\langle !, \ell_{\mathcal{D}}(\Gamma) \rangle \circ \psi) \\
&= (r_{\mathcal{D}}(\Gamma) \circ \langle !, \ell_{\mathcal{D}}(\Gamma) \rangle) \circ \psi \\
&= \mathit{id}_\Gamma \circ \psi \\
&= \psi.
\end{align*}
Moreover, given $\varphi : \Theta \to \Delta . \mathscr{R}_{\mathcal{C}}(\Gamma)\{!\}$ in $\mathsf{S}(\mathcal{D})$, we have:
\begin{align*}
\langle p_1 \circ \varphi, p_2 \circ \varphi \rangle &= \langle \pi_1 \circ \varphi, \ell_{\mathcal{D}}(\Gamma) \{ (r_{\mathcal{D}}(\Gamma) \circ \langle !, \pi_2 \rangle) \circ \varphi \} \rangle \\
&= \langle \pi_1 \circ \varphi, \ell_{\mathcal{D}}(\Gamma) \{ r_{\mathcal{D}}(\Gamma) \} \{ \langle !, \pi_2 \rangle \} \{ \varphi \} \rangle \\
&= \langle \pi_1, \ell_{\mathcal{D}}(\Gamma) \{ r_{\mathcal{D}}(\Gamma) \} \{ \langle !, \pi_2 \rangle \} \rangle \circ \varphi \\
&= \langle \pi_1, \pi_2 \{ \langle ! \circ r_{\mathcal{D}}(\Gamma), \ell_{\mathcal{D}}(\Gamma) \{ r_{\mathcal{D}}(\Gamma) \} \rangle \} \{ \langle !, \pi_2 \rangle \} \rangle \circ \varphi \\
&= \langle \pi_1, \pi_2 \{ \langle !, \ell_{\mathcal{D}}(\Gamma) \rangle \circ r_{\mathcal{D}}(\Gamma) \} \{ \langle !, \pi_2 \rangle \} \rangle \circ \varphi \\
&= \langle \pi_1, \pi_2 \{ \mathit{id}_{\Delta . \mathscr{R}_{\mathcal{C}}(\Gamma)\{!\}} \circ \langle !, \pi_2 \rangle \} \rangle \circ \varphi \\
&= \langle \pi_1, \pi_2 \rangle \circ \varphi \\
&= \varphi.
\end{align*}
Hence, the diagram $\Delta \stackrel{p_1}{\leftarrow} \Delta \times \Gamma \stackrel{p_2}{\rightarrow} \Gamma$ is a well-defined product of $\Delta$ and $\Gamma$.

Next, let us consider the exponential $\Delta \Rightarrow \Gamma$ from $\Delta$ to $\Gamma$.
Since the base case $\Gamma = T$ is trivial, assume $\Gamma = \Gamma' . A$ for some $\Gamma' \in \mathcal{D}$ and $A \in \mathscr{D}_{\mathcal{D}}(\Gamma')$.
Given $\vartheta : \Theta \times \Delta \to \Gamma' . A$ in $\mathsf{S}(\mathcal{D})$, we define:
\begin{equation*}
\lambda_{\Delta, \Gamma}(\vartheta) \stackrel{\mathrm{df. }}{=} \langle \lambda_{\Delta, \Gamma'}(\pi_1 \circ \vartheta), \Lambda (\pi_2 \{\vartheta \}) \{ \pi_1 \} \rangle : \Theta \times \Delta \rightarrow (\Delta \Rightarrow \Gamma') . \Pi(\mathscr{R}_{\mathcal{C}}(\Delta)\{!\}, \underline{A}\{!\})
\end{equation*}
such that:
\begin{align*}
&\mathit{ev}_{\Delta, \Gamma} \circ (\lambda_{\Delta, \Gamma}(\vartheta) \circ p_1, p_2) \\
= \ &\langle \mathit{ev}_{\Delta, \Gamma'} \circ \langle \pi_1 \circ \pi_1, \pi_2 \rangle, \mathit{dev}_{\mathscr{R}(\Delta)\{!\}, A\{!\}} \circ \langle \langle !, \pi_2 \rangle, \pi_2\{\pi_1\}\rangle \rangle \circ \langle \langle \lambda_{\Delta, \Gamma'}(\pi_1 \circ \vartheta) \circ p_1, \Lambda (\pi_2 \{\vartheta \}) \{ \pi_1 \} \rangle, \ell(\Delta)\{p_2\} \rangle \\
= \ &\langle \mathit{ev}_{\Delta, \Gamma'} \circ \langle \lambda_{\Delta, \Gamma'}(\pi_1 \circ \vartheta) \circ p_1, \ell(\Delta)\{p_2\} \rangle, \mathit{dev}_{\mathscr{R}(\Delta)\{!\}, A\{!\}} \circ \langle \langle !, \ell(\Delta)\{p_2\} \rangle, \Lambda(\pi_2\{\vartheta \}) \{ \pi_1\} \rangle \rangle \\
= \ &\langle \mathit{ev}_{\Delta, \Gamma'} \circ \langle \lambda_{\Delta, \Gamma'}(\pi_1 \circ \vartheta) \circ p_1, \ell(\Delta)\{r(\Delta)\{\langle !, \pi_2 \rangle\}\} \rangle, \mathit{dev}_{\mathscr{R}(\Delta)\{!\}, A\{!\}} \circ \langle \langle !, \ell(\Delta)\{p_2\} \rangle, \Lambda(\pi_2\{\vartheta \}) \{ \pi_1\} \rangle \rangle \\
= \ &\langle \mathit{ev}_{\Delta, \Gamma'} \circ \langle \lambda_{\Delta, \Gamma'}(\pi_1 \circ \vartheta) \circ p_1, \pi_2\{r(\Delta)\{\langle !, \ell(\Delta) \rangle\}\} \rangle, \mathit{dev}_{\mathscr{R}(\Delta)\{!\}, A\{!\}} \circ \langle \langle !, \ell(\Delta)\{p_2\} \rangle, \Lambda(\pi_2\{\vartheta \}) \{ \pi_1\} \rangle \rangle \\
= \ &\langle \mathit{ev}_{\Delta, \Gamma'} \circ \langle \lambda_{\Delta, \Gamma'}(\pi_1 \circ \vartheta) \circ p_1, \pi_2\{\mathit{id}_{T.\mathscr{R}(\Delta)}\} \rangle, \mathit{dev}_{\mathscr{R}(\Delta)\{!\}, A\{!\}} \circ \langle \langle !, \ell(\Delta)\{p_2\} \rangle, \Lambda(\pi_2\{\vartheta \}) \{ \pi_1\} \rangle \rangle \\
= \ &\langle \pi_1 \circ \vartheta, \mathit{dev}_{\mathscr{R}(\Delta)\{!\}, A\{!\}} \circ \langle \langle \pi_1, \pi_2 \rangle, \Lambda(\pi_2\{\vartheta \}) \{ \pi_1\} \rangle \rangle \\
= \ &\langle \pi_1 \circ \vartheta, \mathit{dev}_{\mathscr{R}(\Delta)\{!\}, A\{!\}} \circ \langle \mathit{id}_{T.\mathscr{R}(\Delta)}, \Lambda(\pi_2\{\vartheta \}) \{ \pi_1\} \rangle \rangle \\
= \ &\langle \pi_1 \circ \vartheta, \pi_2\{\vartheta \} \rangle \\
= \ &\langle \pi_1, \pi_2 \rangle \circ \vartheta \\
= \ &\mathit{id}_{T.\mathscr{R}(\Delta)} \circ \vartheta \\
= \ &\vartheta
\end{align*}
and for any $\kappa : \Theta \rightarrow (\Delta \Rightarrow \Gamma)$ in $\mathsf{S}(\mathcal{D})$ we have:
\begin{align*}
& \lambda_{\Delta, \Gamma}(\mathit{ev}_{\Delta, \Gamma} \circ (\kappa \circ p_1, p_2)) \\
= \ &\lambda_{\Delta, \Gamma}(\mathit{ev}_{\Delta, \Gamma} \circ \langle \kappa \circ \pi_1, \pi_2 \rangle) \\
= \ &\langle \lambda_{\Delta, \Gamma'} (\pi_1 \circ \mathit{ev}_{\Delta, \Gamma} \circ \langle \kappa \circ \pi_1, \pi_2 \rangle), \Lambda(\pi_2\{\mathit{ev}_{\Delta, \Gamma} \circ \langle \kappa \circ \pi_1, \pi_2 \rangle\})\{\pi_1\} \rangle \\
= \ &\langle \lambda_{\Delta, \Gamma'}(\mathit{ev}_{\Delta, \Gamma'} \circ \langle \pi_1 \circ \pi_1, \pi_2 \rangle \circ \langle \kappa \circ \pi_1, \pi_2 \rangle), \Lambda(\mathit{dev}_{\mathscr{R}(\Delta)\{!\}, A\{!\}} \{ \langle \langle !, \pi_2 \rangle, \pi_2\{\pi_1\} \rangle \} \{ \langle \kappa \circ \pi_1, \pi_2 \rangle \})\{ \pi_1 \} \rangle \\
= \ &\langle \lambda_{\Delta, \Gamma'}(\mathit{ev}_{\Delta, \Gamma'} \circ \langle \pi_1 \circ \kappa \circ \pi_1, \pi_2 \rangle), \Lambda(\mathit{dev}_{\mathscr{R}(\Delta)\{!\}, A\{!\}} \{ \langle \langle \pi_1, \pi_2 \rangle, \pi_2\{\kappa\}\{\pi_1\}\rangle\})\{\pi_1\} \rangle \\
= \ &\langle \pi_1 \circ \kappa, \Lambda(\mathit{dev}_{\mathscr{R}(\Delta)\{!\}, A\{!\}} \{ \langle \mathit{id}_{T.\mathscr{R}(\Delta)}, \pi_2\{\kappa\}\{\pi_1\}\rangle\})\{\pi_1\} \rangle \\
= \ &\langle \pi_1 \circ \kappa, \pi_2\{\kappa\} \rangle \ \text{(by Lemma~\ref{LemTranspositionInPseudoPiSpaces})} \\
= \ &\langle \pi_1, \pi_2 \rangle \circ \kappa \\
= \ &\mathit{id}_{\Gamma^\Delta} \circ \kappa \\
= \ &\kappa
\end{align*}
which establishes that the exponential $\Delta \Rightarrow \Gamma$ in $\mathsf{S}(\mathcal{D})$ is well-defined. 
\end{proof}

\begin{theorem}[DS-theorem]
\label{ThmDSTheorem}
There is the 2-equivalence: 
\begin{align*}
\mathsf{D} : \mathsf{Ctx}\mathbb{CCC} &\simeq \mathsf{ConCtx}\mathbb{CCC_D} : \mathsf{S}.
\end{align*}
\end{theorem}
\begin{proof}
First, the 2-natural isomorphism $\eta : \mathit{id}_{\mathsf{Ctx}\mathbb{CCC}} \stackrel{\sim}{\to} \mathsf{S} \circ \mathsf{D}$ has the identity CtxCCF $\mathit{id}_{\mathcal{C}} : \mathcal{C} \stackrel{\sim}{\rightarrow} \mathcal{C}$ as the component $\eta_{\mathcal{C}} : \mathcal{C} \stackrel{\sim}{\to} \mathsf{S} \circ \mathsf{D}(\mathcal{C})$ for each CtxCCC $\mathcal{C}$, whose naturality is obvious.

Next, the 2-natural isomorphism $\epsilon : \mathsf{D} \circ \mathsf{S} \stackrel{\sim}{\rightarrow} \mathit{id}_{\mathsf{ConCtx}\mathbb{CCC_D}}$ has for each constant CtxCCCwD $\mathcal{D}$ as the component $\epsilon_{\mathcal{D}} : \mathsf{D} \circ \mathsf{S}(\mathcal{D}) \to \mathcal{D}$ the strict CCFwD $\mathsf{D} \circ \mathsf{S}(\mathcal{D}) \to \mathcal{D}$ that maps objects and morphisms in $\mathsf{D} \circ \mathsf{S}(\mathcal{D})$ to themselves, D-objects $T . \underline{A} \in \mathscr{D}_{\mathsf{D} \circ \mathsf{S}(\mathcal{D})}(\Gamma)$ (resp. $T \in \mathscr{D}_{\mathsf{D} \circ \mathsf{S}(\mathcal{D})}(\Gamma)$) to $\underline{A}\{!_\Gamma\} \in \mathscr{D}_{\mathcal{D}}(\Gamma)$ (resp. $1\{!_\Gamma\} \in \mathscr{D}_{\mathcal{D}}(\Gamma)$), and D-morphisms $f \in \mathscr{D}_{\mathsf{D} \circ \mathsf{S}(\mathcal{D})}(\Gamma, T.\underline{A})$ (resp. $!_\Gamma \in \mathscr{D}_{\mathsf{D} \circ \mathsf{S}(\mathcal{D})}(\Gamma, T)$) to $\pi_2\{f\} \in \mathscr{D}_{\mathcal{D}}(\Gamma, \underline{A}\{!_\Gamma\})$ (resp. $\bm{!}_\Gamma \in \mathscr{D}_{\mathcal{D}}(\Gamma, 1\{!_\Gamma\})$), which has the obvious inverse.
For naturality of $\epsilon$, let $\mathcal{D}, \mathcal{D'} \in \mathsf{ConCtx}\mathbb{CCC_D}$; then, it is easy to see that the diagram
\begin{diagram}
[\mathcal{D}, \mathcal{D'}] && \rTo^{\mathsf{D} \circ \mathsf{S}} && [\mathsf{D} \circ \mathsf{S} (\mathcal{D}), \mathsf{D} \circ \mathsf{S} (\mathcal{D'})] \\
&&&& \\
\dTo^{\mathit{id}_{\mathsf{ConCtx}\mathbb{CCC_D}}} &&&& \dTo_{\mathsf{ConCtx}\mathbb{CCC_D}(\mathsf{D} \circ \mathsf{S} (\mathcal{D}), \epsilon_{\mathcal{D'}})} \\
&&&& \\
[\mathcal{D}, \mathcal{D'}] && \rTo^{\mathsf{ConCtx}\mathbb{CCC_D}(\epsilon_{\mathcal{D}}, \mathcal{D'})} && [\mathsf{D} \circ \mathsf{S} (\mathcal{D}), \mathcal{D'}]
\end{diagram}
commutes by Lemma~\ref{LemCtxSCCwDLemma}, which completes the proof.
\end{proof}

In particular, the 2-natural isomorphism $\epsilon : \mathsf{D} \circ \mathsf{S} \stackrel{\sim}{\rightarrow} \mathit{id}_{\mathsf{ConCtx}\mathbb{CCC_D}}$ in the proof of Theorem~\ref{ThmDSTheorem} induces for each constant CtxCCCwD $\mathcal{D}$, an object $\Gamma \in \mathcal{D}$ and a D-object $A \in \mathscr{D}_{\mathcal{D}}(\Gamma)$ the bijective correspondence
\begin{equation*}
\mathscr{D}_{\mathcal{D}}(\Gamma, A)  \cong \mathcal{D}(\Gamma, T . \underline{A})
\end{equation*}
between D-morphisms $\Gamma \rightarrowtriangle A$ and morphisms $\Gamma \rightarrow T . \underline{A}$ in $\mathcal{D}$.
In the term models of STLCs, this explains why we may identify terms and singleton context morphisms. 

Combined with the main result in the next section, we may recover the conventional interpretation of STLCs in CCCs from our interpretation of them in CtxCCCwDs.

\section{Correspondence between MLTTs and CCCwDs}
\label{TheoryCategoryCorrespondenceBetweenMLTTsAndCCCwDs}
In this final section, analogous to the 2-categorical correspondence between CCCs and STLCs, we establish a 2-categorical correspondence between CCCwDs and MLTTs.

\subsection{Algebras for MLTTs}
\begin{definition}[Structures for MLTTs]
Let $\mathcal{T}$ be an MLTT and $\mathcal{C} = (\mathcal{C}, T, \_ . \_, \pi, 1, \Sigma, \varpi, \Pi, \mathit{dev})$ a strict CCCwD. 
A \emph{\bfseries structure} for $\mathcal{T}$ in $\mathcal{C}$ is an assignment $S$ of:
\begin{itemize}

\item A pair of an object $S(\mathscr{F}_{\mathcal{T}}(\mathsf{C})) \in \mathcal{C}$ and a D-object $S(\mathsf{C}) \in \mathscr{D}_{\mathcal{C}}(S(\mathscr{F}_{\mathcal{T}}(\mathsf{C})))$ to each type-constant $\mathsf{C} \in C_{\mathcal{T}}$;

\item A triple of an object $S(\mathscr{F}^{\mathit{dom}}_{\mathcal{T}}(\mathsf{F})) \in \mathcal{C}$, a D-object $S(\mathscr{F}^{\mathit{cod}}_{\mathcal{T}}(\mathsf{F})) \in \mathscr{D}_{\mathcal{C}}(S(\mathscr{F}^{\mathit{dom}}_{\mathcal{T}}(\mathsf{F})))$ and a D-morphism $S(\mathsf{F}) \in \mathscr{D}_{\mathcal{C}}(S(\mathscr{F}^{\mathit{dom}}_{\mathcal{T}}(\mathsf{F})),  S(\mathscr{F}^{\mathit{cod}}_{\mathcal{T}}(\mathsf{F})))$ to each term-constant $\mathsf{F} \in C_{\mathcal{T}}$.

\end{itemize}
\end{definition}

\begin{definition}[The partial interpretation of MLTTs by structures]
Let $\mathcal{T}$ be an MLTT and $\mathcal{C} = (\mathcal{C}, T, \_ . \_, \pi, 1, \Sigma, \varpi, \Pi, \mathit{dev})$ a strict CCCwD.
Any structure $S$ for $\mathcal{T}$ in $\mathcal{C}$ induces the partial interpretation $\llbracket \_ \rrbracket_S$ of $\mathsf{MLTT(1, \Pi, \Sigma)}$ in the induced CwF $\mathcal{C}$ that supports $1$-, $\Pi$- and $\Sigma$-types in the strict sense that extends Definition~\ref{DefMLTTInCwFs} to $\mathcal{T}$ by:
\begin{align*}
\llbracket \mathsf{C(\bm{\mathsf{f}})} \rrbracket_S &\stackrel{\mathrm{df. }}{\simeq} S(\mathsf{C}) \{ \llbracket \bm{\mathsf{f}} \rrbracket \}_{\mathcal{C}} \\
\llbracket \mathsf{F(\bm{\mathsf{f}})} \rrbracket_S &\stackrel{\mathrm{df. }}{\simeq} S(\mathsf{F}) \{ \llbracket \bm{\mathsf{f}} \rrbracket \}_{\mathcal{C}}
\end{align*}
for all type-constants $\mathsf{C} \in C_{\mathcal{T}}(\mathsf{A_1}, \mathsf{A_2}, \dots, \mathsf{A_k})$, term-constants $\mathsf{F} \in C_{\mathcal{T}}(\mathsf{A_1}, \mathsf{A_2}, \dots, \mathsf{A_k}; \mathsf{B})$ and pre-context morphisms $\bm{\mathsf{f}}$.
\end{definition}

\begin{remark}
It is partially defined because  $\bm{\mathsf{f}}$ may not be a well-defined context morphism with the appropriate domain and codomain.
\end{remark}

\begin{definition}[Algebras for MLTTs]
An \emph{\bfseries algebra} for an MLTT $\mathcal{T}$ (or a \emph{\bfseries $\bm{\mathcal{T}}$-algebra}) in a strict CCCwD $\mathcal{C}$ is a structure $S$ for $\mathcal{T}$ in $\mathcal{C}$ such that the induced partial interpretation $\llbracket \_ \rrbracket_S$ satisfies:
\begin{itemize}

\item If $\mathsf{C} \in C_{\mathcal{T}}(\mathsf{A_1}, \mathsf{A_2}, \dots, \mathsf{A_k})$ is a type-constant, and $\mathsf{\vdash_{\mathcal{T}} x_1 : A_1, x_2 : A_2, \dots, x_k : A_k \ ctx}$, then $\llbracket \mathsf{x_1 : A_1, x_2 : A_2, \dots, x_k : A_k} \rrbracket_S \simeq S(\mathscr{F}_{\mathcal{T}}(\mathsf{C}))$\footnote{It of course means that both sides are defined because so is the right-hand side. The same remark is applied to similar cases below.};

\item If $\mathsf{F} \in C_{\mathcal{T}}(\mathsf{A_1}, \mathsf{A_2}, \dots, \mathsf{A_k}; \mathsf{B})$ is a term-constant, and $\mathsf{x_1 : A_1, x_2 : A_2, \dots, x_k : A_k \vdash_{\mathcal{T}} B \ type}$, then $\llbracket \mathsf{x_1 : A_1, x_2 : A_2, \dots, x_k : A_k} \rrbracket_S \simeq S(\mathscr{F}^{\mathit{dom}}_{\mathcal{T}}(\mathsf{F}))$ and $\llbracket \mathsf{B} \rrbracket_S \simeq S(\mathscr{F}^{\mathit{cod}}_{\mathcal{T}}(\mathsf{F}))$;

\item If $\mathsf{A = A' : Ty} \in A_{\mathcal{T}}$, then $\llbracket \mathsf{A} \rrbracket_S \simeq \llbracket \mathsf{A'} \rrbracket_S$;

\item If $\mathsf{a = a' : A} \in A_{\mathcal{T}}$, then $\llbracket \mathsf{a} \rrbracket_S \simeq \llbracket \mathsf{a'} \rrbracket_S$.

\end{itemize}
\end{definition}

\begin{lemma}[Substitution lemma]
\label{LemSubstLem}
Let $\mathcal{T}$ be an MLTT.
Given an algebra $S$ for $\mathcal{T}$ in a strict CCCwD $\mathcal{C}$, the induced interpretation $\llbracket \_ \rrbracket_S$ of $\mathcal{T}$ in $\mathcal{C}$ satisfies:
\begin{align*}
\llbracket \mathsf{A} \rrbracket_S \{ \llbracket \bm{\mathsf{f}} \rrbracket_S \} &\simeq \llbracket \mathsf{A[\bm{\mathsf{f}}]} \rrbracket_S \\
\llbracket \mathsf{a} \rrbracket_S \{ \llbracket \bm{\mathsf{f}} \rrbracket_S \} &\simeq \llbracket \mathsf{a[\bm{\mathsf{f}}]} \rrbracket_S \\
\llbracket \bm{\mathsf{g}} \rrbracket_S \{ \llbracket \bm{\mathsf{f}} \rrbracket_S \} &\simeq \llbracket \bm{\mathsf{g}} \circ \bm{\mathsf{f}} \rrbracket_S
\end{align*}
for all pre-types $\mathsf{A}$, pre-terms $\mathsf{a}$ and pre-context morphisms $\bm{\mathsf{f}}$ and $\bm{\mathsf{f}}$ in $\mathcal{T}$.
\end{lemma}
\begin{proof}
Since the lemma has been established for $\mathsf{MLTT(1, \Pi, \Sigma)}$ by induction on the lengths of the pre-syntax $\mathsf{A}$, $\mathsf{a}$ and $\bm{\mathsf{g}}$ in \cite{hofmann1997syntax}, it suffices to consider the case $\mathsf{A} \equiv \mathsf{C}(\bm{\mathsf{h}})$ and $\mathsf{a} \equiv \mathsf{F}(\bm{\mathsf{h}})$ for some $\mathsf{C} \in C_{\mathcal{T}}^{\mathsf{Ty}}$ and $\mathsf{F} \in C_{\mathcal{T}}^{\mathsf{Tm}}$.
Then, observe that:
\begin{align*}
\llbracket \mathsf{C}(\bm{\mathsf{h}}) \rrbracket_S \{ \llbracket \bm{\mathsf{f}} \rrbracket_S \} &\simeq \llbracket \mathsf{C} \rrbracket_S \{ \llbracket \bm{\mathsf{h}} \rrbracket_S \} \{ \llbracket \bm{\mathsf{f}} \rrbracket_S \} \ \text{(by the definition of the induced interpretation $\llbracket \_ \rrbracket_S$)} \\
&\simeq \llbracket \mathsf{C} \rrbracket_S \{ \llbracket \bm{\mathsf{h}} \rrbracket_S \circ \llbracket \bm{\mathsf{f}} \rrbracket_S \} \ \text{(by an axiom of CwDs)} \\
&\simeq \llbracket \mathsf{C} \rrbracket_S \{ \llbracket \bm{\mathsf{h}} \circ \bm{\mathsf{f}} \rrbracket_S \} \ \text{(by the induction hypothesis)} \\
&\simeq \llbracket \mathsf{C(\bm{\mathsf{h}} \circ \bm{\mathsf{f}})} \rrbracket_S \ \text{(by the definition of $\llbracket \_ \rrbracket$)} \\
&\simeq \llbracket \mathsf{C(\bm{\mathsf{h}})[\bm{\mathsf{f}}]} \rrbracket_S \ \text{(by the definition of partial composition of pre-context morphisms)}
\end{align*}
for any pre-context morphism $\bm{\mathsf{f}}$ in $\mathcal{T}$.
In the same manner, we may show that $\llbracket \mathsf{F}(\bm{\mathsf{h}}) \rrbracket_S \{ \llbracket \bm{\mathsf{f}} \rrbracket_S \} \simeq \llbracket \mathsf{F(\bm{\mathsf{h}})[\bm{\mathsf{f}}]} \rrbracket_S$ for any pre-context morphism $\bm{\mathsf{f}}$ in $\mathcal{T}$.

Finally, $\llbracket \bm{\mathsf{g}} \rrbracket_S \{ \llbracket \bm{\mathsf{f}} \rrbracket_S \} \simeq \llbracket \bm{\mathsf{g}} \circ \bm{\mathsf{f}} \rrbracket_S$ for any pre-context morphisms $\bm{\mathsf{f}}$ and $\bm{\mathsf{g}}$ in $\mathcal{T}$ may be shown by induction on the length of $\bm{\mathsf{g}}$, completing the proof.
\end{proof}

\begin{theorem}[Soundness of algebras for MLTTs]
The interpretation $\llbracket \_ \rrbracket_S$ of an MLTT $\mathcal{T}$ in a strict CCCwD $\mathcal{C}$ induced by an algebra $S$ for $\mathcal{T}$ in $\mathcal{C}$ is sound.
\end{theorem}
\begin{proof}
It suffices to consider the additional four rules: Type-Const, Term-Const, Type-Eq and Term-Eq.
Let us first consider a type $\mathsf{\Gamma \vdash C(\bm{\mathsf{f}}) \ type}$ generated by the rule Type-Const, where $\mathsf{C} \in C_{\mathcal{T}}(\mathsf{A_1}, \mathsf{A_2}, \dots, \mathsf{A_k})$.
Since $\mathcal{T}$ is well-formed, $\bm{\mathsf{f}}$ is a context morphism $\mathsf{\Gamma} \to \mathsf{\Delta}$ in $\mathcal{T}$, where $\mathsf{\Delta \stackrel{\mathrm{df. }}{\equiv} x_1 : A_1, x_2 : A_2, \dots, x_k : A_k}$.
By the induction hypothesis, $\llbracket \bm{\mathsf{f}} \rrbracket_S$ is a morphism $\llbracket \mathsf{\Gamma} \rrbracket_S \to \llbracket \mathsf{\Delta} \rrbracket_S$ in $\mathcal{C}$.
Also, $S(\mathscr{F}_{\mathcal{T}}(\mathsf{C})) = \llbracket \mathsf{\Delta} \rrbracket_S$ for $S$ is a $\mathcal{T}$-algebra.
Therefore, we have: 
\begin{equation*}
\llbracket \mathsf{C}(\bm{\mathsf{f}}) \rrbracket_S = S(\mathsf{C}) \{ \llbracket \bm{\mathsf{f}} \rrbracket_S \} \in \mathscr{D}_{\mathcal{C}}(\llbracket \mathsf{\Gamma} \rrbracket_S).
\end{equation*}

Similarly, with the help of Lemma~\ref{LemSubstLem}, we may handle the rule Term-Const.
Finally, the remaining two rules Type-Eq and Term-Eq are even simpler to deal with.
\end{proof}

\begin{definition}[Generic models of MLTTs]
Let $\mathcal{T}$ be an MLTT.
The \emph{\bfseries generic model} of $\mathcal{T}$ is a structure $G$ for $\mathcal{T}$ in the classifying CCCwD $\mathit{Cl}(\mathcal{T})$ defined by:
\begin{align*}
G(\mathsf{A_1}, \mathsf{A_2}, \dots, \mathsf{A_k})) &\stackrel{\mathrm{df. }}{=} \mathsf{\vdash x_1 : A_1, x_2 : A_2, \dots, x_k : A_k \ ctx} \\
G(\mathsf{C}) &\stackrel{\mathrm{df. }}{=} \mathsf{x_1 : A_1, x_2 : A_2, \dots, x_k : A_k \vdash C(x_1, x_2, \dots, x_k) \ type} \\
G(\mathsf{B}) &\stackrel{\mathrm{df. }}{=} \mathsf{x_1 : A_1, x_2 : A_2, \dots, x_k : A_k \vdash B[(x_1, x_2, \dots, x_k)] \ type} \\
G(\mathsf{F}) &\stackrel{\mathrm{df. }}{=} \mathsf{x_1 : A_1, x_2 : A_2, \dots, x_k : A_k \vdash F(x_1, x_2, \dots, x_k) : B[(x_1, x_2, \dots, x_k)]}
\end{align*}
for all type-constants $\mathsf{C} \in C_{\mathcal{T}}(\mathsf{A_1}, \mathsf{A_2}, \dots, \mathsf{A_k})$ and term-constants $\mathsf{F} \in C_{\mathcal{T}}(\mathsf{A_1}, \mathsf{A_2}, \dots, \mathsf{A_k}; \mathsf{B})$.
\end{definition}

\begin{lemma}[G-lemma]
The generic model $G$ of any MLTT $\mathcal{T}$ is an algebra for $\mathcal{T}$ in the classifying CCCwD $\mathit{Cl}(\mathcal{T})$. 
Moreover, the induced interpretation $\llbracket \_ \rrbracket_G$ of $\mathcal{T}$ in $\mathit{Cl}(\mathcal{T})$ by $G$ satisfies:
\begin{enumerate}

\item $\llbracket \mathsf{\vdash \Gamma \ ctx} \rrbracket_G = [\mathsf{\vdash \Gamma \ ctx}]$ for any context $\mathsf{\vdash \Gamma \ ctx}$ in $\mathcal{T}$;

\item $\llbracket \mathsf{\Gamma \vdash A \ type} \rrbracket_G = [\mathsf{\Gamma \vdash A \ type}]$ for any type $\mathsf{\Gamma \vdash A \ type}$ in $\mathcal{T}$;

\item $\llbracket \mathsf{\Gamma \vdash a : A} \rrbracket_G = [\mathsf{\Gamma \vdash a : A}]$ for any term $\mathsf{\Gamma \vdash a : A}$ in $\mathcal{T}$;

\item $\llbracket \mathsf{\Gamma \vdash A \ type} \rrbracket_G = \llbracket \mathsf{\Gamma \vdash A' \ type} \rrbracket_G$ iff $\mathsf{\Gamma \vdash A =A' \ type}$ is a theorem of $\mathcal{T}$ for any types $\mathsf{\Gamma \vdash A \ type}$ and $\mathsf{\Gamma \vdash A' \ type}$ in $\mathcal{T}$; and

\item $\llbracket \mathsf{\Gamma \vdash a : A} \rrbracket_G = \llbracket \mathsf{\Gamma \vdash a' : A} \rrbracket_G$ iff $\mathsf{\Gamma \vdash a = a' : A}$ is a theorem of $\mathcal{T}$ for any terms $\mathsf{\Gamma \vdash a : A}$ and $\mathsf{\Gamma \vdash a' : A}$ in $\mathcal{T}$.

\end{enumerate}
\end{lemma}
\begin{proof}
Again, it suffices to consider the four rules, but it is just straightforward.
\end{proof}

\begin{lemma}[Universal property of generic models]
Let $S$ be an algebra for an MLTT $\mathcal{T}$ in a CCCwD $\mathcal{C}$.
Then, there exists a strict CCFwD $U_S : \mathit{Cl}(\mathcal{T}) \to \mathcal{C}$ such that $U_S \circ \llbracket \_ \rrbracket_G = \llbracket \_ \rrbracket_S$, and such a CCFwD $U_S$ is unique up to NIwDs.
\end{lemma}
\begin{proof}
Since the proof just follows the proof of the corresponding statement in \cite{pitts2001categorical}, here we just give a proof sketch. 
Let us define the CCFwD $U_S : \mathit{Cl}(\mathcal{T}) \to \mathcal{C}$ by:
\begin{align*}
\mathsf{\vdash_{\mathcal{T}} \Gamma \ ctx} &\mapsto \llbracket \mathsf{\vdash_{\mathcal{T}} \Gamma \ ctx} \rrbracket_{S} \\
\mathsf{\Gamma \vdash_{\mathcal{T}} A \ ctx} &\mapsto \llbracket \mathsf{\Gamma \vdash_{\mathcal{T}} A \ type} \rrbracket_{S} \\
\mathsf{\Gamma \vdash_{\mathcal{T}} a : A} &\mapsto \llbracket \mathsf{\Gamma \vdash_{\mathcal{T}} a : A} \rrbracket_{S}
\end{align*}
It is easy to verify that $U_S$ is cartesian closed.
Also, we may show that $U_S \circ \llbracket \_ \rrbracket_G = \llbracket \_ \rrbracket_S$ by induction on derivations of judgements.
Finally, uniqueness of $U_S$ is straightforward to establish. 
\end{proof}

\begin{definition}[Morphisms between algebras for MLTTs]
Let $\mathcal{T}$ be an MLTT, and $S$ and $R$ $\mathcal{T}$-algebras in a strict CCCwD $\mathcal{C}$.
A \emph{\bfseries $\bm{\mathcal{T}}$-morphisms} from $S$ to $R$ is an NTwD $U_S \to U_R$.
\end{definition}

\begin{definition}[Categories of algebras for MLTTs]
Given an MLTT $\mathcal{T}$ and a strict CCCwD $\mathcal{C}$, the category $\mathcal{T}\text{-}\mathcal{ALG}(\mathcal{C})$ is given by:
\begin{itemize}

\item Objects are $\mathcal{T}$-algebras in $\mathcal{C}$;

\item Morphisms $S \to R$ are $\mathcal{T}$-morphisms from $S$ to $R$;

\item Composition is the vertical composition of NTwDs; and

\item Identities are identity NTwDs.

\end{itemize}
\end{definition}

Given an MLTT $\mathcal{T}$, if we regard $\mathcal{T}$-algebras $S$ in a strict CCCwD $\mathcal{C}$ as the strict CCFwDs $U_S : \mathit{Cl}(\mathcal{T}) \to \mathcal{C}$, then the category $\mathcal{T}\text{-}\mathcal{ALG}(\mathcal{C})$ is clearly a subcategory of $\mathbb{CCC_D}(\mathit{Cl}(\mathcal{T}), \mathcal{C})$.
It then turns out that this construction constitutes the following equivalence of categories:
\begin{corollary}[Equivalence between algebras and interpretations]
\label{CoroEqBetweenAlgebrasAndInterpretations}
Given an MLTT $\mathcal{T}$ and a strict CCCwD $\mathcal{C}$, we have $\mathcal{T}\text{-}\mathcal{ALG}(\mathcal{C}) \simeq \mathbb{CCC_D}(\mathit{Cl}(\mathcal{T}), \mathcal{C})$.
\end{corollary}
\begin{proof}
Let us define the functor $U^{\mathcal{T}, \mathcal{C}}_{(\_)} : \mathcal{T}\text{-}\mathcal{ALG}(\mathcal{C}) \to \mathbb{CCC_D}(\mathit{Cl}(\mathcal{T}), \mathcal{C})$ that maps $\mathcal{T}$-algebras $S$ in $\mathcal{C}$ to the induced CCFwDs $U_S : \mathit{Cl}(\mathcal{T}) \to \mathcal{C}$, and 1- and 2-cells in $\mathcal{T}\text{-}\mathcal{ALG}(\mathcal{C})$ to themselves.
Then, it suffices to show that this functor $U^{\mathcal{T}, \mathcal{C}}_{(\_)}$  is essentially surjective on objects. 

Let $F : \mathit{Cl}(\mathcal{T}) \to \mathcal{C}$ be any CCFwD.
Note that $F$ is completely determined up to NIwDs by its restriction to types of the form $\mathsf{\Gamma \vdash_{\mathcal{T}} C(\bm{\mathsf{f}}) \ type}$, where $\mathsf{C} \in C_{\mathcal{T}}^{\mathsf{Ty}}$, and terms of the form $\mathsf{\Gamma \vdash_{\mathcal{T}} F(\bm{\mathsf{f}}) : B(\bm{\mathsf{f}})}$, where $\mathsf{F} \in C_{\mathcal{T}}^{\mathsf{Tm}}$.
But then, this restriction of $F$ clearly defines a $\mathcal{T}$-algebra $S$ in $\mathcal{C}$ such that $U_S \cong F$, which completes the proof.
\end{proof}

\subsection{Theory-Category Correspondence between MLTTs and CCCwDs}
\begin{definition}[The 2-category $\mathbb{MLTT}$]
The 2-category $\mathbb{MLTT}$ is given by:
\begin{itemize}

\item 0-cells are MLTTs;

\item The hom-category $\mathbb{MLTT}(\mathcal{T}, \mathcal{T'})$ for each pair $\mathcal{T}, \mathcal{T'} \in \mathbb{MLTT}$ is $\mathcal{T}\text{-}\mathcal{ALG}(\mathit{Cl}(\mathcal{T'}))$;

\item The composition functor $\mathbb{MLTT}(\mathcal{T}, \mathcal{T'}) \times \mathbb{MLTT}(\mathcal{T'}, \mathcal{T''}) \to \mathbb{MLTT}(\mathcal{T}, \mathcal{T''})$ for each triple $\mathcal{T}, \mathcal{T'}, \mathcal{T''} \in \mathbb{MLTT}$ is given by the composition of CCFwDs and the horizontal composition of NTwDs; and

\item The functor $1 \to \mathbb{MLTT}(\mathcal{T}, \mathcal{T})$ for each $\mathcal{T} \in \mathbb{MLTT}$ maps the unique object $\star \in 1$ to the generic model $G$ of $\mathcal{T}$ and the unique morphism $\mathit{id}_\star$ to the identity NTwD $\mathit{id}_{U_G}$.

\end{itemize}
\end{definition}

Note that the 2-category $\mathbb{MLTT}$ is essentially the sub-2-category of $\mathbb{CCC_D}$ whose 0-cells are the classifying CCCwDs of MLTTs, and 1-cells are the CCFwDs $U_S$ induced by algebras $S$ for MLTTs.
Thus, it is easy to see that $\mathbb{MLTT}$ is a well-defined 2-category. 
Moreover, we have:

\begin{theorem}[Theory-category correspondence between MLTTs and CCCwDs]
The 2-categories $\mathbb{MLTT}$ and $\mathbb{CCC_D}$ are 2-equivalent. 
More precisely, the 2-functor $\mathit{Cl} : \mathbb{MLTT} \to \mathbb{CCC_D}$ is essentially surjective on objects, and the functor $\mathit{Cl}_{\mathcal{T}, \mathcal{T'}} : \mathbb{MLTT}(\mathcal{T}, \mathcal{T'}) \to \mathbb{CCC_D}(\mathit{Cl}(\mathcal{T}), \mathit{Cl}(\mathcal{T'}))$ constitues an equivalence $\mathbb{MLTT}(\mathcal{T}, \mathcal{T'}) \simeq \mathbb{CCC_D}(\mathit{Cl}(\mathcal{T}), \mathit{Cl}(\mathcal{T'}))$ for all $\mathcal{T}, \mathcal{T'} \in \mathbb{MLTT}$. 
\end{theorem}
\begin{proof}
By Corollary~\ref{CoroEqBetweenAlgebrasAndInterpretations}, it suffices to show that the 2-functor $\mathit{Cl}$ is essentially surjective on objects.
Then, given a CCCwD $\mathcal{C}$, the standard `theory-construction' $\mathit{Th}$ induces an MLTT $\mathit{Th}(\mathcal{C})$ such that $\mathit{Cl}(\mathit{Th}(\mathcal{C})) \cong \mathcal{C}$, which completes the proof.
\end{proof}

\section{Conclusion and Future Work}
\label{ConclusionAndFutureWork}
We have generalized CCCs, viz., CCCwDs, and proved that they give a categorical semantics of MLTTs in a true-to-syntax fashion. 
We have also analyzed the relation between the standard interpretation of STLCs in CCCs and ours in CCCwDs by establishing a 2-equivalence between CtxCCCs and constant CtxCCCwDs.
For these results, it would be fair to say that we have discovered the categorical counterpart of the path from STTs to DTTs.

As immediate future work, we would like to consider additional structures on CCCwDs to interpret other type constructions in MLTTs, e.g., identity types, well-founded tree types and universes. 
Moreover, it may be possible to extend the present framework as a semantics of \emph{homotopy type theory (HoTT)} \cite{hottbook}.
Finally, it would be interesting to develop a general theory of categories with dependence in its own right, generalizing major theorems in category theory.

\subparagraph*{Acknowledgements.}
The author was supported by Funai Overseas Scholarship. 
Also, he is grateful to Samson Abramsky for fruitful discussions.


\bibliographystyle{alpha}
\bibliography{GamesAndStrategies,TypeTheoriesAndProgrammingLanguages,CategoricalLogic,HoTT,LogicalCalculus}


\end{document}